\documentclass[dvips,ejs,preprint]{imsart}

\usepackage{amsthm,amsmath}
\usepackage{amssymb}
\usepackage{amsfonts}
\usepackage{epsfig}
\usepackage{graphicx}
\usepackage[colorlinks]{hyperref}
\usepackage[authoryear]{natbib}
\usepackage{hypernat}

\doi{10.1214/154957804100000000}
\pubyear{0000}
\volume{0}
\firstpage{0}
\lastpage{0}

\startlocaldefs

\newcommand{\meil}[1]{\mathbf{#1}} 
\newcommand{\latin}[1]{\emph{#1}}

\theoremstyle{plain}
\newtheorem{theorem}{Theorem}
\newtheorem{lemma}{Lemma}

\newtheorem{proposition}[lemma]{Proposition}
\theoremstyle{definition}

\newtheorem{proc}{Procedure}
\theoremstyle{remark}
\newtheorem{remark}{Remark}

\newcommand{\egaldef}{:=} 
\newcommand{\flens}{\mapsto} 
\newcommand{\flapp}{\mapsto} 
\newcommand{\telque}{\, \mbox{ s.t. } \,} 

\newcommand{\un}{\mathbf{1}}

\newcommand{\R}{\mathbb{R}} 
\newcommand{\N}{\mathbb{N}} 
\newcommand{\X}{\mathcal{X}}
\newcommand{\Y}{\mathcal{Y}}

\DeclareMathOperator{\card}{Card} 
\DeclareMathOperator{\diam}{diam} 

\newcommand{\mini}[2]{#1 \wedge #2}
\newcommand{\minipar}[2]{\mini{\left(#1\right)}{\left(#2\right)}}
\newcommand{\maxi}[2]{#1 \vee #2}
\newcommand{\maxipar}[2]{\maxi{\left(#1\right)}{\left(#2\right)}}

\newcommand{\nupl}{_{1 \ldots n}} 

\newcommand{\paren}[1]{\left( \left. #1 \right. \right)}
\newcommand{\parenb}[1]{\bigl( \left. #1 \right. \big)}
\newcommand{\parenB}[1]{\Bigl( \left. #1 \right. \Big)}
\newcommand{\croch}[1]{\left[ \left. #1 \right. \right]}
\newcommand{\crochB}[1]{\Bigl[\left. #1 \right.\Big]}
\newcommand{\crochb}[1]{\bigl[\left. #1 \right.\big]}
\newcommand{\crochbb}[1]{\biggl[\left. #1 \right.\bigg]}
\newcommand{\set}[1]{\left\{ \left. #1 \right. \right\}}
\newcommand{\setb}[1]{\bigl\{#1\big\}}
\newcommand{\absj}[1]{\left\lvert #1 \right\rvert} 
\providecommand{\norm}[1]{\left \lVert #1 \right\rVert}
\newcommand{\carre}[1]{\left(#1\right)^2}

\renewcommand{\P}{\mathbb{P}}
\newcommand{\Prob}{\mathbb{P}} 
\newcommand{\E}{\mathbb{E}} 
\DeclareMathOperator{\var}{var} 
\newcommand{\sachant}{\, \right| \left. \,} 
\newcommand{\loi}{\mathcal{D}} 
\DeclareMathOperator{\Leb}{Leb} 

\newcommand{\bayes}{s}
\newcommand{\perte}[1]{\ell \paren{\bayes , #1 }}
\newcommand{\ERM}{\widehat{s}}

\newcommand{\M}{\mathcal{M}}
\newcommand{\Mh}{\widehat{\mathcal{M}}} 
\newcommand{\mM}{m \in \M}
\newcommand{\mMh}{m \in \Mh}
\newcommand{\mh}{\widehat{m}}

\DeclareMathOperator{\pen}{pen}
\DeclareMathOperator{\crit}{crit}
\newcommand{\penid}{\pen_{\mathrm{id}}} 
\newcommand{\ph}{\widehat{p}} 
\newcommand{\penidglo}{\pen_{\mathrm{id,g}}} 
\newcommand{\penoptlin}{\pen_{\mathrm{opt,lin}}} 

\newcommand{\ERMb}[1][]{\widehat{s}^{W #1}} 
\newcommand{\Pnb}[1][]{{P_n^{W #1}}}
\newcommand{\Es}{{\E_W}}
\newcommand{\CWinf}{\ensuremath{C_{W}}}

\newcommand{\hyp}[1]{\ensuremath{\mathbf{(#1)}}} 
\newcommand{\hypAb}{\hyp{Ab}} 
\newcommand{\hypAn}{\hyp{An}} 
\newcommand{\hypAsig}{\hyp{A\sigma}} 
\newcommand{\hypAsigmax}{\hyp{A\sigmax}} 
\newcommand{\hypAsg}{\hyp{A_{\textrm{gauss}}}} 
\newcommand{\hypAp}{\hyp{Ap}} 
\newcommand{\hypAh}{\hyp{Ah}} 
\newcommand{\hypAlip}{\hyp{Al}} 
\newcommand{\hypAciblemax}{\hyp{A\bayes_{\max}}} 
\newcommand{\hypAd}{\hyp{Ad_{\ell}}} 
\newcommand{\hypArXl}{\hyp{Ar^{X}_{\ell}}} 
\newcommand{\hypArLu}{\hyp{Ar_{u}}} 
\newcommand{\hypArDu}{\hyp{Ar_{u}^{d}}} 
\newcommand{\hypAreg}{\hyp{Ar_{\ell,u}}} 
\newcommand{\hypPpoly}{\hyp{P1}} 
\newcommand{\hypPrich}{\hyp{P2}} 
\newcommand{\hypPech}{\hyp{P3}} 
\newcommand{\hypPconst}{\hyp{P4}} 
\newcommand{\hypPtout}{\hyp{P1-3}} 
\newcommand{\hypBg}{\hyp{Bg}} 
\newcommand{\hypUg}{\hyp{Ug}} 
\newcommand{\hypH}{\hyp{H}} 
\newcommand{\crXl}{\ensuremath{c_{\mathrm{r},\ell}^X}}

\newcommand{\crDu}{\ensuremath{c_{\mathrm{r},\mathrm{u}}^d}}
\newcommand{\crLl}{\ensuremath{c_{\mathrm{r},\ell}}}
\newcommand{\crLu}{\ensuremath{c_{\mathrm{r},\mathrm{u}}}}
\newcommand{\cbiasmaj}{C_{\mathrm{b}}^{+}} 
\newcommand{\cbiasmin}{C_{\mathrm{b}}^{-}} 
\newcommand{\aM}{\alpha_{\M}}
\newcommand{\cM}{c_{\M}}
\newcommand{\sigmin}{\sigma_{\min}} 
\newcommand{\sigmax}{\sigma_{\max}} 
\newcommand{\cgauss}{c_{\mathrm{gauss}}} 

\newcommand{\hypAml}{\hyp{A_{m,\ell}}} 
\newcommand{\hypAgeps}{\hyp{A_{g,\epsilon}}} 
\newcommand{\hypAdel}{\hyp{A\delta}} 
\newcommand{\hypQ}{\hyp{A_Q}} 
\newcommand{\cdelmglo}{c_{\Delta,m}^{g}}

\newcommand{\Qmp}{\ensuremath{Q_{m}^{(p)}}} 
\newcommand{\cQpm}{\ensuremath{c_Q^-}} 

\newcommand{\einv}[1]{e^+_{#1}}
\newcommand{\einvz}[1]{e^0_{#1}}

\newcommand{\delc}{\overline{\delta}} 
\newcommand{\punmin}{\widetilde{p_1}} 
\newcommand{\punzero}{\punmin^{(0)}} 

\newcommand{\El}{\E^{\Lambda_m}}
\newcommand{\Il}{I_{\lambda}}
\newcommand{\lamm}{\lambda \in \Lambda_m} 

\newcommand{\pl}{p_{\lambda}} 
\newcommand{\betl}{\beta_{\lambda}} 
\newcommand{\sigl}{\sigma_{\lambda}}
\newcommand{\sigld}{\sigma_{\lambda}^d} 
\newcommand{\sigla}{\sigma_{\lambda}^r} 

\newcommand{\phl}{\widehat{p}_{\lambda}} 
\newcommand{\bethl}{\widehat{\beta}_{\lambda}} 

\newcommand{\phlW}[1][]{\ensuremath{\widehat{p}^{W #1}_{\lambda}}} 
\newcommand{\bethlW}[1][]{\ensuremath{\widehat{\beta}^{W #1}_{\lambda}}} 
\newcommand{\Wl}{\widehat{W}_{\lambda}}

\newcommand{\mefr}{\ensuremath{M}} 
\newcommand{\mnefr}{\ensuremath{M_n}} 

\endlocaldefs

\begin{document}
\begin{frontmatter}

\title{Model selection by resampling penalization}
\runtitle{Resampling penalization}

\begin{aug}
\author{\fnms{Sylvain} \snm{Arlot}\thanksref{t1}\ead[label=e1]{sylvain.arlot@ens.fr}}

\affiliation{CNRS, Willow Project-Team}

\address{Sylvain Arlot\\
CNRS; Willow Project-Team \\
Laboratoire d'Informatique de l'Ecole Normale Superieure\\
(CNRS/ENS/INRIA UMR 8548)\\
45, rue d'Ulm, 75230 Paris, France \\
\printead{e1}}
\thankstext{t1}{The author was financed in part by Univ Paris-Sud (Laboratoire de Mathematiques, CNRS-UMR 8628).}
\end{aug}

\runauthor{S. Arlot}

\begin{abstract}
In this paper, a new family of resampling-based penalization procedures
for model selection is defined in a general framework. It generalizes
several methods, including Efron's bootstrap penalization and the
leave-one-out penalization recently proposed by Arlot (2008), to any
exchangeable weighted bootstrap resampling scheme.
In the heteroscedastic regression framework, assuming the models to
have a particular structure, these resampling penalties are proved to
satisfy a non-asymptotic oracle inequality with leading constant close
to 1. In particular, they are asympotically optimal. Resampling
penalties are used for defining an estimator adapting simultaneously to
the smoothness of the regression function and to the heteroscedasticity
of the noise.
This is remarkable because resampling penalties are general-purpose
devices, which have not been built specifically to handle
heteroscedastic data. Hence, resampling penalties naturally adapt to
heteroscedasticity.
A simulation study shows that resampling penalties improve on $V$-fold
cross-validation in terms of final prediction error, in particular when
the signal-to-noise ratio is not large.
\end{abstract}


\end{frontmatter}

\tableofcontents

\clearpage

%
\section{Introduction} \label{sec.intro}
%
%
In the last decades, model selection has received much interest. When the final goal is prediction, model selection can be seen more generally as the question of choosing between the outcomes of several prediction algorithms. With such a general formulation, a natural and classical answer is the following. First, estimate the prediction error for each model or algorithm; second, select the model minimizing this criterion.
Model selection procedures mainly differ on the way of estimating the prediction error.

The empirical risk, also known as the apparent error or the resubstitution error, is a natural estimator of the prediction error.
Nevertheless, minimizing the empirical risk can fail dramatically: the empirical risk is strongly biased for models involving a number of parameters growing with the sample size because the same data are used for building predictors and for comparing them.

In order to correct this drawback, {\em cross-validation} methods have been introduced \cite{All:1974,Sto:1974}, relying on a data-splitting idea for estimating the prediction error with much less bias. In particular, $V$-fold cross-validation (VFCV, \cite{Gei:1975}) is a popular procedure in practice because it is both general and computationally tractable.
A large number of papers exist about the properties of cross-validation methods, showing that they are efficient for a suitable choice of the way data are split (or $V$ for VFCV).
Asymptotic optimality results for leave-one-out cross-validation (that is the $V=n$ case) in regression have been proved for instance by Li \cite{KCLi:1987} and by Shao \cite{Sha:1997}. However, when $V$ is fixed, VFCV can be asymptotically suboptimal, as showed by Arlot \cite{Arl:2008a}. We refer to the latter paper for more references on cross-validation methods, including the small amount of available non-asymptotic results.


%
Another way to correct the empirical risk for its bias is {\em penalization}.
In short, penalization selects the model minimizing the sum of the empirical risk and of some measure of complexity\footnote{Note that ``complexity'' here and in the following refers to the implicit modelization of a model or an algorithm, such as the number of estimated parameters. ``Complexity'' does not refer at all to the {\em computational} complexity of algorithms, which will always be called ``computational complexity'' in the following.} of the model (called penalty); see FPE \cite{Aka:1969}, AIC \cite{Aka:1973}, Mallows' $C_p$ or $C_L$ \cite{Mal:1973}.
Model selection can target two different goals.
On the one hand, a procedure is {\em efficient} (or asymptotically optimal) when its quadratic risk is asymptotically equivalent to the risk of the oracle.
On the other hand, a procedure is {\em model consistent} when it selects the smallest true model asymptotically with probability one.
This paper deals with {\em efficient} procedures, without assuming the existence of a true model.
Therefore, the {\em ideal penalty} for prediction is the difference between the prediction error (the ``true risk'') and the empirical risk; penalties should be data-dependent estimates of the ideal
penalty.\looseness=-1

Many penalties or complexity measures have been proposed. Consider for instance regression and least-squares estimators on finite-dimensional vector spaces (the models). When the design is fixed and the noise-level constant equal to $\sigma$, Mallows' $C_p$ penalty \cite{Mal:1973} is equal to $2 n^{-1} \sigma^2 D$ for a model of dimension $D$ and it can be modified according to the number of models \cite{Bir_Mas:2002,Sau:2006}. Mallows' $C_p$-like penalties satisfy some optimality properties \cite{Shi:1981,KCLi:1987,Bar:2002,Bir_Mas:2006} but they can fail when the data are heteroscedastic \cite{Arl:2008:shape} because these penalties are linear functions of the dimension of the models.

In the binary supervised classification framework, several penalties have been proposed.
First, VC-dimension-based penalties have the drawback of being independent of the underlying measure, so that they are adapted to the worst case.
Second, global Rademacher complexities \cite{Kol:2001,Bar_Bou_Lug:2002} (generalized by Fromont with resampling ideas \cite{Fro:2004}) take into account the distribution of the data, but they are still too large to achieve fast rates of estimation when the margin condition \cite{Mam_Tsy:1999} holds.
Third, local Rademacher complexities \cite{Bar_Bou_Men:2005,Kol:2006} are tighter estimates of the ideal penalty, but their computational cost is heavy and they involve huge (and sometimes unknown) constants.
Therefore, easy-to-compute penalties that can achieve fast rates are still needed.


%
All the above penalties have serious drawbacks making them less often used in practice than cross-validation methods: AIC and Mallows' $C_p$ rely on strong assumptions (such as homoscedasticity of the data and linearity of the models) and some mainly asymptotic arguments; VC-dimension-based penalties and global Rademacher complexities are far too pessimistic; local Rademacher complexities are computationally intractable, and their calibration is a serious issue.
Another approach for designing penalties in the general framework may not suffer from these drawbacks: the {\em resampling} idea.

Efron's resampling heuristics \cite{Efr:1979} was first stated for the bootstrap, then generalized to the exchangeable weighted bootstrap by Mason and Newton \cite{Mas_New:1992} and by Pr\ae stgaard and Wellner \cite{Pra_Wel:1993}.
In short, according to the resampling heuristics, the distribution of any function of the (unknown) distribution of the data and the sample can be estimated by drawing ``resamples'' from the initial sample.
In particular, the resampling heuristics can be used to estimate the variance of an estimator \cite{Efr:1979}, a prediction error \cite{Wu:1986,Efr_Tib:1997} or the ideal penalty (using the bootstrap \cite{Efr:1983,Efr:1986,Ish_Sak_Kit:1997}, the $\mefr$ out of $n$ bootstrap\footnote{Shao's goal in \cite{Sha:1996} was not efficiency but model consistency.} \cite{Sha:1996} or a $V$-fold subsampling scheme \cite{Arl:2008a}).
The asymptotic optimality of Efron's bootstrap penalty for selecting among maximum likelihood estimators has been proved by Shibata \cite{Shi:1997}.
Note also that global and local Rademacher complexities are using an i.i.d. Rademacher resampling scheme for estimating different upper bounds on the ideal penalty and Fromont's penalties \cite{Fro:2007} generalize the global Rademacher complexities to the exchangeable weighted bootstrap.


%
The first goal of this paper is to define and study {\em general-purpose penalties}, that is penalties well-defined in almost every framework and performing reasonably well in most of them, including regression and classification. The main interest of such penalties would be the ability to solve difficult problems (for instance heteroscedastic data, a non-smooth regression function or the fact that the oracle model achieves fast rates of estimation) {\em without knowing them in advance}. From the practical point of view, such a property is crucial.

To this aim, the resampling heuristics with the general exchangeable weighted bootstrap is used for estimating the ideal penalty (Section~\ref{sec.cadreRP}).
This defines a wide family of model selection procedures, called ``Resampling Penalization'' (RP), which includes Efron's and Shao's penalization methods \cite{Efr:1983,Sha:1996} as well as the leave-one-out penalization defined in \cite{Arl:2008a}. To our knowledge, it has never been proposed with such general resampling schemes, so that the RP family contains a wide range of new procedures.
Note that RP is well-defined in a general framework, including regression and classification, but also many other application fields (Section~\ref{RP.sec.discussion.classif}). Even if the main results are proved in the least-squares regression framework only, we obviously do not mean that RP should be restricted to this framework.

In this paper, the model selection efficiency of RP is studied with a {\em unified approach} for all the exchangeable resampling schemes.
Therefore, comparing bootstrap with subsampling is quite straightforward (Section~\ref{RP.sec.simus}) which is not common in the resampling literature (except a few asymptotic results, see Barbe and Bertail \cite{Bar_Ber:1995}).


The point of view used in the paper is {\em non-asymptotic}, which has two major implications.
First, non-asymptotic results allow to consider collections of models depending on the sample size $n$: in practice, it is usual to increase the number of explanatory variables with the number of observations.
Considering models with a large number of parameters (for instance of order $n^{\alpha}$ for some $\alpha >0$) is also particularly useful for designing adaptive estimators of a function which is only assumed to belong to some H\"olderian ball (see Section~\ref{RP.sec.main.holder}).
Thus, the non-asymptotic point of view allows not to assume that the regression function is described with a small number of parameters.

Second, several practical problems are ``non-asymptotic'' in the sense that the signal-to-noise ratio is low. As noticed in \cite{Arl:2008a}, with such data, VFCV can have serious drawbacks which can be naturally fixed by using the flexibility of penalization procedures.
It is worth noting that such a non-asymptotic approach is not common in the model selection literature and few non-asymptotic results exist on general resampling methods.

Another important point is that the framework of the paper includes several kinds of {\em heteroscedastic data}. The observations $(X_i,Y_i)_{1 \leq i \leq n}$ are only assumed to be i.i.d. with
\[ Y_i = \bayes(X_i) + \sigma(X_i) \epsilon_i , \]
where $\bayes: \X \flens \R$ is the (unknown) regression function, $\sigma: \X \flens \R$ is the (unknown) noise-level and $\epsilon_i$ has zero mean and unit variance conditionally on $X_i$. In particular, the noise-level $\sigma(X)$ can strongly depend on $X$ and the distribution of $\epsilon_i$ can depend on $X_i$.
Such data are generally considered as difficult to handle because no information on $\sigma$ is known, making irregularities of the signal difficult to distinguish from noise. As already mentioned, simple model selection procedures such as Mallows' $C_p$ can fail in this framework \cite{Arl:2008:shape} whereas it is natural to expect that resampling methods are robust to heteroscedasticity. In this article, both theoretical and simulation results confirm this fact (Sections~\ref{RP.sec.main} and~\ref{RP.sec.simus}).


The two main results of the paper are stated in Section~\ref{RP.sec.main}.
First, making mild assumptions on the distribution of the data, a non-asymptotic oracle inequality for RP is proved with leading constant close to~1 (Theorem~\ref{RP.the.oracle_traj_non-as}). It holds for several kinds of resampling schemes (including bootstrap, leave-one-out, half-subsampling and i.i.d. Rademacher weighted bootstrap) and implies the asymptotic optimality of RP, even when the data are highly heteroscedastic.
For proving such a result, each model is assumed to be the vector space of piecewise constant functions (histograms) on some partition of the feature space.
This is indeed a restriction, but we conjecture that it is mainly technical and that RP remains efficient in a much more general framework (see Section~\ref{RP.sec.discussion.classif}).
Moreover, studying extensively the toy model of histograms allows to derive precise heuristics for the general framework. A major goal of the paper is to help practicioners who would like to know how to use resampling for performing model selection (see in particular Sections~\ref{RP.sec.pratique} and~\ref{RP.sec.discussion.conclu}).

Second, RP is used to build an estimator simultaneously adaptive to the smoothness of the regression function (assuming that $\bayes$ is $\alpha$-H\"olderian for some unknown $\alpha \in (0,1]$) and to the unknown noise-level $\sigma\paren{\cdot}$ (Theorem~\ref{RP.the.holder}). This result may seem surprising since RP has never been designed specifically for such a purpose. We interpretate Theorem~\ref{RP.the.holder} as a confirmation that RP is {\em naturally adaptive} and should work well in several other difficult frameworks.

Several results similar to Theorem~\ref{RP.the.oracle_traj_non-as} exist in the literature for other procedures such as Mallows' $C_p$ (with homoscedastic data only), VFCV and leave-one-out cross-validation. Moreover, there exist several minimax adaptive estimators for heteroscedastic data with a smooth noise-level, for instance \cite{Efr_Pin:1996,Gal_Per:2008}, and the regression function and the noise level can be estimated simultaneously \cite{Gen:2008}.
In comparison, the interest of RP is both its {\em generality} (contrary to Mallows' $C_p$ and specific adaptive estimators) and its {\em flexibility} (contrary to VFCV, see \cite{Arl:2008a}), as detailed in Section~\ref{RP.sec.comparaison}.

A simulation study is conducted in Section~\ref{RP.sec.simus} with small sample sizes. RP is showed to be competitive with Mallows' $C_p$ for ``easy'' problems, and much better for some harder ones (for instance with a variable noise-level). Moreover, a well-calibrated RP yields almost always better model selection performance than VFCV.
Therefore, RP can be of great interest in situations where no \latin{a priori} information is known about the data. RP can deal with difficult problems, and compete with procedures that are fitted for easier problems.
In short, RP is an efficient alternative to VFCV.


This article is organized as follows.
The framework and the Resampling Penalization (RP) family of procedures are defined in Section~\ref{sec.cadreRP}.
The main results are stated in Section~\ref{RP.sec.main}.
The differences between the resampling weights are investigated in Section~\ref{RP.sec.calc.delta}.
Then, a simulation study is presented in Section~\ref{RP.sec.simus}.
Practical issues concerning the implementation of RP are considered in Section~\ref{RP.sec.pratique}.
RP is compared to other penalization methods in Section~\ref{RP.sec.comparaison} and the extension of RP to the general framework is discussed in Section~\ref{RP.sec.discussion.classif}.
Finally, Section~\ref{RP.sec.preuves} is devoted to the proofs. Some additional material (other simulation experiments and proofs) is available in a technical Appendix \cite{Arl:2008b:app}.

\section{The Resampling Penalization procedure} \label{sec.cadreRP}
In order to simplify the presentation, we choose to focus on the particular framework of least-squares regression on models of piecewise constant functions (histograms), which is the framework of the main results of Section~\ref{RP.sec.main} and the simulation study of Section~\ref{RP.sec.simus}.

Nevertheless, the RP family is a {\em general-purpose method} which can easily be defined in the general prediction framework. The main interest of the histogram framework is to provide general heuristics about RP, so that the practicioner can make the best possible use of RP in the general framework. A discussion on RP in the general prediction framework is provided in Section~\ref{RP.sec.discussion.classif}, including a general definition of RP.

\subsection{Framework} \label{sec.cadreRP.cadre}
%
%
Suppose we observe some data $(X_1,Y_1), \ldots, (X_n,Y_n) \in \X \times \R$, independent with common distribution $P$, where the feature space $\X$ is typically a compact set of $\R^k$.
Let $\bayes$ denote the regression function, that is $\bayes(x) =\E\croch{Y \sachant X = x}$. Then,
\begin{equation} \label{RP.eq.donnees.reg} Y_i = \bayes(X_i) + \sigma(X_i) \epsilon_i
\end{equation}
where $\sigma: \X \mapsto \R$ is the heteroscedastic noise-level and $\epsilon_i$ are i.i.d. centered noise terms; the $\epsilon_i$ possibly depend on $X_i$, but they are have zero mean and unit variance conditionally on $X_i$.

The goal is to predict $Y$ given $X$ where $(X,Y) \sim P$ is independent of the data.
The quality of a predictor $t: \X \mapsto \R$ is measured by the quadratic prediction loss
$ P \gamma(t) \egaldef \E_{(X,Y)} \croch{ \gamma(t,(X,Y)) }$, where $(X,Y) \sim P$ and $\gamma(t,(x,y)) \egaldef \paren{ t(x) - y }^2$ is the least-squares contrast.
Since $P \gamma(t)$ is minimal when $t = \bayes$, the excess loss is defined as
\[ \perte{t} \egaldef P\gamma\paren{t} - P\gamma\paren{\bayes} = \E_{(X,Y)} \carre{t(X)-\bayes(X)} . \]
Given a particular set of predictors $S_m$ (called a {\em model}), the best predictor over $S_m$ is defined as \[ \bayes_m \egaldef \arg\min_{t \in S_m} \set{ P \gamma(t) } , \]
with its empirical counterpart
\[ \ERM_m \egaldef \arg\min_{t \in S_m} \set{ P_n \gamma(t) } \] (when it exists and is unique) where $P_n = n^{-1} \sum_{i=1}^n \delta_{(X_i,Y_i)}$ is the empirical distribution.
The estimator $\ERM_m$ is the well-known
{\em empirical risk minimizer}, also called least-squares estimator since $\gamma$ is the least-squares contrast.



In this article, we mainly consider histogram models $S_m$, that is of the following form.
Let $\paren{\Il}_{\lamm}$ be some fixed partition of $\X$. Then, $S_m$ denotes the set of functions $\X \flens \R$ which are constant over $\Il$ for every $\lamm$; $S_m$ is a vector space of dimension $D_m = \card(\Lambda_m)$, spanned by the family $(\un_{\Il})_{\lamm}$.
The empirical risk minimizer $\ERM_m$ over an histogram model $S_m$ is often called a {\em regressogram}.

Explicit computations are easier with regressograms because $(\un_{\Il})_{\lamm}$ is an orthogonal basis of $L^2(\mu)$ for any probability measure $\mu$ on $\X$.
In particular,
\begin{gather*}
\bayes_m = \sum_{\lamm} \betl \un_{\Il} \quad \mbox{and} \quad \ERM_m = \sum_{\lamm} \bethl \un_{\Il} ,\\
\mbox{where } \enspace
\betl \egaldef \E_P \croch{ Y \sachant X \in \Il } , \enspace\enspace \bethl \egaldef \frac{1}{n\phl} \sum_{X_i \in \Il} Y_i
\enspace\mbox{ and }\enspace \phl \egaldef P_n(X \in \Il) .
\end{gather*}
Note that $\ERM_m$ is uniquely defined if and only if each $\Il$ contains at least one of the $X_i$, that is $\min_{\lamm} \phl >0$.



Let us assume that a collection of models $(S_m)_{\mM_n}$ is given.
Model selection consists in selecting some data-dependent $\mh \in \M_n$ such that $\perte{\ERM_{\mh}}$ is as small as possible.
General penalization procedures can be described as follows.
Let $\pen: \M_n \flapp \R^+$ be some penalty function, possibly data-dependent, and define
\begin{equation} \label{def.mh.pen}
\mh \in \arg\min_{\mM_n} \set{ P_n \gamma\paren{\ERM_m} + \pen(m) } .
\end{equation}
Since the goal is to minimize the loss $P \gamma\paren{\ERM_m}$, the {\em ideal penalty} is
\begin{equation} \label{RP.def.penid}
 \penid(m) \egaldef (P-P_n) \gamma\paren{\ERM_m} , \end{equation}
and we would like $\pen(m)$ to be as close to $\penid(m)$ as possible for every $\mM_n$.
In the histogram framework, note that $\ERM_m$ is not uniquely defined when $\min_{\lamm} \phl = 0$;
then, we consider that the model $S_m$ cannot be chosen, which is formally equivalent to add $+\infty \un_{\min_{\lamm} \phl = 0}$ to the penalty $\pen(m)$.

When $S_m$ is the histogram model associated with some partition $\paren{\Il}_{\lamm}$ of $\X$, the ideal penalty \eqref{RP.def.penid} can be computed explicitly:
\begin{align} \notag
\penid(m) &= (P-P_n) \gamma(\bayes_m)
+ P \paren{ \gamma\paren{\ERM_m} - \gamma\paren{\bayes_m}} + P_n \paren{ \gamma\paren{\bayes_m} - \gamma\paren{\ERM_m}} \\ \label{eq.penid.histos}
&= (P-P_n) \gamma(\bayes_m) + \sum_{\lamm}
\croch{ \pl \parenb{ \bethl - \betl }^2 + \phl \parenb{ \bethl - \betl }^2}
\end{align}
where $\pl \egaldef \Prob\paren{ X \in \Il}$.
The ideal penalty $\penid(m)$ is unknown because it depends on the true distribution $P$; therefore, resampling is a natural method for estimating $\penid(m)$.

\subsection{The resampling heuristics} \label{sec.cadreRP.heur} 

Let us recall briefly the {\em resampling heuristics}, which has been introduced by Efron \cite{Efr:1979} in the context of variance estimation.
Basically, it says that one can mimic the relationship between $P$ and $P_n$ by drawing a $n$-sample with common distribution $P_n$, called the ``resample''; let $\Pnb$ denote the empirical distribution of the resample.
Then, the conditional distribution of the pair $(P_n,\Pnb)$ given $P_n$ should be close to the distribution of the pair $(P,P_n)$.
Hence, the expectation of any quantity of the form $F(P,P_n)$ can be estimated by $\Es \croch{ F(P_n, \Pnb) }$. The expectation $\Es\croch{\cdot}$ means that we integrate with respect to the resampling randomness only. Let us emphasize that $\penid(m)$ has the form $F(P,P_n)$.

Later on, this heuristics has been generalized to other resampling schemes, with the exchangeable weighted bootstrap \cite{Mas_New:1992,Pra_Wel:1993}. The empirical distribution of the resample then has the general form
\[ \Pnb \egaldef \frac{1}{n} \sum_{i=1}^n W_i \delta_{(X_i,Y_i)} \] where $W \in \R^n$ is an exchangeable\footnote{$W$ is said to be {\em exchangeable} when its distribution is invariant by any permutation of its coordinates.} weight vector independent of the data and such that $\forall i, \, \E_W\croch{W_i}=1$.
In this article, $W$ is also assumed to satisfy $\forall i$, $W_i \geq 0$ a.s. and $\E_W[W_i^2]<\infty$.


We mainly consider the following weights, which include the more classical resampling schemes:
\begin{enumerate}
\item {\em Efron} ($\mefr$), $\mefr \in \N \backslash \{0 \}$ (Efr): $((\mefr /n) W_i)_{1 \leq i \leq n}$ is a multinomial vector with parameters $(\mefr ;n^{-1},\ldots, n^{-1})$. A classical choice is $\mefr=n$.
\item {\em Rademacher} ($p$), $p \in (0;1)$ (Rad): $(p W_i)$ are independent, with a Bernoulli ($p$) distribution. A classical choice is $p=1/2$.
\item {\em Poisson} ($\mu$), $\mu \in (0,\infty)$ (Poi): $(\mu W_i)$ are independent, with a Poisson ($\mu$) distribution. A classical choice is $\mu=1$.
\item {\em Random hold-out} ($q$), $q \in \{1, \ldots, n\}$ (Rho): $W_i = (n/q) \un_{i \in I}$ where $I$ is a uniform random subset of cardinality $q$ of $\{1, \ldots, n\}$. A classical choice is $q=n/2$.
\item {\em Leave-one-out} (Loo) = Rho ($n-1$).
\end{enumerate}
In the following, Efr, Rad, Poi, Rho and Loo respectively denote the above resampling weight vector distributions with the ``classical'' value of the parameter.

\begin{remark}
The above terminology explicitly links the weight vector distributions with some classical resampling schemes. See \cite{Mas_New:1992,Hal_Mam:1994,vdV_Wel:1996} for more details about classical resampling weight names, as well as other classical examples.
\begin{itemize}
\item The name ``Efron'' comes from the classical choice $\mefr=n$ for which Efron weights actually are the bootstrap weights. When $\mefr < n$, Efron($\mefr$) is the $\mefr$ out of $n$ bootstrap, used for instance by Shao \cite{Sha:1996}.
\item The name ``Rademacher'' for the i.i.d. Bernoulli weights comes from the classical choice $p=1/2$ for which $(W_i-1)_i$ are i.i.d. Rademacher random variables.
For instance, global and local Rademacher complexities use this resampling scheme to estimate different upper bounds on $\penid(m)$ (see Section~\ref{RP.sec.discussion.classif.classif}).
\item Poisson weights are often used as approximations to Efron weights, via the so-called ``Poissonization'' technique (see \cite[Chapter~3.5]{vdV_Wel:1996} and \cite{Fro:2004}). They are known to be efficient for estimating several non-smooth functionals (see \cite[Chapter~3]{Bar_Ber:1995} and \cite[Section~1.4]{Mam:1992}).
\item The Random hold-out ($q$) weights can also be called ``delete-$(n-q)$ jackknife'', as well as the Leave-one-out weights also refer to the jackknife (sometimes called cross-validation). They are both resampling schemes without replacement \cite[Example~3.6.14]{vdV_Wel:1996}, more often called {\em subsampling weights} (see for instance the book by Politis, Romano and Wolf \cite{Pol_Rom_Wol:1999} on subsampling). They are close to the idea of splitting the data into a training set and a validation set (for instance, leave-one-out, hold-out and cross-validation). Indeed, if one defines the training set as \[ \set{ (X_i,Y_i) \telque W_i \neq 0 } \] and the validation set as its complement, there is a one-to-one correspondence between subsampling weights and data splitting.
\end{itemize}
\end{remark}

\subsection{Resampling Penalization} \label{sec.cadreRP.RP}

Applying directly the resampling heuristics of Section~\ref{sec.cadreRP.heur} for estimating the ideal penalty \eqref{RP.def.penid}, we would get the penalty
\begin{gather} \label{RP.def.penreech}
\Es \croch{ P_n \gamma \paren{ \ERMb_m } - \Pnb \gamma \paren{\ERMb_m } } , \\
\notag \mbox{where} \quad \ERMb_m \egaldef \arg\min_{t \in S_m} \Pnb \gamma(t) = \sum_{\lamm} \bethlW \un_{\Il} , \quad \bethlW \egaldef \frac{1}{n\phlW} \sum_{X_i \in \Il} W_i Y_i ,\\
\notag \phlW \egaldef \Pnb(X \in \Il) = \phl \Wl \quad \mbox{and} \quad \Wl \egaldef \frac{1}{n \phl} \sum_{X_i \in \Il} W_i .
\end{gather}
Two problems have to be solved before defining properly the Resampling Penalization procedure. Here, we focus on the histogram framework; the general framework will be considered in Section~\ref{RP.sec.discussion.classif}.

First, \eqref{RP.def.penreech} is not well-defined because $\ERMb_m$ is not unique if $\min_{\lamm} \phlW = 0$.
Hence, even when $\min_{\lamm} \phl >0$, the problem occurs as soon as $\Wl=0$ for some $\lamm$, which has a positive probability (except when $D_m=1$) for most of the resampling schemes since $\Prob_W \paren{ \forall i \geq 2, \, W_i = 0 }>0$.
In order to make \eqref{RP.def.penreech} well-defined, let us rewrite the resampling penalty as the resampling estimate of \eqref{eq.penid.histos}, that is \[ \Es \croch{ P_n \gamma \paren{ \ERMb_m } - \Pnb \gamma \paren{\ERMb_m } } = \ph_0(m) + \ph_1(m) + \ph_2(m) \] where
\[ \ph_0(m) \egaldef \Es\croch{ (P_n-\Pnb) \gamma(\ERM_m)} = \frac{1}{n} \sum_{i=1}^n \paren{\Es\croch{1-W_i} \gamma\paren{\ERM_m ; (X_i,Y_i)}} = 0 \]
because $\Es[W_i]=1$ for every $i$,
\begin{align*}
\ph_1(m) &\egaldef \sum_{\lamm} \paren{ \Es\paren{ \phl \paren{ \bethlW - \bethl }^2}} \\
\mbox{and} \quad \ph_2(m) &\egaldef \sum_{\lamm} \paren{ \Es\croch{\phlW \paren{ \bethlW - \bethl }^2}} .
\end{align*}
With the convention $\phlW ( \bethlW - \bethl )^2 =0$ when $\phlW=0$, $\ph_2(m)$ is well-defined since $\bethlW$ is well-defined when $\phlW>0$.
It remains to define properly $\ph_1(m)$.
We suggest to replace the expectation over all the resampling weights by an expectation conditional on $\Wl>0$, {\em separately for each $\mM_n$ and $\lamm$}, which ensures that we only remove a small proportion of the possible resampling weights.
To summarize, \eqref{RP.def.penreech} is replaced by
\begin{equation} \label{RP.def.pen.his.tmp}
\sum_{\lamm} \paren{ \Es \croch{ \phl \paren{\bethlW - \bethl}^2 \sachant \Wl >0 } + \Es \croch{ \phlW \paren{\bethlW - \bethl}^2 }} .
\end{equation}

Second, \eqref{RP.def.pen.his.tmp} is strongly biased as an estimate of $\penid$ when $\var(W_1)$ is small, because $\Pnb$ is then much closer to $P_n$ than $P_n$ is close to $P$.
Assuming the $S_m$ to be histogram models, we will prove in Section~\ref{RP.sec.histos.exp} (see Propositions~\ref{RP.pro.EpenRP-Epenid} and~\ref{RP.pro.comp.Epen.Ep2}) that the bias can be corrected by multiplying \eqref{RP.def.pen.his.tmp} by a constant \CWinf\ which only depends on the distribution of $W$. The values of \CWinf\ for the classical weights are reported in Table~\ref{RP.TableCWinfty}.
Remark that $\CWinf=1$ in the bootstrap case (Efr), as well as for Rad, Poi and Rho.
%
\begin{table}
\caption{\CWinf\ for several resampling schemes (see Section~\ref{RP.sec.histos.exp})\label{RP.TableCWinfty}}
   \begin{tabular}{|r|c|c|c|c|c|} \hline
   $\loi(W)$ & Efr($\mefr$) & Rad($p$) & Poi($\mu$) & Rho($q$) & Loo \\ \hline
   \CWinf & $\mefr /n$ & $p/(1-p)$ & $\mu$ & $q/(n-q)$ & $n-1$ \\ \hline
   \end{tabular}
\end{table}


We are now in position to define properly the Resampling Penalization (RP) procedure for selecting among histogram models. See Section~\ref{RP.sec.discussion.classif} for the definition of RP in the general framework (Procedure~\ref{RP.def.proc.gal}).

\begin{proc}[Resampling Penalization for histograms] \label{RP.def.proc.his} \hfill
\begin{enumerate}
\item Replace $\M_n$ by
\begin{equation*}
\Mh_n = \Big\{ \mM_n \telque \min_{\lamm} \{ n\phl \} \geq 3 \Big\} .
\end{equation*}
\item Choose a resampling scheme $\loi(W)$.
\item Choose a constant $C \geq \CWinf$ where $\CWinf$ is defined in Table~\ref{RP.TableCWinfty}.
\item Define, for each $\mMh_n$, the resampling penalty $\pen(m)$ as
{ 
\begin{equation} \label{RP.def.pen.his}
C \sum_{\lamm} \paren{ \Es \croch{ \phl \paren{\bethlW - \bethl}^2 \sachant \Wl >0 } + \Es \croch{ \phlW \paren{\bethlW - \bethl}^2 }} .
\end{equation} }
\item Select $\mh \in \arg\min_{\mMh_n} \set{P_n\gamma\paren{\ERM_m} + \pen(m)}$.
\end{enumerate}
\end{proc}

\begin{remark}
\begin{enumerate}
  \item At step 1, we remove more models than those for which $\ERM_m$ is not uniquely defined.
When $n \phl = 1$ for some $\lamm$, estimating the quality of estimation of $\bethl$ with only one data-point is hopeless with no assumption on the noise-level $\sigma$.
The reason why we remove also models for which $\min_{\lamm} \set{n \phl} =2$ is that the oracle inequalities of Section~\ref{RP.sec.main} require it for some of the weights; nevertheless, such models generally have a poor prediction performance, so that step 1 is reasonable.
\item At step 3, $C$ can be larger than $\CWinf$ because overpenalizing can be fruitful from the non-asymptotic point of view, in particular when the sample size $n$ is small or the noise level $\sigma$ is large. The simulation study of Section~\ref{RP.sec.simus} provides experimental evidence for this fact (see also Section~\ref{RP.sec.pratique.const.overpen}).
\item RP (Procedure~\ref{RP.def.proc.his}) generalizes several model selection procedures.
With a bootstrap resampling scheme (Efr) and $C=1$, RP is Efron's bootstrap penalization \cite{Efr:1983}, which has also been called EIC in the log-likelihood framework \cite{Ish_Sak_Kit:1997}.
With an $\mefr$ out of $n$ bootstrap resampling scheme (Efr($\mefr$)) and $C=1$, RP has been proposed and studied by Shao \cite{Sha:1996} in the context of model identification. Note that $\CWinf \neq 1$ for Efr($\mefr$) weights if $\mefr \neq n$; this crucial point will be discussed in Section~\ref{RP.sec.histos.exp}.
RP with a (non-exchangeable) $V$-fold subsampling scheme has also been proposed recently in \cite{Arl:2008a}.
\item When $W$ are the ``leave-one-out'' weights, RP is not the classical leave-one-out model selection procedure. Nevertheless, according to \cite{Arl:2008a}, when $C=n-1$, it is identical to Burman's $n$-fold corrected cross-validation \cite{Bur:1989}, hence close to the uncorrected one.
\end{enumerate}
\end{remark}

\section{Main results} \label{RP.sec.main}
In this section, we state some non-asymptotic properties of Resampling Penalization (Procedure~\ref{RP.def.proc.his}) for model selection.
First, Theorem~\ref{RP.the.oracle_traj_non-as} is an oracle inequality with leading constant close to 1.
In particular, Theorem~\ref{RP.the.oracle_traj_non-as} implies the asymptotic optimality of RP.
Second, Theorem~\ref{RP.the.holder} is an adaptivity result for an estimator built upon RP, when the regression function belongs to some H\"olderian ball.
A remarkable point is that both results remain valid under mild assumptions on the distribution of the noise, which can be non-Gaussian and highly heteroscedastic.

Throughout this section, we assume the existence of non-negative constants $\aM$, $\cM$, $c_{\mathrm{rich}}$ such that:
\begin{enumerate}
\item[\hypPpoly] Polynomial size of $\M_n$: $\card(\M_n) \leq \cM n^{\aM}$.
\item[\hypPrich] Richness of $\M_n$: $\exists m_0 \in \M_n$ s.t. $D_{m_0} \in \croch{ \sqrt{n}; c_{\mathrm{rich}} \sqrt{n} }$.
\item[\hypPech] The weight vector $W$ is chosen among Efr, Rad, Poi, Rho and Loo (defined in Section~\ref{sec.cadreRP.heur}, with the classical value of their parameter).
\end{enumerate}
\hypPpoly\ is a natural restriction since RP plugs an estimator of the ideal penalty into \eqref{def.mh.pen}.
When $\card(\M_n)$ is larger, say proportional to $e^{a n}$ for some $a >0$, Birg\'e and Massart \cite{Bir_Mas:2006} proved that penalties estimating the ideal penalty cannot be asymptotically optimal.
\hypPrich\ is merely technical.
\hypPech\ can be relaxed, as explained in Section~\ref{sec.Thm1.autresW}.

\subsection{Oracle inequality} \label{RP.sec.main.oracle}

\makeatletter
\setlength\leftmargini   {33\p@}
\makeatother

\begin{theorem} \label{RP.the.oracle_traj_non-as}
Assume that the data $(X_i,Y_i)_{1 \leq i \leq n}$ satisfy the following:
\begin{enumerate}
\item[\hypAb] Bounded data: $\norm{Y_i}_{\infty} \leq A < \infty$.
\item[\hypAn] Noise-level bounded from below: $\sigma(X_i) \geq \sigmin>0$ a.s.
\item[\hypAp] Polynomially decreasing bias: there exist $\beta_1\geq\beta_2>0$ and $\cbiasmaj,\cbiasmin>0$ such that \[ \forall\mM_n, \quad \cbiasmin D_m^{-\beta_1} \leq \perte{\bayes_m} \leq \cbiasmaj D_m^{-\beta_2} . \]
\item[\hypArXl] Lower regularity of the partitions for $\loi(X)$: there exists $\crXl>0$ such that \[ \forall\mM_n, \quad D_m \min_{\lamm} \pl \geq \crXl . \]
\end{enumerate}

Let $\mh$ be defined by Procedure~\ref{RP.def.proc.his} \textup{(}under restrictions \hypPtout, with $C=\CWinf$\textup{)}. Then, there exist a constant $K_1>0$ and an absolute sequence $\varepsilon_n$ converging to zero at infinity such that, with probability at least $1 - K_{1} n^{-2}$,
\begin{equation} \label{RP.eq.oracle_traj_non-as}
\perte{\ERM_{\mh}} \leq \paren{ 1 + \varepsilon_n } \inf_{\mM_n} \set{ \perte{\ERM_m} } .\end{equation}

Moreover,
\begin{equation} \label{RP.eq.oracle_class_non-as}
\E \croch{\perte{\ERM_{\mh}}} \leq \paren{ 1 + \varepsilon_n } \E \crochB{ \inf_{\mM_n} \set{ \perte{\ERM_m} } } + \frac{A^2 K_{1}}{n^2} .
\end{equation}

The constant $K_1$ may depend on constants in \hypAb, \hypAn, \hypAp, \hypArXl\ and \hypPtout\ but not on $n$. The term $\varepsilon_n$ is smaller than $\paren{\ln(n)}^{-1/5}$; $\varepsilon_n$ can also be made smaller than $n^{-\delta}$ for any $0 < \delta < \delta_0 (\beta_1, \beta_2)$ at the price of enlarging $K_1$.
\end{theorem}

Theorem~\ref{RP.the.oracle_traj_non-as} is proved in Section~\ref{sec.proof.the.oracle_traj_non-as}.
The non-asymptotic oracle inequality \eqref{RP.eq.oracle_traj_non-as} implies that Procedure~\ref{RP.def.proc.his} is {\em a.s. asymptotically optimal} in this framework if $\lim_{n \rightarrow \infty} (C / \CWinf) = 1 $.
When $W$ are Efr weights, the asymptotic optimality of RP was proved by Shibata \cite{Shi:1997} for selecting among maximum likelihood estimators, assuming that the distribution $P$ belongs to some parametric family of densities (see also Remark~\ref{RP.rem.R1W.R2W} in Section~\ref{RP.sec.histos.exp}).

Resampling Penalization yields an estimator with an excess loss as small as the one of the oracle without requiring any knowledge about $P$ such as the smoothness of $\bayes$ or the variations of the noise-level $\sigma$. Therefore, RP is a {\em naturally adaptive procedure}.
Note that \eqref{RP.eq.oracle_traj_non-as} is even stronger than an adaptivity result because of the leading constant close to one, whereas adaptive estimators only achieve the correct estimation rate up to a possibly large absolute constant.
Hence, one can expect that an estimator obtained with RP and a well chosen collection of models is almost optimal.

We now comment on the assumptions of Theorem~\ref{RP.the.oracle_traj_non-as}:
\makeatletter
\setlength\leftmargini   {22\p@}
\makeatother
\begin{enumerate}
\item The constant $C$ can differ from $\CWinf$. For instance, when a constant $\eta>1$ exists such that $C \in \croch{\CWinf ; \eta \CWinf}$, the oracle inequalities \eqref{RP.eq.oracle_traj_non-as} and \eqref{RP.eq.oracle_class_non-as} hold with leading constant $2\eta - 1 + \varepsilon_n$ instead of $1 + \varepsilon_n$.
\item \hypAb\ and \hypAn\ are rather mild and neither $A$ nor $\sigmin$ need to be known by the statistician. In particular, quite general heteroscedastic noises are allowed; \hypAb\ and \hypAn\ can even be relaxed as explained in Section~\ref{RP.sec.main.hyp.alt}.
\item When $X$ has a lower bounded density with respect to $\Leb$, \hypArXl\ is satisfied for ``almost piecewise regular'' histograms, including all those considered in the simulation study of Section~\ref{RP.sec.simus}.
\item The upper bound in \hypAp\ holds with $\beta=2 \alpha k^{-1}$ when $(\Il)_{\lamm}$ is regular on $\X \subset \R^k$ and $\bayes$ is $\alpha$-H\"olderian with $\alpha >0$. The lower bound in \hypAp\ is discussed extensively in Section~\ref{RP.sec.main.hyp.Ap}.
\end{enumerate}

\subsection{An adaptive estimator} \label{RP.sec.main.holder}
A natural framework in which Theorem~\ref{RP.the.oracle_traj_non-as} can be applied is when $\X$ is a compact subset of $\R^k$, $X$ has a lower bounded density with respect to the Lebesgue measure and $\bayes$ is $\alpha$-H\"olderian with $\alpha \in (0,1]$.
Indeed, the latter condition ensures that regular histograms can approximate $\bayes$ well.
In this subsection, we show that Resampling Penalization can be used to build an estimator adaptive to the smoothness of $\bayes$ in this framework.

We first define the estimator. For the sake of simplicity\footnote{If $\X$ has a smooth boundary, Procedure~\ref{RP.def.proc.his.holder} can be modified so that the proof of Theorem~\ref{RP.the.holder} remains valid.}, $\X$ is assumed to be a closed ball of $(\R^k, \norm{\cdot}_{\infty})$, say $[0,1]^k$.
\begin{proc}[Resampling Penalization with regular histograms] \label{RP.def.proc.his.holder} 
For every $T \in \N\backslash\set{0}$, let $S_{m(T)}$ be the model of regular\footnote{When $\X$ has a general shape, assume that both $\Leb(\X)$ and $\diam(\X)$ for $\norm{\cdot}_{\infty}$ are finite. Then, a partition $\paren{\Il}_{\lamm}$ of $\X$ is {\em regular with $T^k$ bins} when $\card(\Lambda_m) = T^k$ and there exist positive constants $c_1$, $c_2$, $c_3$, $c_4$ such that for every $\lamm$, $c_1 T^{-k} \leq \Leb(\Il) \leq c_2 T^{-k}$ and $c_3 T^{-1} \leq \diam(\Il) \leq c_4 T^{-1}$.} histograms with $T^k$ bins, that is the histogram model associated with the partition
\[ \paren{\Il}_{\lambda\in\Lambda_{m(T)}} \egaldef \paren{ \prod_{i=1}^k \left[ \frac{j_i}{T} ; \frac{j_i +1}{T} \right) }_{0 \leq j_1,\ldots, j_k \leq T - 1} . \] Then, define $(S_m)_{\mM_n} \egaldef \paren{S_{m(T)}}_{1 \leq T \leq n^{1/k}}$.
\begin{enumerate}
\item[0.] Replace $\M_n$ by
\begin{equation*} 
\Mh_n = \Big\{ \mM_n \telque \min_{\lamm} \{ n\phl \} \geq 3 \Big\} .
\end{equation*}
\item[1.] Choose a resampling scheme $\loi(W)$ among Efr, Rad, Poi, Rho and Loo.
\item[2.] Take the constant $C = \CWinf$ as defined in Table~\ref{RP.TableCWinfty}.
\item[3.] For each $\mMh_n$, compute the resampling penalty $\pen(m)$ defined by \eqref{RP.def.pen.his}.
\item[4.] Select $\mh \in \arg\min_{\mMh_n} \set{P_n\gamma\paren{\ERM_m} + \pen(m)}$.
\item[5.] Define $\widetilde{s} \egaldef \ERM_{\mh}$.
\end{enumerate}
\end{proc}

\begin{theorem} \label{RP.the.holder}
Let $\X = [0,1]^k$. Assume that the data $(X_i,Y_i)_{1 \leq i \leq n}$ satisfy the following:
\makeatletter
\setlength\leftmargini   {30\p@}
\makeatother
\begin{enumerate}
\item[\hypAb] Bounded data: $\norm{Y_i}_{\infty} \leq A < \infty$.
\item[\hypAn] Noise-level bounded from below: $\sigma(X_i) \geq \sigmin>0$ a.s.
\item[\hypAd] Density bounded from below:
\[ \exists c_X^{\min}>0, \quad \forall I \subset \X, \qquad P(X \in I) \geq c_X^{\min} \Leb(I) . \]
\item[\hypAh] H\"olderian regression function: there exist $\alpha\in(0;1]$ and $R>0$ such that
\begin{equation*} \bayes \in \mathcal{H}(\alpha,R) \quad \mbox{ that is } \quad \forall x_1, x_2 \in \X, \, \absj{\bayes(x_1) - \bayes(x_2)} \leq R \norm{x_1 - x_2}_{\infty}^{\alpha} . \end{equation*}
\end{enumerate}

Let $\widetilde{s}$ be the estimator defined by Procedure~\ref{RP.def.proc.his.holder} and $\sigmax \egaldef \sup_{\X} \absj{\sigma} \leq 2A$. Then, there exist positive constants $K_2$ and $K_3$ such that,
\begin{gather}
\label{RP.eq.oracle_class_non-as.hold}
\E\croch{ \perte{\widetilde{s}} } \leq
K_2 R^{\frac{2k}{2\alpha +k}} n^{\frac{-2\alpha }{2\alpha +k}} \sigmax^{\frac{4\alpha }{2\alpha+k}} + K_3 A^2 n^{-2} .
\end{gather}
%
\noindent
If moreover the noise-level is smooth, that is
\begin{enumerate}
\item[\hypAsig] $\sigma$ is piecewise $K_{\sigma}$-Lipschitz with at most $J_{\sigma}$ jumps,
\end{enumerate}
then, assumption \hypAn\ can be removed and \eqref{RP.eq.oracle_class_non-as.hold} holds with $\sigmax$ replaced by $\norm{\sigma}_{L^2(\Leb)} \egaldef [(\Leb(\X))^{-1} \int_{\X} \sigma^2(t) dt]^{1/2}$.

For both results, $K_2$ may only depend on $\alpha$ and $k$.
The constant $K_3$ may only depend on $k$, $A$, $c_X^{\min}$, $R$, $\alpha$ \textup{(}and $\sigmin$ for \eqref{RP.eq.oracle_class_non-as.hold}; $K_{\sigma}$ and $J_{\sigma}$ for the latter result\textup{)}.
\end{theorem}

Theorem~\ref{RP.the.holder} is proved in Section~\ref{RP.sec.proof.holder}.
The upper bounds given by Theorem~\ref{RP.the.holder} coincide with several classical minimax lower bounds on the estimation of functions in $\mathcal{H}(\alpha,R)$ with $\alpha \in (0,1]$, up to an absolute constant.
In the homoscedastic case, lower bounds have been proved by Stone \cite{Sto:1980} and generalized by several authors among which Korostelev and Tsybakov \cite{Kor_Tsy:1993} and Yang and Barron \cite{Yan_Bar:1999}. Up to a multiplicative factor independent of $n$, $R$ and $\sigma$ the best achievable rate is
\[ R^{\frac{2k}{2\alpha +k}} n^{\frac{-2\alpha }{2\alpha +k}} \sigma^{\frac{4\alpha }{2\alpha+k}} . \]
Hence, \eqref{RP.eq.oracle_class_non-as.hold} shows that Procedure~\ref{RP.def.proc.his.holder} achieves the right estimation rate in terms of $n$, $R$ and $\sigma$, without using the knowledge of $\alpha$, $R$ or $\sigma$.
%

Moreover, \eqref{RP.eq.oracle_class_non-as.hold} still holds in a wide heteroscedastic framework, without using any information on the noise-level $\sigma(\cdot)$. Then, up to a multiplicative constant independent of $n$ and $R$ (but possibly of the order of some power of $\sigmax / \sigmin$), the upper bound \eqref{RP.eq.oracle_class_non-as.hold} is the best possible estimation rate.

Minimax lower bounds proved in the heteroscedastic case (see for instance \cite{Efr_Pin:1996,Gal_Per:2008} and references therein) show that when $k=\alpha=1$ and the noise-level is smooth enough, the best achievable estimation rate depends on $\sigma$ through the multiplicative factor $\norm{\sigma}_{L^2(\Leb)}^{\frac{4 \alpha}{2\alpha + k}}$. Therefore, the upper bound given by Theorem~\ref{RP.the.holder} under assumption \hypAsig\ is tight, even through its dependence on the noise-level. Up to our best knowledge, such an upper bound had never been obtained when $\alpha \in (0,1)$ and $k>1$, even with estimators using the knowledge of $\alpha$, $\sigma$ and $R$.


Theorem~\ref{RP.the.holder} shows that Procedure~\ref{RP.def.proc.his.holder} defines an {\em adaptive estimator}, uniformly over distributions such that $\bayes$ belongs to some H\"olderian ball $\mathcal{H} (\alpha,R)$ with $\alpha \in (0,1]$ and the noise-level $\sigma$ is not too pathological. This result is quite strong.
Although similar properties have already been proved for ``\latin{ad hoc}'' estimators (see \cite{Efr_Pin:1996,Gal_Per:2008} and Section~\ref{RP.sec.comparaison.ad-hoc}), {\em Resampling Penalization has not been designed specifically to have such a property}.
Therefore, exchangeable resampling penalties are {\em naturally adaptive} to the smoothness of $\bayes$ and to the heteroscedasticity of the data.

\makeatletter
\setlength\leftmargini   {22\p@}
\makeatother
\begin{remark} \hfill
\begin{enumerate}
\item The proof of Theorem~\ref{RP.the.holder} shows that $\ERM_{\mh}$ achieves the minimax rate of estimation on an event of probability larger than $1 - K_3^{\prime} n^{-2}$. In particular, with probability one,
\[ \limsup_{n \rightarrow \infty} \parenB{ \perte{\widetilde{s}} R^{\frac{-2k}{2\alpha +k}} n^{\frac{2\alpha }{2\alpha +k}} \norm{\sigma}_{L^2(\Leb)}^{\frac{-4\alpha }{2\alpha+k}} } \leq K_2 (\alpha,k) . \]
\item If $\bayes$ is piecewise $\alpha$-H\"olderian with at most $J_{\bayes}$ jumps (each jump of height bounded by $2 A$), then \eqref{RP.eq.oracle_class_non-as.hold} holds with $K_3$ depending also on $J_{\bayes}$.
\item As for Theorem~\ref{RP.the.oracle_traj_non-as}, the boundedness of the data and the lower bound on the noise level can be replaced by other assumptions (see Section~\ref{RP.sec.main.hyp.alt}).
\end{enumerate}
\end{remark}

\subsection{Discussion on some assumptions} \label{RP.sec.main.hyp}
The aim of this subsection is to discuss some of the main assumptions made in Theorems~\ref{RP.the.oracle_traj_non-as} and~\ref{RP.the.holder}.
We first tackle the lower bound in \hypAp\ which is required in Theorem~\ref{RP.the.oracle_traj_non-as}.
Then, two alternative assumption sets to Theorems~\ref{RP.the.oracle_traj_non-as} and~\ref{RP.the.holder} are provided, allowing the noise level to vanish or the data to be unbounded.
%
\subsubsection{Lower bound in $\mathbf{(Ap)}$} \label{RP.sec.main.hyp.Ap}
The lower bound $\perte{\bayes_m} \geq \cbiasmin D_m^{-\beta_1}$ in \hypAp\ may seem unintuitive because it means that $\bayes$ is not too well approximated by the models $S_m$.
Assuming that $\inf_{\mM_n} \perte{\bayes_m}>0$ is classical for proving the asymptotic optimality of Mallows' $C_p$ \cite{Shi:1981,KCLi:1987,Bir_Mas:2006}.

Let us explain why \hypAp\ is used for proving Theorem~\ref{RP.the.oracle_traj_non-as}. According to Remark~\ref{RP.rem.oracle.traj.sansAp} in Section~\ref{RP.sec.preuves.gal}, when the lower bound in \hypAp\ is no longer assumed, \eqref{RP.eq.oracle_traj_non-as} holds with two modifications on its right-hand side: the infimum is restricted to models of dimension larger than $\paren{\ln(n)}^{\gamma_1}$ and a remainder term $\paren{\ln(n)}^{\gamma_2} n^{-1}$ is added (where $\gamma_1$ and $\gamma_2$ are absolute constants).
This is essentially the same as \eqref{RP.eq.oracle_traj_non-as} unless there exists a model of small dimension with a small bias; the lower bound in \hypAp\ is sufficient to ensure this does not happen.
%
Note that assumption \hypAp\ was made in the density estimation framework \cite{Sto:1985,Bur:2002} for the same technical reasons.


As showed in \cite{Arl:2008b:app}, \hypAp\ is at least satisfied with
\[ \beta_1 = k^{-1} + \alpha^{-1} - (k-1) k^{-1} \alpha^{-1} \qquad \mbox{and} \qquad \beta_2 = 2 \alpha k^{-1} \] in the following case:
$(\Il)_{\lamm}$ is ``regular'' (as defined in Procedure~\ref{RP.def.proc.his.holder} below), $X$ has a lower-bounded density with respect to the Lebesgue measure $\Leb$ on $\X \subset \R^k$ and $\bayes$ is non-constant and $\alpha$-H\"olderian (with respect to $\norm{\cdot}_{\infty}$).

The general formulation of \hypAp\ is crucial to make Theorem~\ref{RP.the.oracle_traj_non-as} valid {\em whatever the distribution of $X$} which can be useful in some practical problems.
Indeed, when $X$ has a general distribution, a collection $\paren{S_m}_{\mM_n}$ satisfying \hypPpoly, \hypPrich, \hypArXl\ and \hypAp\ can always be chosen either thanks to prior knowledge on $\loi(X)$ or to unlabeled data.
In the latter case, classical density estimation procedures can be applied for estimating $\loi(X)$ from unlabeled data (see for instance \cite{Dev_Lug:2001} on density estimation).
Assumption \hypAp\ then means that the collection of models has good approximation properties, uniformly over some appropriate function space (depending on $\loi(X)$) to which $\bayes$ belongs.

\subsubsection{Two alternative assumption sets} \label{RP.sec.main.hyp.alt}
Theorems~\ref{RP.the.oracle_traj_non-as} and~\ref{RP.the.holder} are corollaries of a more general result, called Lemma~\ref{RP.le.gal} in Section~\ref{RP.sec.preuves.gal}.
The assumptions of Theorems~\ref{RP.the.oracle_traj_non-as} and~\ref{RP.the.holder}, in particular \hypAb\ and \hypAn\ on the distribution of the noise $\sigma(X) \epsilon$, are only sufficient conditions for the assumptions of Lemma~\ref{RP.le.gal} to hold.
The following two alternative sufficient conditions are proved to be valid in Section~\ref{RP.sec.preuves.Theoremalt}.

First, one can have $\sigmin=0$ in \hypAn\ if moreover $\E\croch{\sigma(X)^2}>0$, $\X \subset \R^k$ is bounded and
\makeatletter
\setlength\leftmargini   {30\p@}
\makeatother
\begin{enumerate}
\item[\hypArDu] Upper regularity of the partitions for $\norm{\cdot}_{\infty}$: $\exists \crDu, \alpha_d > 0$ such that
\[ \forall \mM_n, \quad \max_{\lamm} \set{\diam(\Il)} \leq \crDu D_m^{-\alpha_d} . \]
\item[\hypArLu] Upper regularity of the partitions for $\Leb$: $\exists \crLu > 0$ such that
\[ \forall \mM_n, \quad \max_{\lamm} \set{\Leb(\Il)} \leq \crLu D_m^{-1} . \]
\item[\hypAsig] $\sigma$ is piecewise $K_{\sigma}$-Lipschitz with at most $J_{\sigma}$ jumps.
\end{enumerate}

Second, the $Y_i$ can be unbounded (assuming now that $\sigmin>0$ in \hypAn) if moreover $\X \subset \R$ is bounded measurable and
\makeatletter
\setlength\leftmargini   {40\p@}
\makeatother
\begin{enumerate}
\item[\hypAsg] The noise is sub-Gaussian: $\exists \cgauss >0$ such that \begin{equation*} \forall q \geq 2, \, \forall x \in \X, \quad \E \croch{ \absj{\epsilon}^q \sachant X=x }^{1/q} \leq \cgauss \sqrt{q} .\end{equation*}
\item[\hypAsigmax] Noise-level bounded from above: $\sigma^2(X) \leq \sigmax^2 < + \infty$ a.s.
\item[\hypAciblemax] Bound on the regression function: $\norm{\bayes}_{\infty} \leq A$.
\item[\hypAlip] $\bayes$ is $B$-Lipschitz, piecewise $C^1$ and non-constant: $\pm \bayes^{\prime} \geq B_0 >0$ on some interval $J \subset \X$ with $\Leb(J) \geq c_J >0$.
\item[\hypAreg] Regularity of the partitions for $\Leb$: $\exists \crLl,\crLu>0$ such that
\begin{equation*}
\forall \mM_n, \, \forall \lamm, \quad \crLl D_m^{-1} \leq \Leb(\Il) \leq \crLu D_m^{-1} .
\end{equation*}
\item[\hypAd] Density bounded from below: $\exists c_X^{\min}>0$, $\forall I \subset \X$, $\Prob (X \in I) \geq c_X^{\min} \Leb(I)$.
\end{enumerate}
\makeatletter
\setlength\leftmargini   {22\p@}
\makeatother

Third, it is possible to have simultaneously $\sigmin=0$ in \hypAn\ and unbounded data, see \cite{Arl:2008b:app} for details.


The above results mean that Theorem~\ref{RP.the.oracle_traj_non-as} holds for most ``reasonably'' difficult problems. Actually, Proposition~\ref{RP.pro.conc.penRP} and Remark~\ref{RP.rem.conc.penRP} show that the resampling penalties are much closer to $\E\croch{\penid(m)}$ than $\penid(m)$ itself, provided that the concentration inequalities for $\penid$ are tight (Proposition~\ref{RP.pro.conc.penid}).
Therefore, up to differences within $\varepsilon_n$, RP with $C = \CWinf$ and the ``ideal'' deterministic penalization procedure $\E\croch{\penid(m)}$ perform equally well on a set of probability $1 - K_1 n^{-2}$.
For every assumption set such that the proof of Theorem~\ref{RP.the.oracle_traj_non-as} gives an oracle inequality for the penalty $\E\croch{\penid(m)}$, the same proof gives a similar oracle inequality for RP.

\subsection{Probabilistic tools} \label{RP.sec.tools}
Theorems~\ref{RP.the.oracle_traj_non-as} and~\ref{RP.the.holder} rely on several probabilistic tools of independent interest: precise computation of the expectations of resampling penalties (Propositions~\ref{RP.pro.EpenRP-Epenid} and~\ref{RP.pro.comp.Epen.Ep2}), concentration inequalities for resampling penalties (Proposition~\ref{RP.pro.conc.penRP}) and bounds on expectations of the inverses of several classical random variables (Lemma~\ref{RP.le.einv.binom}--\ref{RP.le.einv.Poi}).
Their originality comes from their non-asymptotic nature: explicit bounds on the deviations or the remainder terms are provided for finite sample sizes.

\allowdisplaybreaks

\subsubsection{Expectations of resampling penalties} \label{RP.sec.histos.exp}
Using only the exchangeability of the weights, the resampling penalty can be computed explicitly (Lemma~\ref{VFCV.le.cal.pen.Wech} in Section~\ref{sec.proof.expect}). This can be used to compare the expectations of the resampling penalties and the ideal penalty. First, Proposition~\ref{RP.pro.EpenRP-Epenid} is valid for {\em general exchangeable weights}.
\begin{proposition} \label{RP.pro.EpenRP-Epenid}
Let $S_m$ be the model of histograms associated with some partition $(\Il)_{\lamm}$ of $\X$ and $W \in [0,\infty)^n$ be an exchangeable random vector independent of the data.
Define $\penid(m)$ by \eqref{RP.def.penid} and $\pen(m)$ by \eqref{RP.def.pen.his}.
Let $\El\croch{\cdot}$ denote expectations conditionally on $(\un_{X_i \in \Il})_{1 \leq i \leq n, \, \lamm}$.
Then, if $\min_{\lamm} \phl >0$,
\begin{align} \label{RP.eq.Em_penid}
\El\croch{ \penid(m) } =& \frac{1}{n} \sum_{\lamm} \paren{ 1 + \frac{\pl}{\phl} } \sigl^2 \\ \label{RP.eq.Em_pen}
\El\croch{ \pen(m) } =& \frac{C}{n} \sum_{\lamm} \paren{ R_{1,W}(n,\phl) + R_{2,W}(n,\phl) } \sigl^2 
\\ \notag \mbox{with} \quad \sigl^2 \egaldef& \E \croch{ \paren{Y - \bayes(X)}^2 \sachant X \in \Il } , \\ 
\label{RP.def.R1}
R_{1,W}(n,\phl) \egaldef& \E\croch{ \frac{(W_1 - \Wl)^2}{\Wl^2} \sachant X_1 \in \Il, \Wl >0 } , \\
\mbox{and} \quad R_{2,W}(n,\phl) \egaldef& \E\croch{ \frac{(W_1 - \Wl)^2}{\Wl} \sachant X_1 \in \Il} . \label{RP.def.R2}
\end{align}
In particular,
\begin{equation} \label{RP.eq.Ep2.Epenid}
\E\croch{ \penid(m) } = \frac{1}{n} \sum_{\lamm} \paren{2 + \delta_{n, \pl}} \sigl^2
\end{equation}
where $\delta_{n,p}$ only depends on $(n,p)$ and satisfies $\absj{\delta_{n, p} } \leq L_1 (np)^{-1/4} $ for some absolute constant $L_1$.
\end{proposition}
Proposition~\ref{RP.pro.EpenRP-Epenid} is proved in Section~\ref{sec.proof.expect}.

\begin{remark}
\begin{itemize}
\item In order to make the expectation in \eqref{RP.eq.Ep2.Epenid} well-defined, a convention for $\penid(m)$ has to be chosen when $\min_{\lamm} \phl = 0$. See Section~\ref{RP.sec.proof.nota} for details.
\item Combining Proposition~\ref{RP.pro.EpenRP-Epenid} with \cite[Lemma~8.4]{Arl:2007:phd}, a similar result holds for non-exchangeable weights (with only a modification of the definitions of $R_{1,W}$ and $R_{2,W}$).
\end{itemize}
\end{remark}

In the general heteroscedastic framework \eqref{RP.eq.donnees.reg}, Proposition~\ref{RP.pro.EpenRP-Epenid} shows that resampling penalties take into account the fact that $\sigl^2$ actually depends on $\lamm$. This is a major difference with the classical Mallows' $C_p$ penalty
\[ \pen_{\mathrm{Mallows}} (m) \egaldef \frac{2 \E\croch{\sigma(X)^2} D_m}{n} \]
which does not take into account the variability of the noise level over $\X$. A more detailed comparison with Mallows' $C_p$ is made in Section~\ref{RP.sec.calc.pen.comp.Mal}.


If $R_{1,W}(n,\phl) + R_{2,W}(n,\phl)$ does not depend too much on $\phl$ (at least when $n \phl$ is large),
Proposition~\ref{RP.pro.EpenRP-Epenid} shows that $\pen(m)$ estimates unbiasedly $\penid(m)$
as soon as\footnote{The definition of \CWinf\ actually used in this paper is slightly different for
Efron($\mefr$) and Poisson($\mu$) weights (see Table~\ref{RP.TableR2}). We arbitrarily choosed the
simplest possible expression making \CWinf\ asymptotically equivalent to $1/\E[{\paren{W_1-1}^2}]$.
The results of the paper also hold when $\CWinf=1/\E[{\paren{W_1-1}^2}]$.}
\[ C = \CWinf \approx \frac{2}{R_{1,W}(n,1) + R_{2,W}(n,1)} = \frac{1}{\E\crochb{\paren{W_1-1}^2}} . \]
In particular, all the examples of resampling weights given in Section~\ref{sec.cadreRP.heur} satisfy that $R_{1,W}(n,\phl) \approx R_{2,W}(n,\phl)$ does not depend on $\phl$ when $n \phl$ is large, which leads to Proposition~\ref{RP.pro.comp.Epen.Ep2} below (see Table~\ref{RP.TableR2} for exact expressions of $R_{2,W}$ and \CWinf).

\begin{table}
\caption{$R_{2,W}(n,\phl)$ and \CWinf\ for several resampling schemes. The formulas for $R_{2,W}$ come from Lemma~\ref{RP.le.calc.R1.R2}\label{RP.TableR2}}
   \begin{tabular}{|r|c|c|c|c|c|} \hline
   $\loi(W)$ & Efr($\mefr$) & Rad($p$) & Poi($\mu$) & Rho($q$) & Loo \\
   \hline\\[-8pt]
   $R_{2,W}(n,\phl)$ & $\frac{n}{\mefr} \paren{1 - \frac{1}{n\phl}}$ & $\frac{1}{p} - 1$ & $\frac{1}{\mu} \paren{1 - \frac{1}{n\phl}}$ & $\frac{n}{q} - 1$ & $\frac{1}{n-1}$ \\[6pt]
   \hline
   \CWinf & $\mefr/n$ & $p/(1-p)$ & $\mu$ & $q/(n-q)$ & $n-1$ \\ \hline
   \end{tabular}
\end{table}

\begin{proposition} \label{RP.pro.comp.Epen.Ep2}
Let $W$ be an exchangeable resampling weight vector among Efr$(\mnefr)$, Rad$(p)$, Poi$(\mu)$, Rho$(\left\lfloor n/2\right\rfloor)$ and Loo, and define $\CWinf$ as in Table~\ref{RP.TableR2}. Let $S_m$ be the model of histograms associated with some partition $\paren{\Il}_{\lamm}$ of $\X$ and $\pen(m)$ be defined by \eqref{RP.def.pen.his}.
Then, there exist real numbers $\delta_{n,\phl}^{(\mathrm{penW})}$ depending only on $n$, $\phl$ and the resampling scheme $\loi(W)$ such that
\begin{equation} \label{RP.eq.comp.Epen.Ep2}
\El \croch{\pen(m)} = \frac{C}{\CWinf n} \sum_{\lamm} \paren{ 2 + \delta_{n,\phl}^{(\mathrm{penW})} } \sigl^2 .
\end{equation}

If $\mnefr n^{-1} \geq B > 0$ (Efr), $p \in (0;1)$ (Rad) or $\mu > 0$ (Poi), then,
\[ \forall n \in \N \backslash\set{0}, \, \forall \phl \in (0,1], \quad \absj{\delta_{n,\phl}^{(\mathrm{penW})}} \leq L_2 \paren{ n \phl }^{-1/4} , \]
where $L_2>0$ is an absolute constant \textup{(}for Rho$(\left\lfloor n/2\right\rfloor)$ and Loo\textup{)} or depends respectively on $B$ \textup{(}Efr\textup{)}, $p$ \textup{(}Rad\textup{)} or $\mu$ \textup{(}Poi\textup{)}.
More precise bounds for each weight distribution are given by \eqref{RP.eq.delta.penEfr}--\eqref{RP.eq.delta.penLoo} in Section~\ref{RP.sec.calc.reech}.
\end{proposition}
Proposition~\ref{RP.pro.comp.Epen.Ep2} is proved in Section~\ref{RP.sec.calc.reech}.

\begin{remark}
Proposition~\ref{RP.pro.comp.Epen.Ep2} can also be generalized to Rho($q_n$) weights with $0 < B_- \leq q_n n^{-1} \leq B_+ < 1$, but the bound on $\delta_{n,\phl}^{(\mathrm{penW})}$ only holds for $n \phl \geq L(B_-,B_+)$ and $L_2$ depends on $B_-,B_+$ (see Section~\ref{RP.sec.calc.reech}).
\end{remark}

\begin{remark} \label{RP.rem.R1W.R2W} Combined with the explicit expressions of \CWinf\ for several resampling weights (Table~\ref{RP.TableR2}), Proposition~\ref{RP.pro.comp.Epen.Ep2} helps to understand several known results.
\begin{itemize}
\item In the maximum likelihood framework, Shibata \cite{Shi:1997} showed the asymptotical equivalence of two bootstrap penalization methods. The first penalty, denoted by $B_1$, is Efron's bootstrap penalty \cite{Efr:1983}, which is defined by \eqref{RP.def.penreech} with Efr weights. The second penalty, denoted by $B_2$, was proposed by Cavanaugh and Shumway \cite{Cav_Shu:1997}; it transposes \[ 2 \ph_1(m) = 2 \Es \croch{ P_n \paren{\gamma(\ERMb_m) - \gamma(\ERM_m)} } \] into the maximum likelihood framework. In the least-squares regression framework (with histogram models), the proofs of Propositions~\ref{RP.pro.EpenRP-Epenid} and~\ref{RP.pro.comp.Epen.Ep2} show that
\[ \El \croch{2 \ph_1(m)} = \frac{2}{n} \sum_{\lamm} R_{1,W}(n,\phl) \sigl^2 \approx \El \croch{\pen(m)} \] for several resampling schemes, including Efron's bootstrap (for which $\CWinf = 1$). The concentration results of Section~\ref{RP.sec.proof.conc} show that this remains true without expectations.
\item With Efron ($\mnefr$) weights (and a bootstrap selection procedure close to RP, but with $C=1$), Shao \cite{Sha:1996} showed that $\mnefr=n$ leads to an inconsistent model selection procedure for identification. On the contrary, when $\mnefr \rightarrow \infty$ and $\mnefr \ll n$, Shao's bootstrap selection procedure is model consistent. Proposition~\ref{RP.pro.comp.Epen.Ep2} shows that these assumptions on $\mnefr$ can be rewritten $C=1 \gg \CWinf = \mnefr / n$. Therefore, the rationale behind Shao's result may mostly be that identification needs overpenalization within a factor tending to infinity with $n$.
\end{itemize}
\end{remark}

\subsubsection{Concentration inequalities for resampling penalties}
From \eqref{eq.penid.histos}, the ideal penalty can be written
\[ \penid(m) = (P-P_n) \gamma\paren{\bayes_m} + \sum_{\lamm} \frac{\pl + \phl}{\paren{n \phl}^2} \paren{\sum_{X_i \in \Il} \paren{Y_i - \betl} }^2 . \]
Hence, $\penid(m)$ is a U-statistics of order 2 conditionally on $(\un_{X_i \in \Il})_{(i,\lamm)}$, which is sufficient to prove that resampling yields a consistent estimate of $\penid(m)$ (Arcones and Gin\'e \cite{Arc_Gin:1992} considered the bootstrap case; Hu\v{s}kov\'a and Janssen \cite{Hus_Jan:1993} extended it to the exchangeable weighted bootstrap).

In the non-asymptotic framework, that is when the models $S_m$ can depend on $n$, the following concentration inequality is needed.
\begin{proposition} \label{RP.pro.conc.penRP}
Let $\gamma>0$, $A_n \geq 2$ and $W$ be an exchangeable weight vector. Let $S_m$ be the model of histograms associated with some partition $(\Il)_{\lamm}$ of $\X$ and $\pen(m)$ be defined by~\eqref{RP.def.pen.his}. Assume that two positive constants $a_{\ell}$ and $\xi_{\ell}$ exist such that for every $q \geq 2$,
\begin{gather*}
\frac{\sqrt{D_m \sum_{\lamm} m_{q,\lambda}^4}}{\sum_{\lamm} m_{2,\lambda}^2 } \leq a_{\ell} q^{\xi_{\ell}} \quad \mbox{where} \quad m_{q,\lambda} \egaldef \left( \E \left[ \absj{Y -
\bayes_m(X)}^q \sachant X \in \Il \right] \right)^{1/q} .
\end{gather*}
Let $\Omega_m(A_n)$ denote the event $\set{ \min_{\lamm} \set{n \phl} \geq A_n}$.
Then, there exist constants $K_4,K_5>0$ and an event of probability at least $1 - K_4 n^{-\gamma}$ on which
\begin{equation*}
\begin{split}
 \absj{\pen(m) - \El \croch{\pen(m)} } \un_{\Omega_m(A_n)} \leq C K_5 \qquad \qquad \qquad \qquad \\
\times \sup_{np \geq A_n} \set{ R_{1,W}(n,p) + R_{2,W}(n,p) } \frac{\paren{\ln(n)}^{\xi_{\ell} + 1}}{\sqrt{A_n D_m}} \E\croch{p_2(m)}
 \end{split}
\end{equation*}
where $R_{1,W}$ and $R_{2,W}$ are defined by \eqref{RP.def.R1} and \eqref{RP.def.R2}. The constant $K_4$ is absolute and $K_5$ may only depend on $a_{\ell}$, $\xi_{\ell}$ and $\gamma$.

If moreover $W$ satisfies the assumptions of the second part of Proposition~\ref{RP.pro.comp.Epen.Ep2} and $\CWinf$ is defined as in Table~\ref{RP.TableR2}, then a constant $K_W>0$ exists such that
\begin{equation} \label{RP.eq.conc.penRP}
\begin{split}
 \absj{\pen(m) - \El [ \pen(m) ]} \un_{\Omega_m(A_n)} \leq \frac{C K_5 K_W \paren{\ln(n)}^{\xi_{\ell} + 1}}{\CWinf \sqrt{A_n D_m}} \E\croch{p_2(m)} 
 .
 \end{split}
\end{equation}
For the Rad$(p)$ weights, $K_W$ is smaller than $(1-p)^{-1}$ multiplied by an absolute constant. For the other weights, $K_W$ is an absolute constant.
\end{proposition}
Proposition~\ref{RP.pro.conc.penRP} is proved in Section~\ref{RP.sec.proof.pro.conc.penRP}.
Note that the moment condition holds under the assumptions of Theorem~\ref{RP.the.oracle_traj_non-as} as well as the alternative assumptions of Section~\ref{RP.sec.main.hyp.alt}. It is here stated in its most general form.
\begin{remark} \label{RP.rem.conc.penRP}
Since the $A_n^{-1/2}$ factor should tend to infinity with $n$ for most reasonable models, Proposition~\ref{RP.pro.conc.penRP} gives better bounds for resampling penalties than what could be obtained for ideal penalties with Proposition~\ref{RP.pro.conc.penid} in the same framework.

Although we do not know how tight are the bounds of Proposition~\ref{RP.pro.conc.penRP}, such a phenomenon is classical with bootstrap and can be understood from the asymptotic point of view through Edgeworth expansions \cite{Hal:1992}. In a non-asymptotic Gaussian framework, \cite[Section~2.3]{Arl_Bla_Roq:2008:RC} shows the same property for resampling estimators, which concentrate at the rate $N^{-1}$ instead of $N^{-1/2}$ ($N$ being the amount of data).
Since $A_n$ plays the role of $N$, the gain $A_n^{-1/2}$ can reasonably be conjectured to be unimprovable without some more assumptions.

Let us emphasize that if resampling penalties estimate $\E\croch{\penid(m)}$ instead of $\penid(m)$, RP with $C=\CWinf$ cannot take into account the fact that $\penid(m)$ may be far from its expectation.
\end{remark}

\subsubsection{Expectations of inverses}\label{RP.sec.app.einv}
For any non-negative random variable $Z$, we define
\[ \einv{Z} = \einv{\loi(Z)} \egaldef \E \croch{Z} \E \croch{ Z^{-1} \sachant Z>0} . \]
This quantity appears in the explicit formulas for $R_{1,W}$ when $W$ is among the examples of resampling weights of Section~\ref{sec.cadreRP.heur} (see Lemma~\ref{RP.le.calc.R1.R2}).
Therefore, in order to prove Proposition~\ref{RP.pro.comp.Epen.Ep2}, non-asymptotic bounds on $\einv{Z}$ are needed when $Z$ has a binomial, hypergeometric or Poisson distribution.

Former results concerning $\einv{Z}$ can be found in papers by Lew \cite{1976:Lew} (for general $Z$), by Jones and Zhigljavsky \cite{2004:Jon_Zhi} (for the Poisson case) and by \v{Z}nidari\v{c} \cite{2005:Zni} (for the binomial and Poisson case), but they are either asymptotic or not precise enough.
Lemmas~\ref{RP.le.einv.binom}--\ref{RP.le.einv.Poi} solve this issue.

In the rest of the paper, for any $a,b \in \R$, $\mini{a}{b}$ denotes the minimum of $a$ and $b$ and $\maxi{a}{b}$ denotes the maximum of $a$ and $b$.

\paragraph{Binomial case}
\begin{lemma} \label{RP.le.einv.binom}
For any $n \in \N \backslash\set{0}$ and $p \in (0;1]$, $\mathcal{B}(n,p)$ denotes the binomial distribution with parameters $(n,p)$, $\kappa_1 \egaldef 5.1$ and $\kappa_2 \egaldef 3.2$. Then, if $np \geq 1$,
\begin{gather} \label{RP.eq.einv-binom.inf-sup}
\mini{ \kappa_2 }{ \parenb{ 1+\kappa_1 (np)^{-1/4} } } \geq \einv{\mathcal{B}(n,p)} \geq 1 - e^{-np} \\
\mbox{and} \qquad
\label{aRP.eq.einv.sym}
2 + 3 \times 10^{-4} \geq \einv{\mathcal{B}\paren {n, \frac{1}{2} } } \geq \un_{n \geq 3} . \end{gather}
\end{lemma}
The first bounds \eqref{RP.eq.einv-binom.inf-sup} were first stated in \cite[Lemma~3]{Arl:2008a} where they are proved. The second ones \eqref{aRP.eq.einv.sym} are proved in Section~\ref{aRP.sec.einv.sym}.
Lemma~\ref{RP.le.einv.binom} implies in particular that $\einv{\mathcal{B}(n,p)} \rightarrow 1$ when $np \rightarrow \infty$, which can be derived from \cite{2005:Zni}.

\paragraph{Hypergeometric case}
Recall that an hypergeometric random variable $X \sim \mathcal{H}(n,r,q)$ is defined by \[ \forall k \in \set{0, \ldots, \mini{q}{r}} , \quad \Prob(X=k) = \frac{\binom{r}{k} \binom{n-r}{q-k}} {\binom{n}{q}} .\]
\begin{lemma} \label{RP.le.einv.hypergeom} Let $n,r,q\in \N$ such that $n \geq r \geq 1$ and $n \geq q \geq 1$.
\begin{enumerate}
\item General lower-bound:
\begin{equation*} 
\einv{\mathcal{H}(n,r,q)} \geq 1 - \un_{r \leq n-q} \exp \left( - \frac{qr}{n} \right) .
\end{equation*}
\item General upper-bound:
let $\epsilon \in (0;1)$ and $\kappa_3(\epsilon) \egaldef 0.9 + 1.4 \times \epsilon^{-2}$. \begin{align} \notag &\mbox{If} \qquad r\geq 2 \quad \mbox{and} \quad \frac{n}{q} \leq (1 - \epsilon) \frac{2r}{2 + \sqrt{3 (r+1) \ln(r) }} \\
\label{RP.eq:einv_hypergeom:maj_non-asympt}
&\mbox{Then,} \qquad 1 + \kappa_3(\epsilon) \frac{n}{q} \sqrt{ \frac{\ln(r)}{r} } \geq \einv{\mathcal{H}(n,r,q)} .
\end{align}
\item ``Rho'' case: if $n \geq 2$,
\begin{equation}
\label{RP.eq.einv.hypergeom.sup}
14.3 \geq \sup_{r \geq 1} \set{ \einv{\mathcal{H}(n,r,\lfloor \frac{n}{2} \rfloor)} } \qquad \mbox{and} \qquad 3 \geq \sup_{r \geq 26}
\set{ \einv{\mathcal{H}(n,r,\lfloor \frac{n}{2} \rfloor)} } .
\end{equation}
\item ``Loo'' case:
\begin{equation}
\label{RP.eq.einv.Loo}
\hspace{-0.5cm} 1 + \frac{\un_{r \geq 2}}{n (r-1)} \geq \einv{\mathcal{H}(n,r,n-1)} = 1 + \frac{1}{n} \paren{ \frac{(n-1)r}{n(r-1)} \un_{r \geq 2} - 1} \geq 1 - \frac{\un_{r = 1}}{n} .
\end{equation}
\item ``Lpo'' case: if $n \geq r \geq n-q+1 \geq 2$,
\begin{equation*}
\frac{r}{r-n+q} \times \frac{n^{n-q}} {n (n-1)\cdots (q+1)} \geq \einv{\mathcal{H}(n,r,q)} \geq 1 .
\end{equation*}
\end{enumerate}
\end{lemma}
Lemma~\ref{RP.le.einv.hypergeom} is proved in Section~\ref{aRP.sec.einv.hypergeom}.
It implies in particular that
\[ \einv{\mathcal{H}(n_k,r_k,q_k)} \xrightarrow[k \rightarrow \infty]{} 1 \quad \mbox{if }
n_k \geq r_k \xrightarrow[k \rightarrow \infty]{} + \infty\] and $\sup_k \set{n_k q_k^{-1}} < + \infty$.

\paragraph{Poisson case}
\begin{lemma}\label{RP.le.einv.Poi}
For every $\mu>0$, $\mathcal{P}(\mu)$ denotes the Poisson distribution with parameter $\mu$. Then,
\begin{equation*} 
\minipar{2 - 2 e^{-2 \mu}} {1 + \frac{2(1+e^{-3}) } {(\mu - 2) \un_{\mu > 2}}} \geq \einv{\mathcal{P}(\mu)} \geq 1 - \un_{\mu < 1.61} e^{-\mu}
. \end{equation*}
\end{lemma}
Lemma~\ref{RP.le.einv.Poi} is proved in Section~\ref{aRP.sec.einv.poisson}.
It implies in particular that $\einv{\mathcal{P}(\mu)} \rightarrow 1$ when $\mu \rightarrow \infty$, which can be derived from \cite{2004:Jon_Zhi,2005:Zni}.

\section{Comparison of the weights} \label{RP.sec.calc.delta}
We investigate in this section how the loss of the final estimator may depend on the distribution of the exchangeable weight vector $W$. First, we consider in Section~\ref{sec.calc.delta.classic} the most classical ones, that is Efr, Rad, Poi, Rho and Loo. Then, we discuss in Section~\ref{sec.Thm1.autresW} whether Theorem~\ref{RP.the.oracle_traj_non-as} can be extended to general exchangeable weights.

\subsection{Comparison of the classical weights} \label{sec.calc.delta.classic}

According to Theorem~\ref{RP.the.oracle_traj_non-as}, any resampling scheme among Efr, Rad, Poi, Rho and Loo leads to an asymptotically optimal procedure. Even from the non-asymptotic point of view, it is not quite clear to distinguish between these weights with the results of Section~\ref{RP.sec.main}. Indeed, the resampling penalties are equal in expectation at first order (Proposition~\ref{RP.pro.comp.Epen.Ep2}), and their deviations are negligible in front of their expectations (Proposition~\ref{RP.pro.conc.penRP}).

Therefore, differences between these weights can only come from second-order terms, either in the expectations or in the sizes of the deviations of resampling penalties.
As a first step, we compare in this subsection second-order terms in the expectations of the penalties (that is, differences between second-order terms in \eqref{RP.eq.Ep2.Epenid} and \eqref{RP.eq.comp.Epen.Ep2}), for a fixed sample size.
Asymptotic considerations can be found in the book by Barbe and Bertail \cite[Chapter~2]{Bar_Ber:1995} where Edgeworth expansions are used to compare the accuracy of estimation with many exchangeable weights. The asymptotic results mentioned in Section~\ref{RP.sec.app.einv} may also be useful.


Propositions~\ref{RP.pro.EpenRP-Epenid} and~\ref{RP.pro.comp.Epen.Ep2} show that $\penid(m)$ and $\pen(m)$ have the same expectation, up to the small terms $\delta_{n,\pl}$ and $\delta_{n,\phl}^{(\mathrm{penW})}$. More precisely,
\begin{gather*}
\E \croch{\pen(m) - \penid(m)} = \frac{1}{n} \sum_{\lamm} \paren{\overline{\delta}_{n,\pl}^{(\mathrm{penW})} - \delta_{n,\pl}} \carre{\sigl} \\
\mbox{with} \qquad \overline{\delta}_{n,\pl}^{(\mathrm{penW})} \egaldef \E\croch{\delta_{n,\phl}^{(\mathrm{penW})} \sachant \phl >0} .\end{gather*}
Using the explicit expressions of $\delta_{n,p}$ and $\delta_{n,\phl}^{(\mathrm{penW})}$, $\delta_{n,p}$ and $\overline{\delta}_{n,p}^{(\mathrm{penW})}$ have been computed numerically as a function of $np$ for several resampling schemes, with $n=200$. The results are given on Figures~\ref{RP.fig.delta.efr}--\ref{RP.fig.delta.rad} (with straight lines for $\delta_{n,p}$ and dots for $\overline{\delta}_{n,p}^{(\mathrm{penW})}$).

\begin{figure}[t!]
\begin{minipage}[b]{.48\linewidth}
\centerline{\epsfig{file=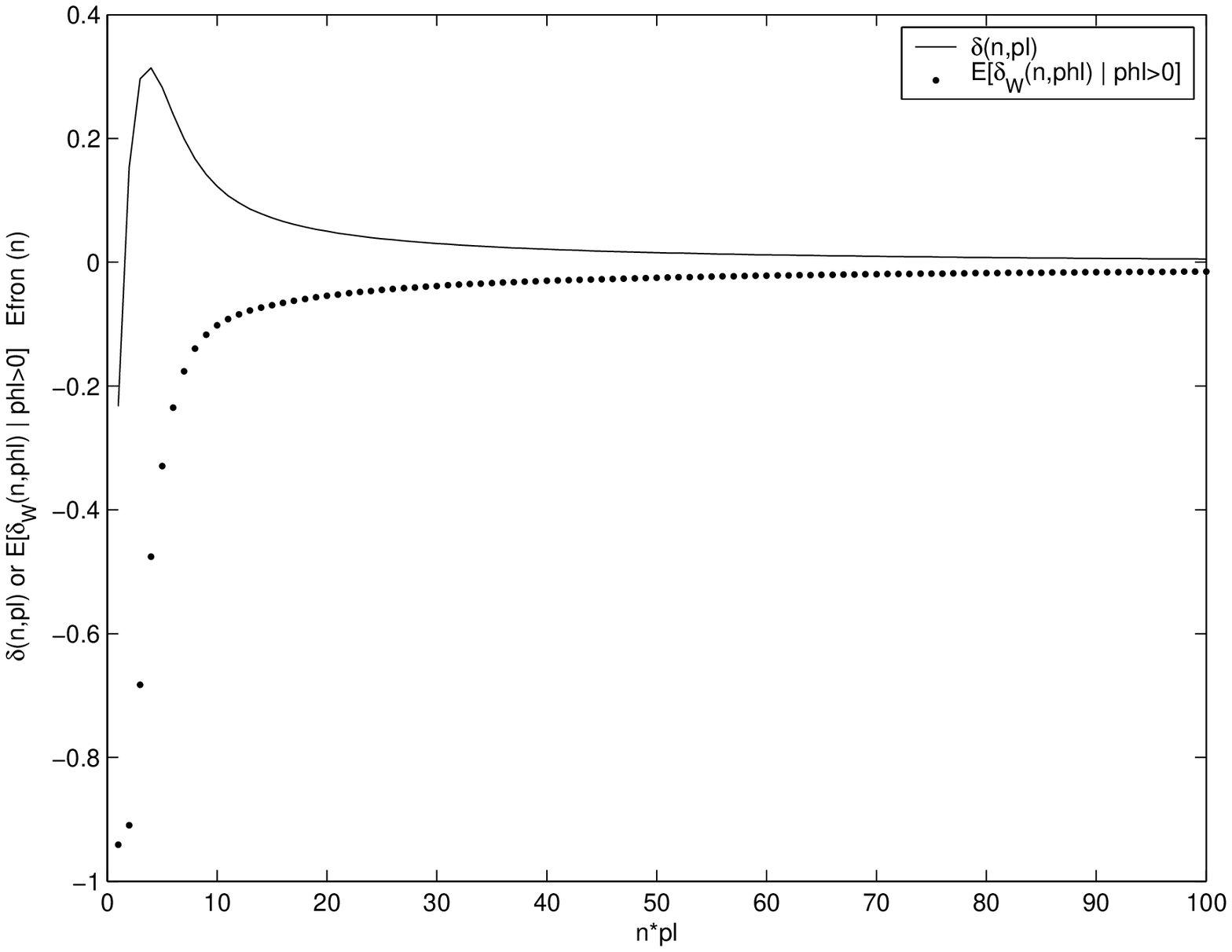,width=\textwidth}}
\vspace*{-3pt}
\caption{$\delta_{n,p}>0>\overline{\delta}_{n,p}^{(\mathrm{penEfr(n)})}$. \label{RP.fig.delta.efr}}
\end{minipage} \hfill
\begin{minipage}[b]{.48\linewidth}
\centerline{\epsfig{file=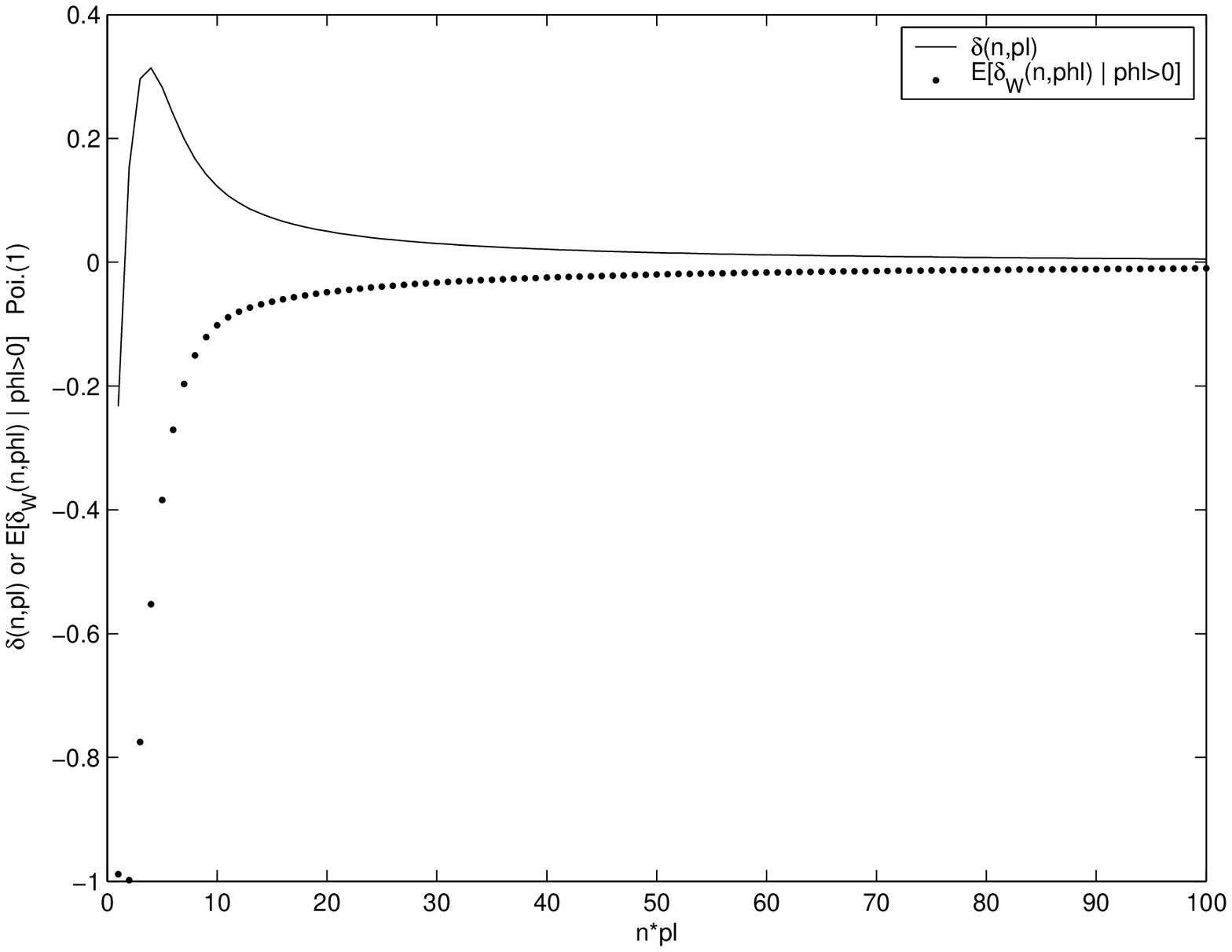,width=\textwidth}}
\vspace*{-3pt}
   \caption{$\delta_{n,p} > 0 > \overline{\delta}_{n,p}^{(\mathrm{penPoi}(1))}$. \label{RP.fig.delta.poi}}
\end{minipage}

\medskip

\begin{minipage}[b]{.48\linewidth}
\centerline{\epsfig{file=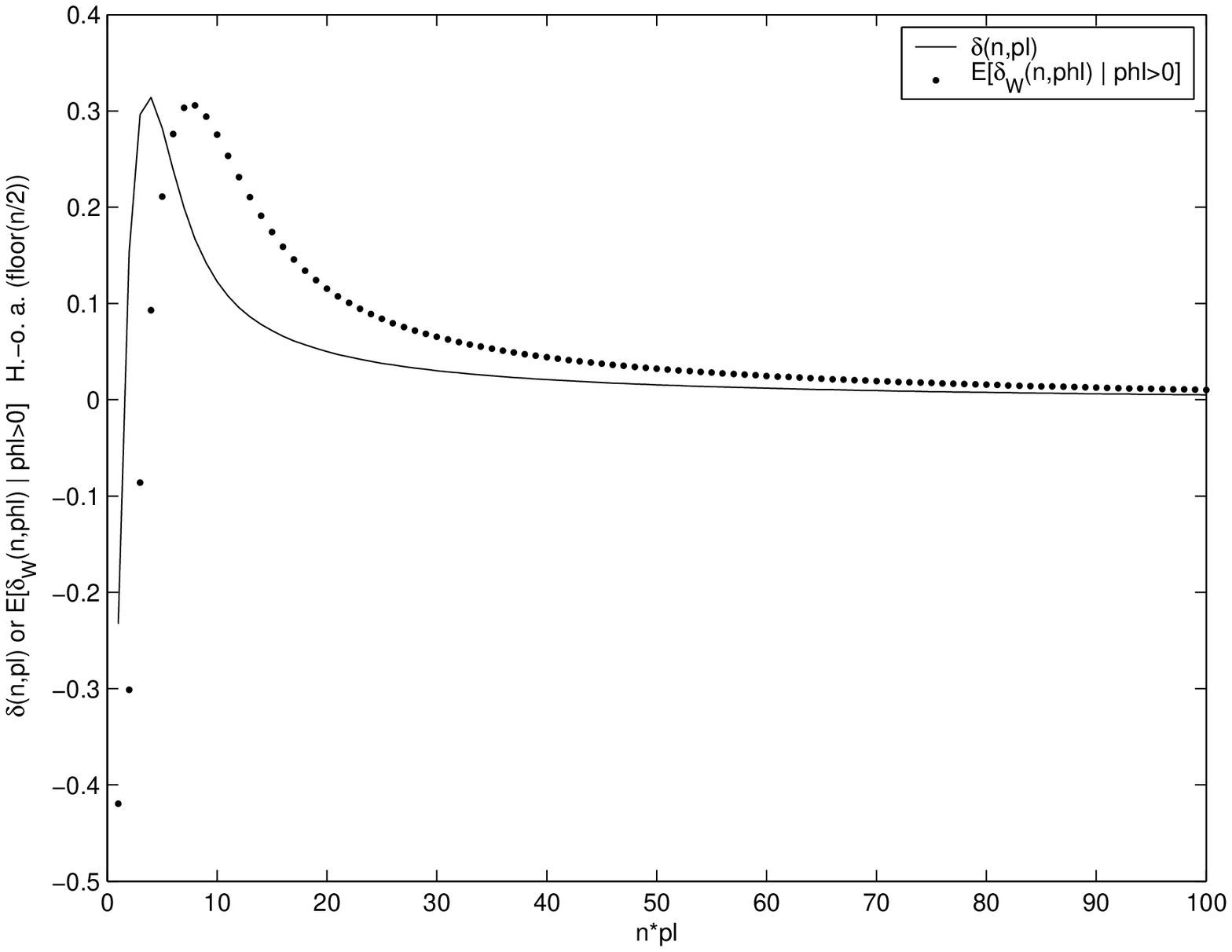,width=\textwidth}}
\vspace*{-3pt}
   \caption{$\delta_{n,p} > \overline{\delta}_{n,p}^{(\mathrm{penRho}(n/2))}$ for $np \geq 6$. \label{RP.fig.delta.hoa}}
\end{minipage} \hfill
\begin{minipage}[b]{.48\linewidth}
\centerline{\epsfig{file=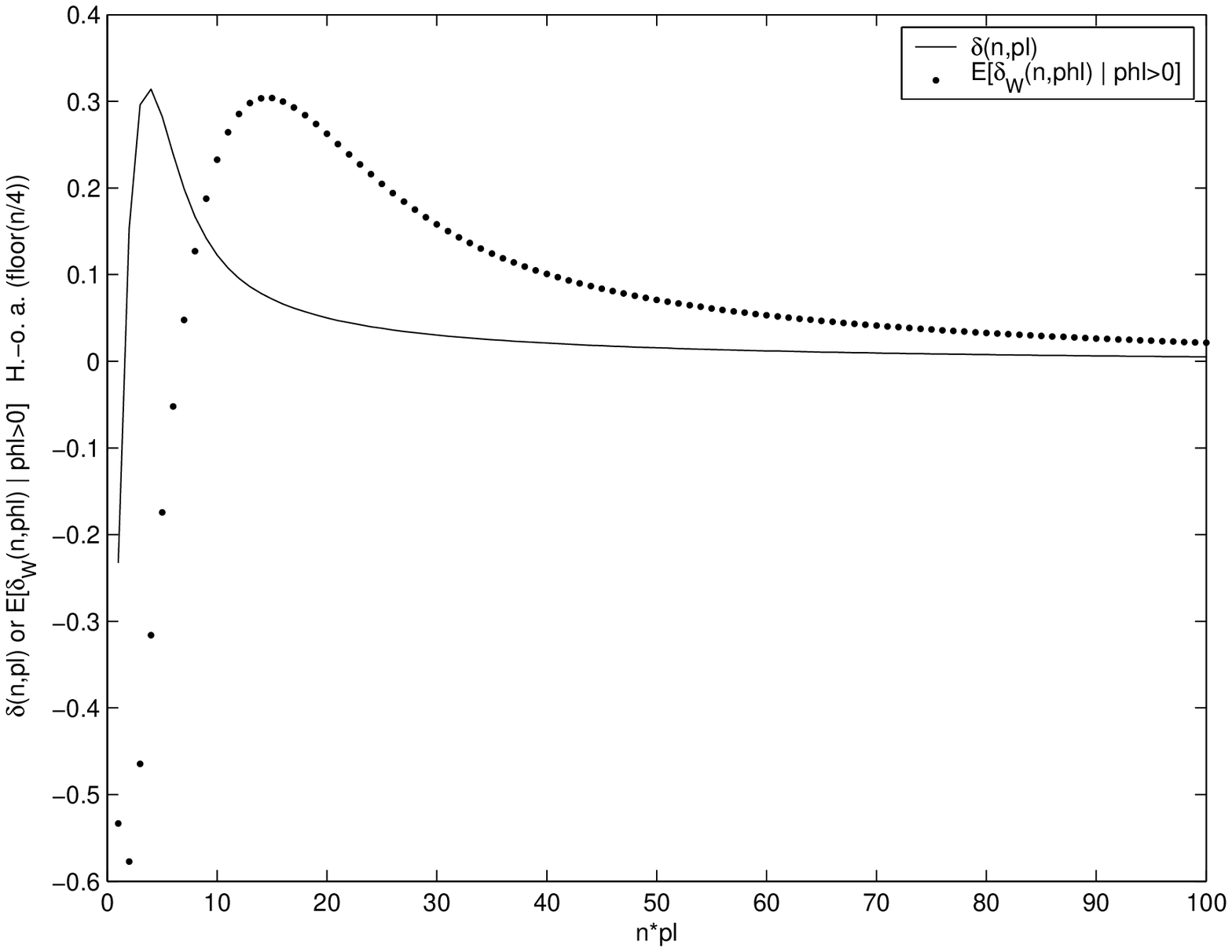,width=\textwidth}}
\vspace*{-3pt}
    \caption{$\delta_{n,p} > \overline{\delta}_{n,p}^{(\mathrm{penRho}(n/4))}$ for $np \geq 9$. \label{RP.fig.delta.hoa4}}
\end{minipage}

\medskip

\begin{minipage}[b]{.48\linewidth}
\centerline{\epsfig{file=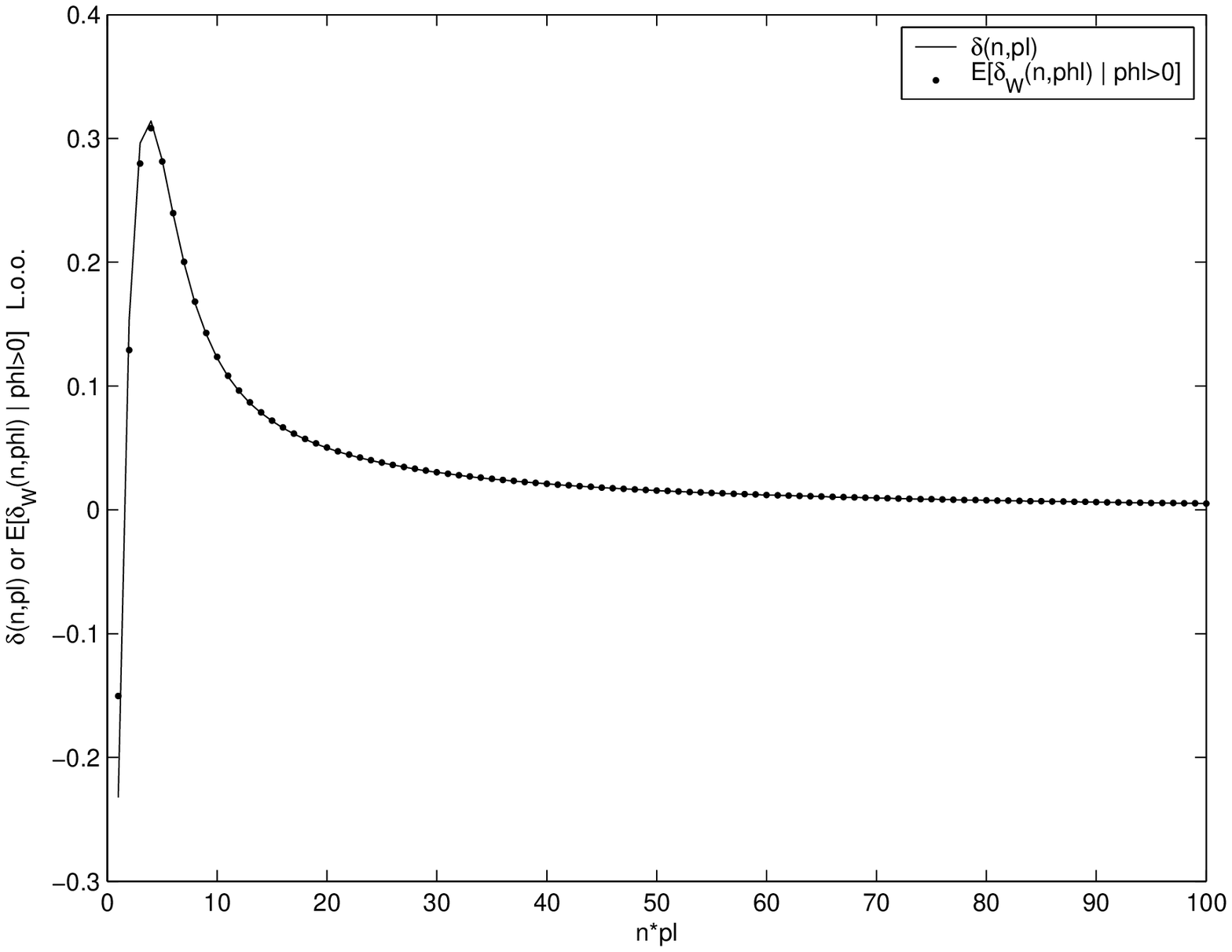,width=\textwidth}}
\vspace*{-3pt}
   \caption{$\delta_{n,p} \approx \overline{\delta}_{n,p}^{(\mathrm{penLoo})}$. \label{RP.fig.delta.loo}}
\end{minipage} \hfill
\begin{minipage}[b]{.48\linewidth}
\centerline{\epsfig{file=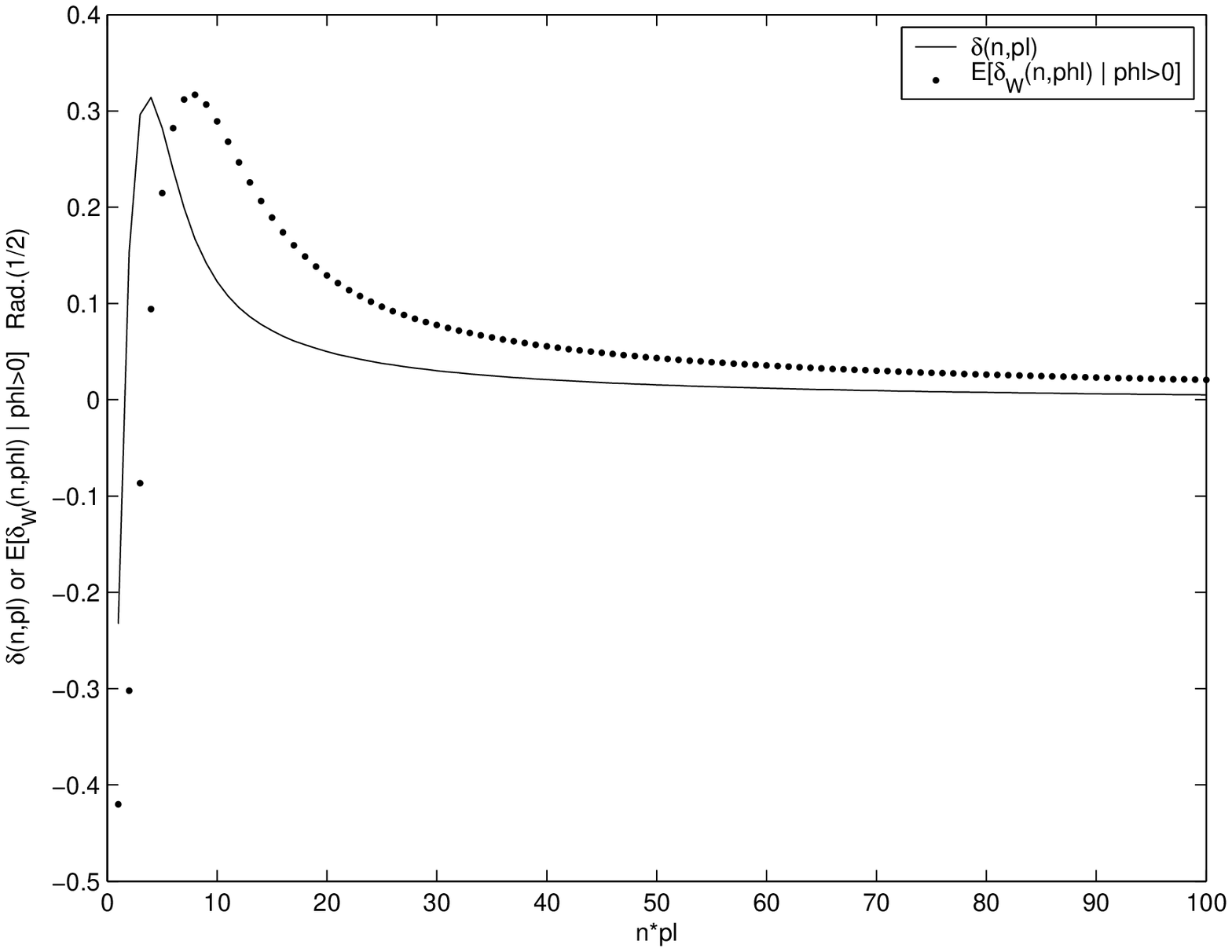,width=\textwidth}}
\vspace*{-3pt}
    \caption{$\delta_{n,p} > \overline{\delta}_{n,p}^{(\mathrm{penRad}(1/2))}$ for $np \geq 6$. \label{RP.fig.delta.rad}}
\end{minipage}
\end{figure}

It follows that Loo weights are the most accurate ones, even when $np$ is small. On the contrary, Rho ($n/2$) and Rad tend to overestimate $\penid$ since $\overline{\delta}_{n,p}^{(\mathrm{penW})} > \delta_{n,p}$ (except when $np$ is small, where the inequality is reversed).
It also seems that the bias of Rho ($q$) is a decreasing function of $q$, as illustrated by Figures~\ref{RP.fig.delta.hoa}--\ref{RP.fig.delta.hoa4}.
Finally, Efr and Poi are strongly underestimating the ideal penalty, mostly because of the $1-(n\phl)^{-1}$ term in $R_{1,W}(n,\phl)$ and $R_{2,W}(n,\phl)$.

This can be summed up as follows:
\begin{equation} \label{RP.eq.comp.penW} \mbox{penRad} \approx \mbox{penRho} > \mbox{penLoo} \approx \penid > > \mbox{penEfr} \approx \mbox{penPoi} , \end{equation} where ``$>>$'' means a comparatively large gap, but still negligible at first order.
Hence, we can expect that the Loo penalty is the most efficient, closely followed by Rad and by Rho.
However, from the non-asymptotic point of view, it turns out that smaller prediction loss is obtained by overpenalizing slightly (and sometimes strongly, see the simulations of Section~\ref{RP.sec.simus} and the discussion of Section~\ref{RP.sec.pratique.const.overpen}).
Then, the ordering of \eqref{RP.eq.comp.penW} may also be the one of the prediction performances of RP, the best performances being obtained with Rad and Rho. This is confirmed by the simulation study of Section~\ref{RP.sec.simus}.

Another interesting point is that $\overline{\delta}_{n,p}^{(\mathrm{penRho})} \propto \delta_{n,p}$ when $np$ is large enough. Then, provided that histograms with too small bins are removed from the collection, penLoo and penRho are almost equivalent, up to the choice of the factor $C$.
If a wise tuning of $C$ is possible, it remains to choose between Loo and Rho according to computational issues (see the discussion of Section~\ref{RP.sec.pratique.weight}).

\subsection{Other exchangeable weights} \label{sec.Thm1.autresW}

The oracle inequality of Theorem~\ref{RP.the.oracle_traj_non-as} is only stated for the five ``classical'' exchangeable weights of Section~\ref{sec.cadreRP.heur}.
Nevertheless, replacing the threshold 3 by some $T \geq 2$ at step~1 of Procedure~\ref{RP.def.proc.his}, the proof of Theorem~\ref{RP.the.oracle_traj_non-as} can be extended to any resampling weight vector $W$
satisfying:
\begin{enumerate}
\item $W$ is exchangeable,
\item $R_{1,W} (n,p) + R_{2,W}(n,p) \approx 2 \CWinf$ for $np$ large enough (with a non-asymptotic control on the ratio between these two quantities, as in the proof of Proposition~\ref{RP.pro.comp.Epen.Ep2}),
\item $R_{1,W} (n,p) + R_{2,W}(n,p) > (1+\epsilon) \CWinf$ for some $\epsilon>0$, as soon as $np \geq T \geq 2$ (as in Lemma~\ref{RP.le.CWinf.minor}).
\end{enumerate}

In particular, the first two conditions hold for all the exchangeable weights considered in Proposition~\ref{RP.pro.comp.Epen.Ep2}.
The third one is satisfied for most of them as soon as $T$ is large enough (see Lemma~\ref{RP.le.CWinf.minor} in Section~\ref{RP.sec.proof.tools}).

\section{Simulation study} \label{RP.sec.simus}
As an illustration of the results of Section \ref{RP.sec.main}, the prediction performances of Procedure~\ref{RP.def.proc.his} (with several resampling schemes), Mallows' $C_p$ and $V$-fold cross-validation are compared on some simulated data.

\subsection{Experimental setup} \label{VFCV.sec.simus.setup}
We consider four experiments, called S1, S2, HSd1 and HSd2.
Data are generated according to
\[ Y_i = \bayes(X_i) + \sigma(X_i) \epsilon_i \]
where $\paren{ X_i }_{1 \leq i \leq n}$ are independent with uniform distribution over $\X=[0;1]$ and $\paren{\epsilon_i}_{1 \leq i \leq n}$ are independent standard Gaussian variables independent of $\paren{ X_i }_{1 \leq i \leq n}$.
The experiments differ from the regression function $\bayes$ (smooth for S, see Figure~\ref{VFCV.fig.sin.fonc}; smooth with jumps for HS, see Figure~\ref{VFCV.fig.Hsin.fonc}), the noise type (homoscedastic for S1 and HSd1, heteroscedastic for S2 and HSd2) and the sample size $n$ (see Table~\ref{RP.Tableun}). Instances of data sets are plotted on Figures~\ref{VFCV.fig.S1.data}--\ref{VFCV.fig.HSd2.data}.

\begin{figure}
\begin{minipage}[b]{.48\linewidth}
  \centerline{\epsfig{file=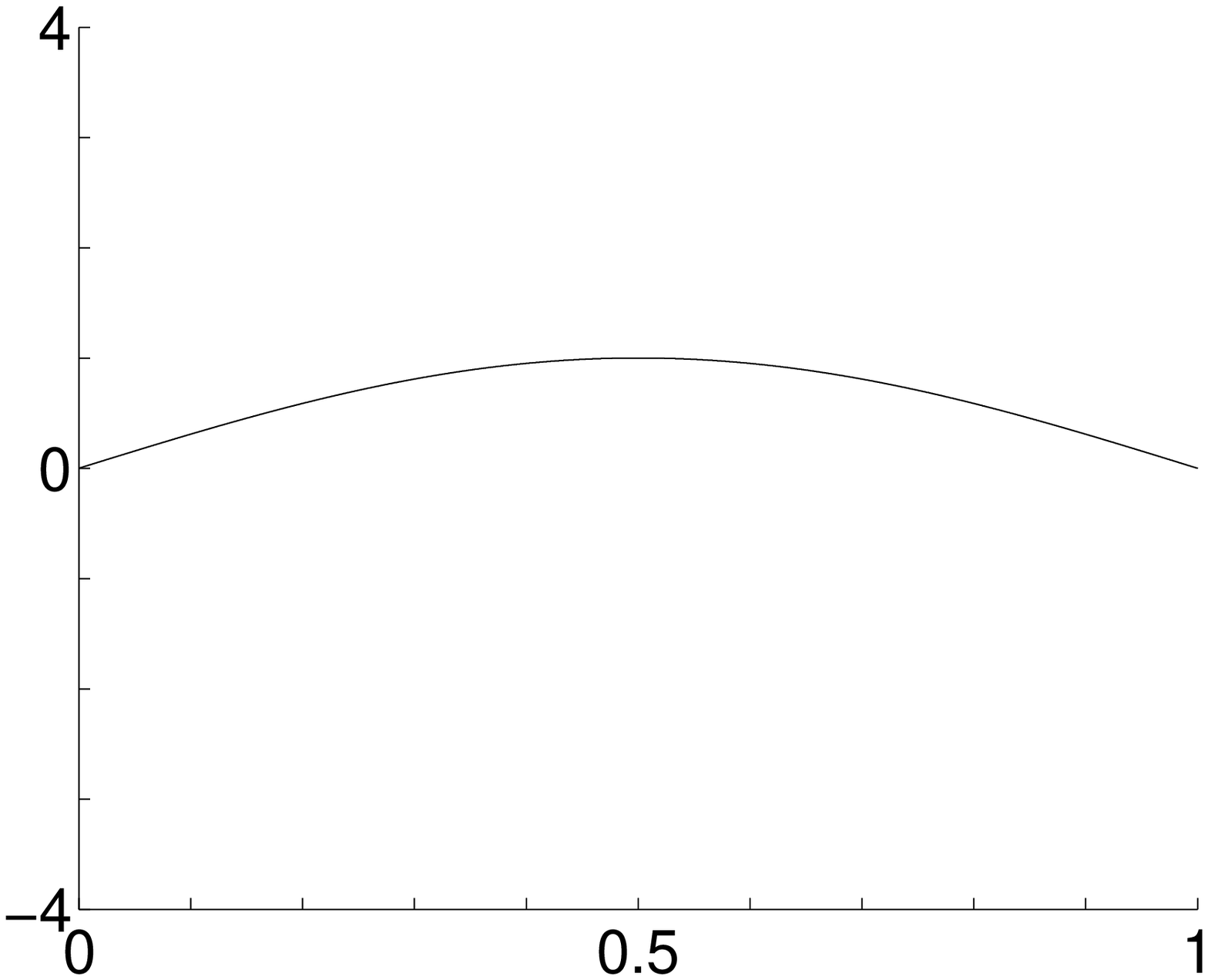,width=\textwidth}}
  \caption{$\bayes(x) = \sin(\pi x)$.\label{VFCV.fig.sin.fonc}}
\end{minipage} \hfill
 \begin{minipage}[b]{.48\linewidth}
  \centerline{\epsfig{file=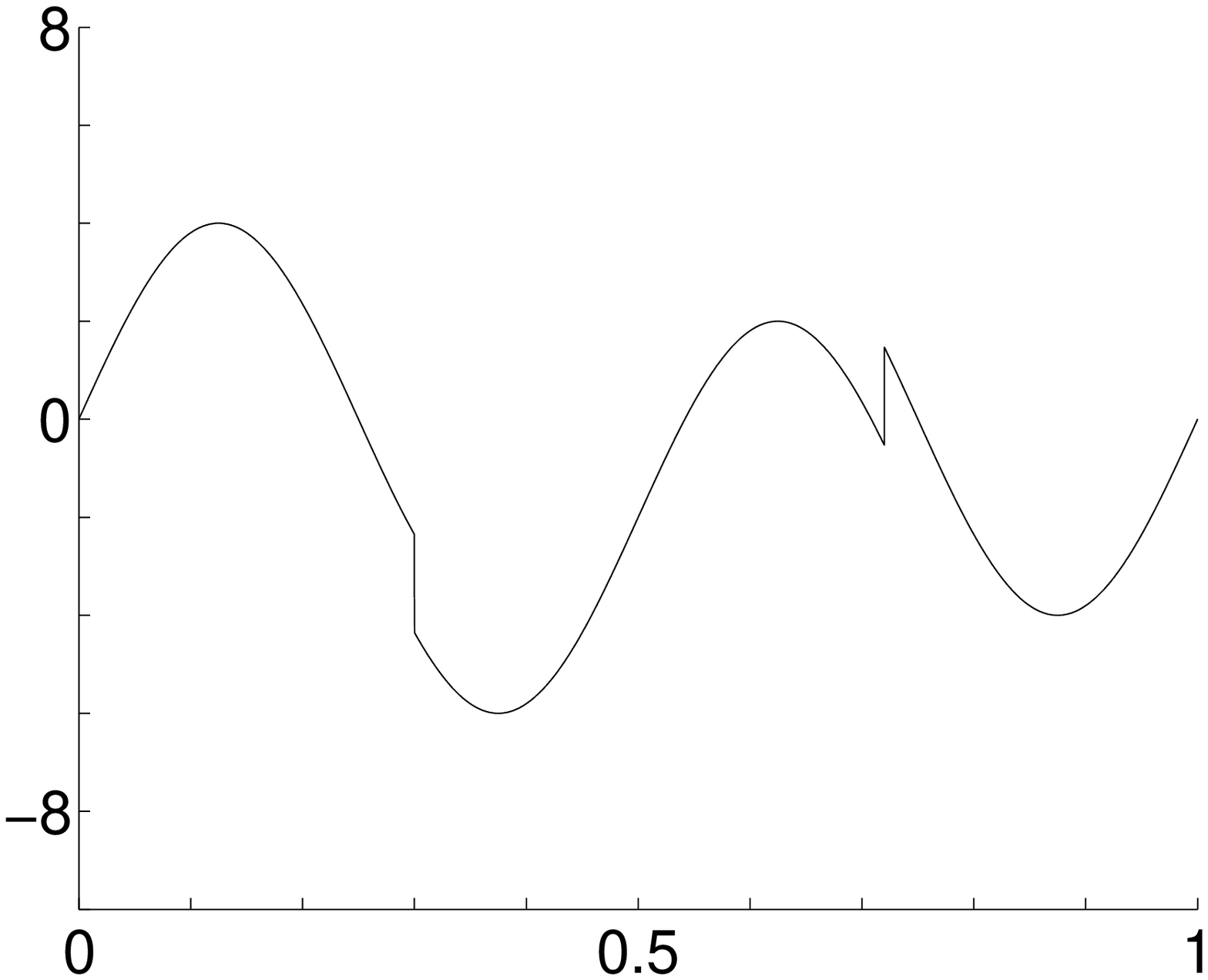,width=\textwidth}}
  \caption{$\bayes(x) = \mathrm{HeaviSine}(x)$ (see \cite{Don_Joh:1995}).\label{VFCV.fig.Hsin.fonc}}
\end{minipage}

\bigskip

\begin{minipage}[b]{.48\linewidth}
\centerline{\epsfig{file=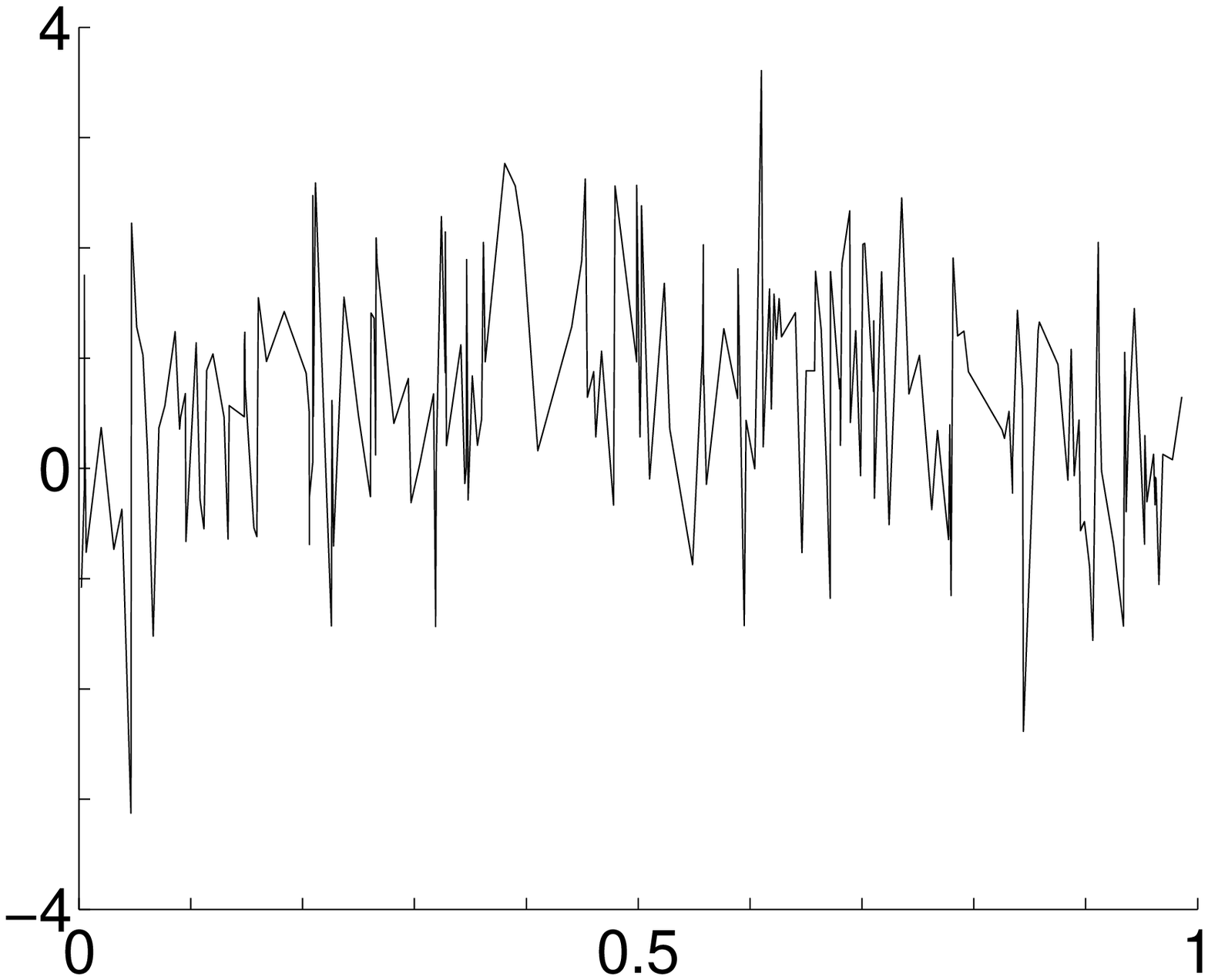,width=\textwidth}}
  \caption{S1: $\bayes(x)=\sin(\pi x)$, $\sigma \equiv 1$, $n=200$.}
  \label{VFCV.fig.S1.data}
\end{minipage} \hfill
 \begin{minipage}[b]{.48\linewidth}
  \centerline{\epsfig{file=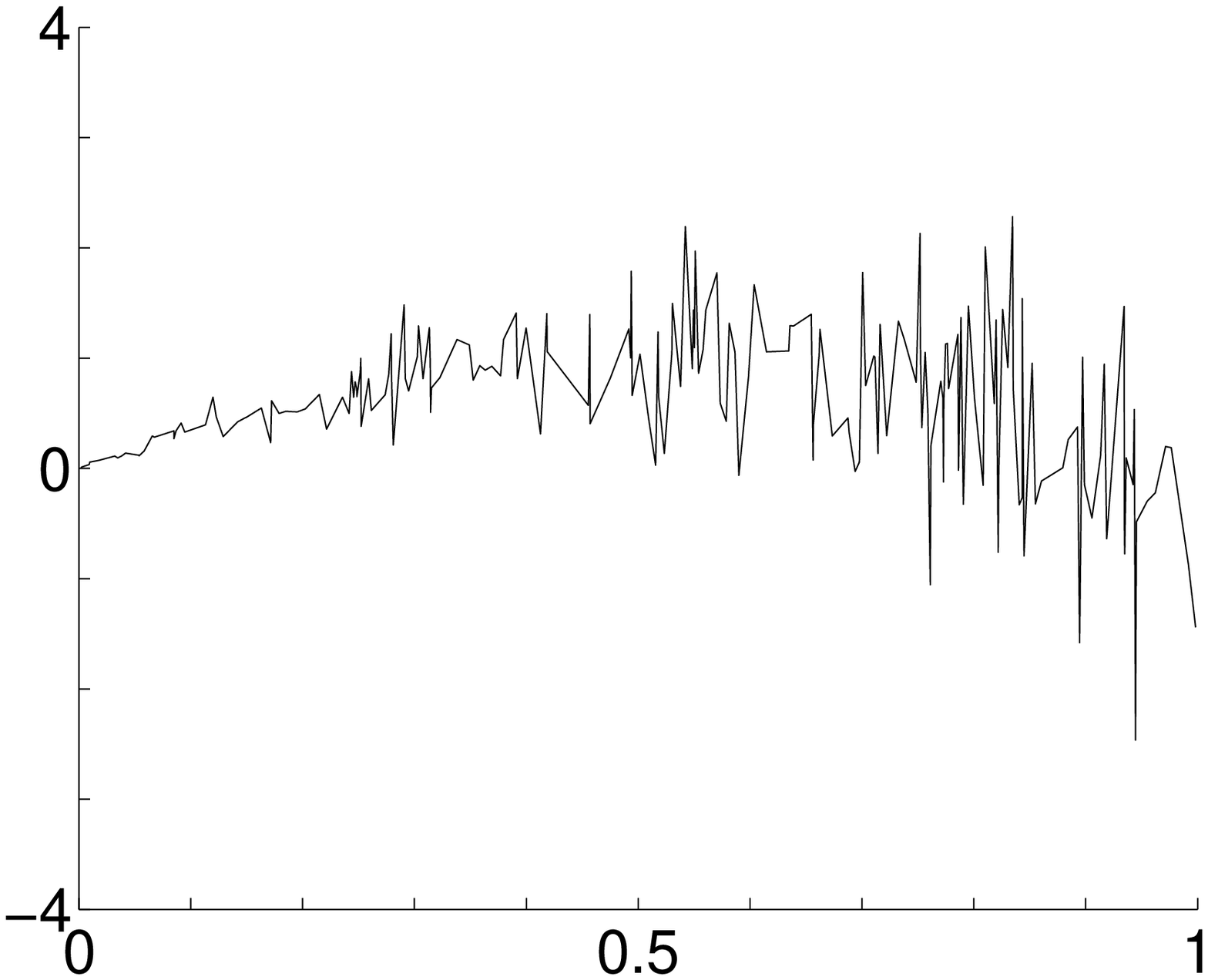,width=\textwidth}}
  \caption{S2: $\bayes(x)=\sin(\pi x)$, $\sigma(x)=x$, $n=200$.}
  \label{VFCV.fig.S2.data}
\end{minipage}

\bigskip

\begin{minipage}[b]{.48\linewidth}
  \centerline{\epsfig{file=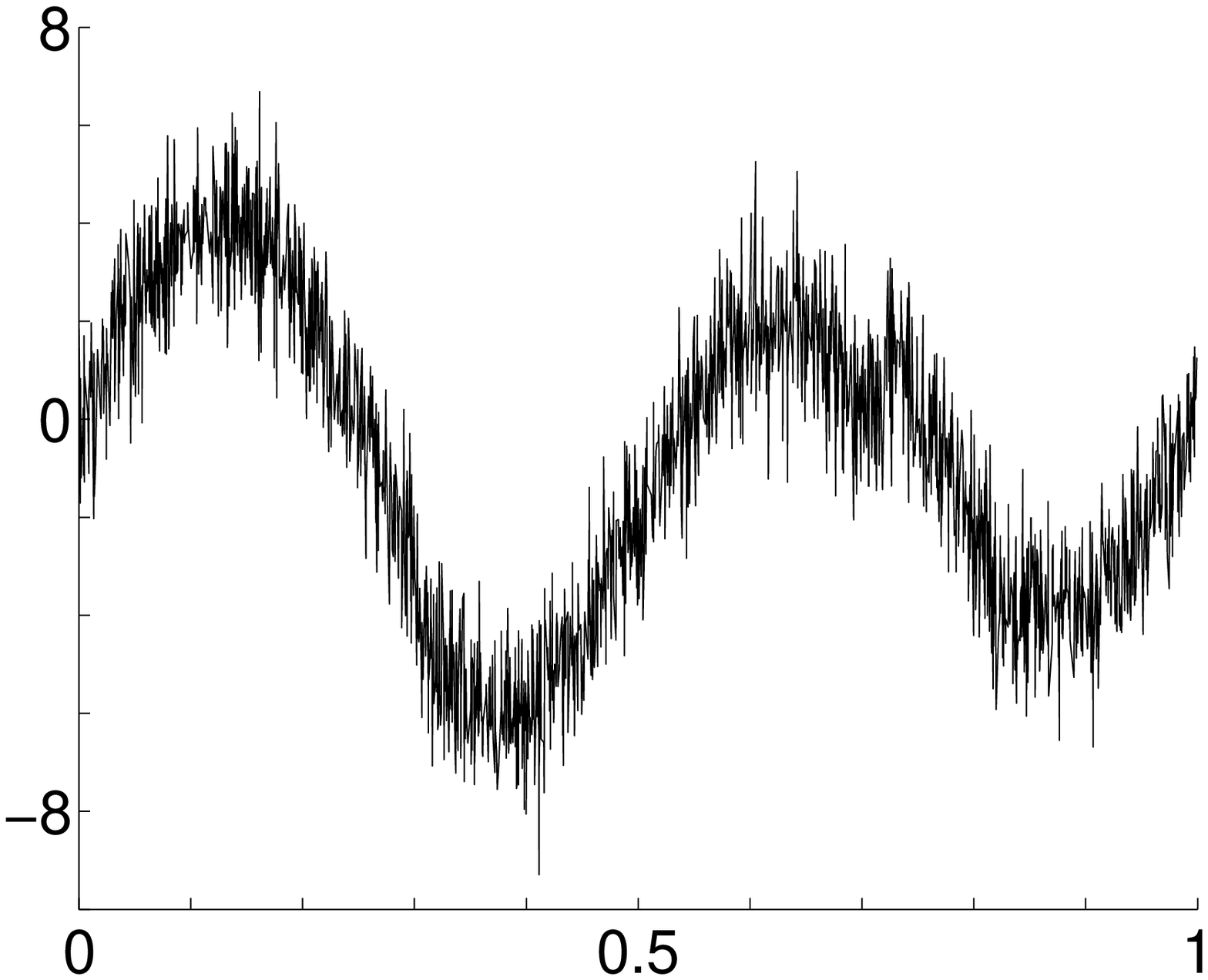,width=\textwidth}}
  \caption{HSd1: HeaviSine, $\sigma\equiv 1$, $n=2048$.}
  \label{VFCV.fig.HSd1.data}
\end{minipage} \hfill
 \begin{minipage}[b]{.48\linewidth}
  \centerline{\epsfig{file=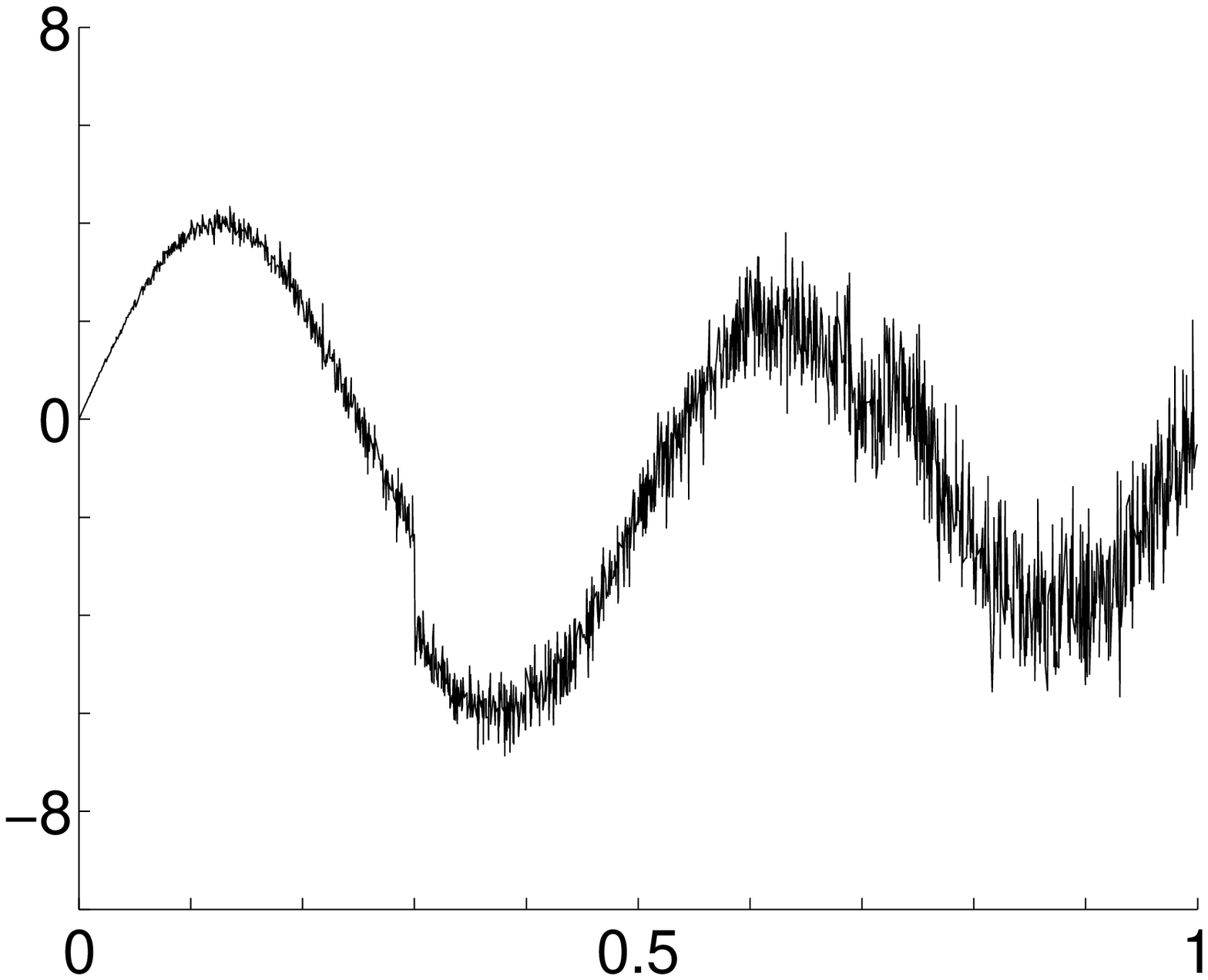,width=\textwidth}}
  \caption{HSd2: HeaviSine, $\sigma(x)=x$, $n=2048$.}
  \label{VFCV.fig.HSd2.data}
\end{minipage}
\end{figure}

The collections of histogram models also differ according to the experiments.
Define \begin{gather*}
\forall k, k_1, k_2 \in \N \backslash \set{0}, \quad \paren{\Il}_{\lambda \in \Lambda_k} = \paren{ \left[ \frac{j}{k}; \frac{j+1}{k} \right) }_{0 \leq j \leq k-1} \mbox{ and } \\
\paren{\Il}_{\lambda \in \Lambda_{(k_1,k_2)}} = \paren{ \left[ \frac{j}{2k_1}; \frac{j+1}{2k_1} \right) }_{0 \leq j \leq k_1-1} \cup \paren{ \left[ \frac{1}{2} + \frac{j}{2k_2}; \frac{1}{2} + \frac{j+1}{2k_2} \right) }_{0 \leq j \leq k_2-1} .
\end{gather*}
For every $m \in \paren{\N \backslash \set{0}} \cup \paren{\N \backslash \set{0}}^2$, let $S_m$ be the histogram model associated with the partition $\paren{\Il}_{\lamm}$.
Then, for each experiment, the collection of models is $\paren{S_m}_{\M_n}$ with different index sets $\M_n$:
\makeatletter
\setlength\leftmargini   {28\p@}
\makeatother
\begin{enumerate}
\item[S1] regular histograms with $1 \leq D \leq n \paren{\ln(n)}^{-1}$ pieces, that is \[ \M_n = \set{ 1, \ldots, \left\lfloor \frac{n}{\ln(n)} \right\rfloor } . \]
\item[S2] histograms regular on $\croch{0;1/2}$ (resp. on $\croch{1/2;1}$), with $D_1$ (resp. $D_2$) pieces, $1 \leq D_1,D_2 \leq n \paren{2 \ln(n)}^{-1}$. The model of constant functions is added to $\M_n$, that is \[ \M_n = \set{1} \cup \set{ 1, \ldots, \left\lfloor \frac{n}{2\ln(n)} \right\rfloor }^2 . \]
\item[HSd1] dyadic regular histograms with $2^k$ pieces, $0 \leq k \leq \ln_2(n) - 1$, that is \[ \M_n = \set{ 2^k \telque 0 \leq k \leq \ln_2(n)-1 } .\]
\item[HSd2] dyadic histograms regular on $\croch{0;1/2}$ (resp. on $\croch{1/2;1}$) with bin sizes $2^{-k_1}$ (resp. $2^{-k_2}$), $0 \leq k_1,k_2 \leq \ln_2(n) - 2$ (dyadic version of S2). The model of constant functions is added to $\M_n$, that is \[ \M_n = \set{1} \cup \set{ 2^k \telque 0 \leq k \leq \ln_2(n)-2 }^2 .\]
\end{enumerate}
Note that the collections of models used in experiments S2 and HSd2 can adapt to $\bayes$ and $\sigma(\cdot)$.
Therefore, the oracle model is generally quite efficient so that the model selection problem is more challenging.


The following procedures\footnote{The code used for computing resampling penalties is available on the author's webpage at
\url{http://www.di.ens.fr/~arlot/index.htm}.} are compared:
\makeatletter
\setlength\leftmargini   {42\p@}
\makeatother
\begin{enumerate}
\item[Mal] Mallows' $C_p$ penalty: $\pen(m) = 2 \widehat{\sigma}^2 D_m n^{-1}$ where $\widehat{\sigma}^2$ is the classical variance estimator defined as
\begin{equation} \label{RP.def.var-estim} \widehat{\sigma}^2 = \frac{d^2 \paren{ Y\nupl , S_{\left\lfloor n/2 \right\rfloor} }}{n - \left\lfloor n/2 \right\rfloor} ,\end{equation}
where $Y\nupl = (Y_i)_{1 \leq i \leq n} \in \R^n$, $S_{\left\lfloor n/2 \right\rfloor}$ is any model of dimension $ \left\lfloor n/2 \right\rfloor$ (only assumed to have a bias negligible in front of $\sigma^2$) and $d$ is the Euclidean distance on $\R^n$.
The non-asymptotic validity of this model selection procedure in homoscedastic regression has been assessed by Baraud \cite{Bar:2000}.
\item[$\E\croch{\penid}$] Expectation of the ideal penalty: $\pen(m) = \E\croch{\penid(m)}$, which witnesses what is a good performance in each experiment.
\item[VFCV] $V$-fold cross-validation, with $V \in \{ 2,5,10,20 \}$ (defined as in \cite{Arl:2008a}).
\item[LOO] Leave-one-out (that is VFCV with $V=n$).
\item[penEfr] Efron ($n$) penalty \eqref{RP.def.pen.his} with $C= \CWinf =1$.
\item[penRad] Rademacher ($1/2$) penalty \eqref{RP.def.pen.his} with $C= \CWinf =1$.
\item[penRho] Random hold-out ($n/2$) penalty \eqref{RP.def.pen.his} with $C= \CWinf =1$.
\item[penLoo] Leave-one-out penalty \eqref{RP.def.pen.his} with $C= \CWinf =n-1$.
\end{enumerate}
\makeatletter
\setlength\leftmargini   {22\p@}
\makeatother
For each of these, the same penalties multiplied by $5/4$ are also considered (and they are denoted by a $+$ symbol added after the shortened names). This intends to test for overpenalization (the choice of the factor $5/4$ being arbitrary and certainly not optimal, see Section~\ref{RP.sec.pratique.const.overpen}).

In each experiment, for each simulated data set, first the models with $2$ data points or less in one piece of their associated partition are removed.
Then, the least-squares estimators $\ERM_m$ are computed for each $\mMh_n$.
Finally, $\mh \in \Mh_n$ is selected using each procedure and its true excess loss $\perte{\ERM_{\mh}}$ is computed as well as the excess loss of the oracle $\inf_{\mM_n} \perte{\ERM_m}$.
$N=1000$ data sets are simulated, thanks to which the model selection performance of each procedure is estimated through the two following benchmarks:
\begin{equation*} 
C_{\mathrm{or}} = \frac{ \E\croch{ \perte{\ERM_{\mh}} }} {\E\croch{ \inf_{\mM_n} \perte{\ERM_m} }} \qquad C_{\mathrm{path-or}} = \E\croch{\frac{ \perte{\ERM_{\mh} }} {\inf_{\mM_n} \perte{\ERM_m} } }
\end{equation*}
Basically, $C_{\mathrm{or}}$ is the constant that should appear in an oracle inequality like \eqref{RP.eq.oracle_class_non-as}, and $C_{\mathrm{path-or}}$ corresponds to a pathwise oracle inequality like \eqref{RP.eq.oracle_traj_non-as}.
Since $C_{\mathrm{or}}$ and $C_{\mathrm{path-or}}$ approximatively give the same rankings between procedures, Table~\ref{RP.Tableun} only reports $C_{\mathrm{or}}$; the values of $C_{\mathrm{path-or}}$ are reported in \cite{Arl:2008b:app}.

\begin{table}[t]
\caption{Accuracy indices $C_{\mathrm{or}}$ for each procedure in four experiments,
$\pm$ a rough estimate of uncertainty of the value reported (that is the empirical
standard deviation divided by $\sqrt{N}$; $N=1000$). In each column, the more accurate
procedures (taking the uncertainty into account) are bolded
\label{RP.Tableun}}
\begin{tabular}
{p{0.16\textwidth}@{\hspace{0.025\textwidth}}p{0.16\textwidth}@{\hspace{0.025\textwidth}}p{0.16\textwidth}@{\hspace{0.025\textwidth}}p{0.17\textwidth}@{\hspace{0.025\textwidth}}p{0.19\textwidth}}
\hline\noalign{\smallskip}
Experiment & S1 & S2 & HSd1 & HSd2 \\ \noalign{\smallskip} \hline \noalign{\smallskip}
 $\bayes$ & $\sin(\pi \cdot)$ & $\sin(\pi \cdot)$ & HeaviSine & HeaviSine \\
 $\sigma(x)$ & 1 & $x$ & 1 & $x$ \\
 $n$ (sample size) & 200 & 200 & 2048 & 2048 \\
 $\M_n$ & regular & 2 bin sizes & dyadic, regular & dyadic, 2 bin sizes \\ \noalign{\smallskip} \hline \noalign{\smallskip}
 $\E\croch{\penid}$ & $ 1.919 \pm 0.03 $ & $ 2.296 \pm 0.05 $ & $ 1.028 \pm 0.004 $ & $ \meil{ 1.102 \pm 0.004 } $ \\
$\E\croch{\penid}$+ & $ \meil{ 1.792 \pm 0.03 } $ & $ \meil{ 2.028 \pm 0.04 } $ & $ \meil{ 1.003 \pm 0.003 } $ & $ \meil{ 1.089 \pm 0.004 } $ \\
Mal & $ 1.928 \pm 0.04 $ & $ 3.687 \pm 0.07 $ & $ 1.015 \pm 0.003 $ & $ 1.373 \pm 0.010 $ \\
Mal+ & $ \meil{ 1.800 \pm 0.03 } $ & $ 3.173 \pm 0.07 $ & $ \meil{ 1.002 \pm 0.003 } $ & $ 1.411 \pm 0.008 $ \\ \noalign{\smallskip} \hline \noalign{\smallskip}
 2-FCV & $ 2.078 \pm 0.04 $ & $ 2.542 \pm 0.05 $ & $ \meil{ 1.002 \pm 0.003 } $ & $ 1.184 \pm 0.004 $ \\
5-FCV & $ 2.137 \pm 0.04 $ & $ 2.582 \pm 0.06 $ & $ 1.014 \pm 0.003 $ & $ 1.115 \pm 0.005 $ \\
10-FCV & $ 2.097 \pm 0.04 $ & $ 2.603 \pm 0.06 $ & $ 1.021 \pm 0.003 $ & $ 1.109 \pm 0.004 $ \\
20-FCV & $ 2.088 \pm 0.04 $ & $ 2.578 \pm 0.06 $ & $ 1.029 \pm 0.004 $ & $ 1.105 \pm 0.004 $ \\
LOO & $ 2.077 \pm 0.04 $ & $ 2.593 \pm 0.06 $ & $ 1.034 \pm 0.004 $ & $ 1.105 \pm 0.004 $ \\ \noalign{\smallskip} \hline \noalign{\smallskip}
 penRad & $ 1.973 \pm 0.04 $ & $ 2.485 \pm 0.06 $ & $ 1.018 \pm 0.003 $ & $ \meil{ 1.102 \pm 0.004 } $ \\
penRho & $ 1.982 \pm 0.04 $ & $ 2.502 \pm 0.06 $ & $ 1.018 \pm 0.003 $ & $ \meil{ 1.103 \pm 0.004 } $ \\
penLoo & $ 2.080 \pm 0.04 $ & $ 2.593 \pm 0.06 $ & $ 1.034 \pm 0.004 $ & $ 1.105 \pm 0.004 $ \\
penEfr & $ 2.597 \pm 0.07 $ & $ 3.152 \pm 0.07 $ & $ 1.067 \pm 0.005 $ & $ 1.114 \pm 0.005 $ \\ \noalign{\smallskip} \hline \noalign{\smallskip}
 penRad+ & $ \meil{ 1.799 \pm 0.03 } $ & $ \meil{ 2.137 \pm 0.05 } $ & $ \meil{ 1.002 \pm 0.003 } $ & $ \meil{ 1.095 \pm 0.004 } $ \\
penRho+ & $ \meil{ 1.798 \pm 0.03 } $ & $ \meil{ 2.142 \pm 0.05 } $ & $ \meil{ 1.002 \pm 0.003 } $ & $ \meil{ 1.095 \pm 0.004 } $ \\
penLoo+ & $ \meil{ 1.844 \pm 0.03 } $ & $ \meil{ 2.215 \pm 0.05 } $ & $ \meil{ 1.004 \pm 0.003 } $ & $ \meil{ 1.096 \pm 0.004 } $ \\
penEfr+ & $ 2.016 \pm 0.05 $ & $ 2.605 \pm 0.06 $ & $ 1.011 \pm 0.003 $ & $ \meil{ 1.097 \pm 0.004 } $ 
 \\ \noalign{\smallskip} \hline
\end{tabular}
\end{table}

\subsection{Results and comments}
First, the above experiments show the interest of both Resampling Penalization (RP) and VFCV in several difficult frameworks, with relatively small sample sizes. Although RP and VFCV cannot compete with simple procedures such as Mallows' $C_p$ from the computational point of view, they are much more efficient when the noise is heteroscedastic (S2 and HSd2). In these difficult frameworks, the prediction performances of RP and VFCV are comparable to those of $\E\croch{\penid}$.
Note that in HSd2, penRad and penRho give smaller losses than any penalty proportional to the dimension of the models (see Section~\ref{RP.sec.calc.penvslin}).
Moreover, penRad and penRho perform slighlty worse than Mallows' $C_p$ for the easiest problems (S1 and HSd1), which can be interpretated as the unavoidable price for robustness.

Second, in the four experiments, the best procedures always are the overpenalizing ones: many of them even beat the perfectly unbiased $\E\croch{\penid}$, showing the crucial need to overpenalize.
This phenomenon disappears for small $\sigma$ and large $n$ \cite[Experiments S0.1 and S1000]{Arl:2008b:app}, hence it is certainly due to the small signal-to-noise ratio.
We would like to insist on the importance of the overpenalization phenomenon, which is seldom mentioned in theoretical papers because it vanishes in the asymptotic framework, and it is quite hard to find from theoretical results.


Let us now compare RP and VFCV. According to the four experiments of Table~\ref{RP.Tableun}, RP with Rad or Rho resampling schemes clearly outperforms VFCV for any $V$, even without overpenalizing.
The only exception to this is HSd1 where $2$-fold cross-validation yields a particularly good model selection performance.

This can be interpretated thanks to the non-asymptotic study of the performance of $V$-fold cross-validation provided in \cite{Arl:2008a}.
In short, VFCV overpenalizes within a factor $1 + 1/(2(V-1))$, while the $V$-fold criterion has a variance decreasing with $V$.

Then, when overpenalization is necessary (for instance in S1, S2 or HSd1), small values of $V$ can outperform the leave-one-out ($V=n$).
Nevertheless, RP with the right overpenalization level $C/\CWinf$ leads to a smaller prediction loss than VFCV, because RP provides a less variable model selection criterion than VFCV.
The reason why penRad and penRho also perform slightly better without overpenalization is that they naturally overpenalize when $C= \CWinf = 1$ (see Section~\ref{RP.sec.calc.delta}).


Let us now consider the model selection performance of RP with several exchangeable resampling schemes. The two best ones are Rad and Rho in the four experiments, with or without overpenalization. Then, Loo performs slightly worse (but not always significantly) and Efr much worse.
Looking carefully at the values of the penalties, it appears that Rad and Rho slightly overpenalize, Loo is exactly at the right level, and Efr underpenalizes (as well as Poi, which has performances quite similar to the ones of Efr, see \cite{Arl:2008b:app}). Note that this comparison can also be derived from theoretical computations (see Section~\ref{RP.sec.calc.delta}).
Since overpenalization is benefic in the four experiments of Table~\ref{RP.Tableun}, this explains why penRad and penRho slightly outperform penLoo. In the case of Efron's boostrap penalty, underpenalizing implies overfitting which explains the comparatively bad performances reported in Table~\ref{RP.Tableun}.


We conclude this section with remarks concerning some particular points of the simulation study.
\begin{itemize}
\item On the same data sets, Mallows' $C_p$ and its overpenalized version Mal+ were performed with the true mean variance $\E\croch{\sigma^2(X)}$ instead of $\widehat{\sigma}^2$ (which would not be possible on a real data set). It yielded worse model selection performance for all experiments but S2, in which $C_{\mathrm{or}}(\mathrm{Mal})=2.657 \pm 0.06$ and $C_{\mathrm{or}}(\mathrm{Mal}+)=2.437 \pm 0.05$.
Therefore, overpenalization is crucial in experiment S2, more than the shape\footnote{The shape of a penalty is defined as the way $\pen(m)$ depends on $m$ up to a linear transformation.} of the penalty itself.
Moreover, the overpenalization level being fixed, resampling penalties remain significantly better than Mallows' $C_p$.
Hence, the performances of Mallows' $C_p$ in Table~\ref{RP.Tableun} are not only due to a bad estimation of the mean noise-level (see also Section~\ref{RP.sec.comparaison}).
\item Eight additional experiments are reported in \cite{Arl:2008b:app}, showing similar results with various $n$, $\sigma$ and $\bayes$ (although the assumptions of Theorem~\ref{RP.the.oracle_traj_non-as} are not always satisfied).
\item Resampling penalties with a $V$-fold subsampling scheme have also been studied in \cite[Section~4]{Arl:2008a} on the same simulated data: exchangeable resampling schemes always give better model selection performance than non-exchangeable ones (significantly when $V$ is small), except for Efr and Poi which tend to underestimate the ideal penalty.
\end{itemize}

\section{Practical implementation}\label{RP.sec.pratique}

This section tackles three main issues for using Procedure~\ref{RP.def.proc.his} in practice: how to compute the resampling penalty \eqref{RP.def.pen.his}? how to choose the weights $W$? how to choose the constant $C$?

\subsection{Computational cost}\label{RP.sec.pratique.cost}

An exact computation of resampling penalties with exchangeable weights (without using formula \eqref{VFCV.eq.pen.Wech} for histograms) would be either impossible or computationally expensive. We suggest two possible ways to fix this problem.

First, one can use a classical Monte-Carlo approximation, that is draw a small number $B$ of independent weight vectors instead of considering each element of the support of $\loi(W)$. Practical Monte-Carlo methods for the boostrap are proposed for instance by Hall \cite[Appendix~II]{Hal:1992}.
Moreover, a non-asymptotic estimation of the accuracy of Monte-Carlo approximation can be obtained via McDiarmid's inequality (see Arlot, Blanchard and Roquain \cite[Proposition~2.7]{Arl_Bla_Roq:2008:RC} for a precise result using the same idea in another framework). This would provide a practical way of quantifying what is lost by making a Monte-Carlo approximation, and choose $B$ consequently (at least for Rad, Rho and Loo
weights).

Second, it is possible to use non-exchangeable weight vectors $W$ such that the cardinality of the support of $\loi(W)$ is much smaller than $n$. A case-example is {\em $V$-fold subsampling}: given a partition $\paren{B_j}_{1 \leq j \leq V}$ of $\set{1, \ldots, n}$ and $J$ a uniform random variable over $\set{1, \ldots, V}$ independent of the data, we define
\[ \forall i \in \set{1, \ldots, n}, \quad W_i = \frac{V}{V-1} \un_{i \notin B_J} . \]
The resulting resampling penalties ---called $V$-fold penalties--- have been introduced and studied in \cite{Arl:2008a}. They are computationally similar to VFCV while being more flexible, since the overpenalization factor is decoupled from the choice of $V$; hence, like resampling penalties, $V$-fold penalties select an estimator with smaller prediction loss than the one selected by VFCV.

Both Monte-Carlo approximation of RP and $V$-fold penalization have been tested on the simulated data of Section~\ref{RP.sec.simus}. The detailed results are given in \cite{Arl:2008b:app}.

\subsection{Choice of the weights}\label{RP.sec.pratique.weight}

The influence of the weights has been investigated from the theoretical
point of view in Section~\ref{RP.sec.calc.delta} with focus on
second-order terms in expectation. However, deviations of $\pen(m)$
around its expectation are likely to depend on the weight vector $W$
since the upper bound in \eqref{RP.eq.conc.penRP} may not be tight. The
simulation study of Section~\ref{RP.sec.simus} allows to take into
account both phenomena in the comparison between the resampling
weights.

In terms of model selection efficiency, Table~\ref{RP.Tableun} shows that the best weights (for accuracy of prediction and for the variability\footnote{The variability of the accuracy is more an indicator of the {\em stability} of the performance of RP than of the variance of the resampling penalty. However, it remains an interesting measure, since a procedure performing always equally well can be preferred to a procedure with better mean efficiency but poor performances on a small probability event.} of this accuracy) are Rho and Rad, whereas Loo perform slightly worse.
On the contrary, from both accuracy and variability points of view, Efron's bootstrap weights perform worse than Rho, Rad and Loo, mainly because they lead to underpenalization.

Note however that this comparison strongly depends on the precise definition\footnote{However, it is quite unclear how to change $\CWinf$ in order to optimize each penalty in the general case.
This is why $\CWinf$ has been chosen as ``simple'' as possible in Table~\ref{RP.TableR2}.} of $\CWinf$, which makes all penalties unbiased at first order but possibly under or over-penalizing at second order. Then, different prediction performances may be observed on data which do not require overpenalization.
Nevertheless, the computations of Section~\ref{RP.sec.calc.delta} show that Efron's bootstrap weights have a real drawback which cannot be fixed only by changing $\CWinf$.


When computing the penalties exactly, Loo weights are the only computationally tractable ones, while being almost as accurate as Rho and Rad. Hence, we suggest their use, enlarging the constant $C$ when needed (see Section~\ref{RP.sec.pratique.const.overpen} on overpenalization).

However, computing $n$ empirical risk minimizers (or the outputs of computationally more expensive algorithms) for each model is not always possible. In such a case, one should avoid using the Leave-one-out with a Monte-Carlo approximation, which would give a large importance to a small number of data points. Rho or Rad weights are much safer in this situation. Alternatively, one may consider the use of $V$-fold penalties \cite{Arl:2008a} as a good alternative when the computational power is limited.


Let us emphasize that this analysis and the subsequent advices should be considered with caution. First, the deviations of resampling penalties around their expectations should be understood much better, because they can be comparable or even larger than the second-order terms in expectations.
Second, the optimal choice of $V$ for $V$-fold cross-validation is known to be different between least-squares regression and binary classification \cite[Section~2.3]{Arl:2008a}.
Such differences are expected to arise for choosing between exchangeable resampling weights.

Remark that the bias of the bootstrap penalty has already been noticed
by Efron \cite{Efr:1983,Efr:1986} who proposed several ways to correct
it, including a double bootstrap procedure and the .632 bootstrap. The
novelty of the approach of this paper is to propose the use of other
exchangeable resampling schemes instead of the boostrap so that the
bias of resampling penalties no longer has to be
corrected.

\subsection{Choice of the constant C} \label{RP.sec.pratique.const}

\subsubsection{Optimal constant for bias} \label{RP.sec.pratique.const.opt}

From the asymptotic point of view, the optimal $C = C^{\star}$ for prediction is generally the one for which $\pen$ estimates the ideal penalty $\penid$ unbiasedly (at least for collections of models of polynomial size).
This is how $\CWinf$ is defined in the histogram framework and Theorem~\ref{RP.the.oracle_traj_non-as} implies that $C=\CWinf$ is asymptotically optimal for prediction. Hence\footnote{See the proof of Theorem~1 in \cite{Arl:2008a} to prove that asymptotic optimality requires $C^{\star} / \CWinf \xrightarrow[n \rightarrow \infty]{} 1 $ as soon as there are enough models close to the oracle.}, $C^{\star}$ is asymptotically equivalent to $\CWinf$.

As showed by Arlot and Massart \cite{Arl_Mas:2008}, $C^{\star}$ can also be estimated directly from data for general penalties, in particular for RP. Hence, the knowledge of $\CWinf$ is not necessary, which can be useful in the general prediction framework (see Section~\ref{RP.sec.discussion.classif}).

\subsubsection{Overpenalization} \label{RP.sec.pratique.const.overpen}

A careful look at the proof of Theorem~\ref{RP.the.oracle_traj_non-as} shows that a similar oracle inequality holds for any $C > 4 \CWinf /5$, the leading constant remaining close to one when $C \sim \CWinf$ {\em asymptotically}. In other words, when the sample size $n$ is small, the optimal constant $C^{\star}$ may not be exactly equal to $\CWinf$.
The simulations of Section~\ref{RP.sec.simus} also support this fact: {\em Overpenalization}, that is, taking $C = C_{\mathrm{ov}} \CWinf$ with $C_{\mathrm{ov}} > 1$, can improve the prediction performance of $\ERM_{\mh}$ when $n$ is small, when $\sigma$ is large or when $\bayes$ is non-smooth.

This problem would appear even if the ``optimal'' constant $C^{\star}$ such that $\pen$ is non-asymptotically unbiased was known.
On Figure~\ref{fig.S2.surpen.Epenid}, the estimated model selection performance of the penalty $C_{\mathrm{ov}} \E\croch{\penid(m)}$ is plotted as a function of $C_{\mathrm{ov}}$, for experiment S2 of Section~\ref{RP.sec.simus}. It appears that the optimal overpenalization constant $C_{\mathrm{ov}}^{\star} \in (1.5 ; 2.35)$ for this particular problem.
More generally, the drawback of using $C=C^{\star}$ is that it does not take into account the deviations of $\penid(m)$ around its expectation. To avoid the possible overfit induced by these deviations, the constant $C$ must be slightly enlarged. A major issue remains: How to estimate $C_{\mathrm{ov}}^{\star}$ from data only, since it strongly depends on $n$, on $\sigma$, on the smoothness of $\bayes$ and on the number of models in $\M_n$?

\begin{figure}
\centerline{\epsfig{file=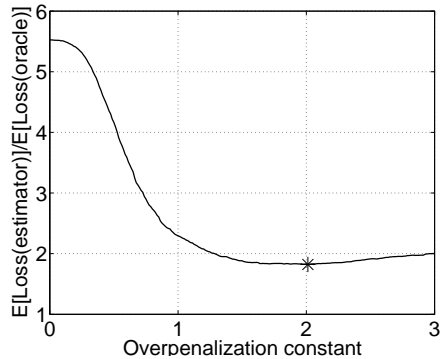,width=0.48\textwidth}}
\caption{The non-asymptotic need for overpenalization: the prediction performance $C_{\mathrm{or}}$ (defined in Section~\ref{VFCV.sec.simus.setup}) of the model selection procedure \eqref{def.mh.pen} with $\pen(m) = C_{\mathrm{ov}} \E\croch{\penid(m)}$ is represented as a function of $C_{\mathrm{ov}}$. Data and models are the ones of experiment S2: $n=200$, $\sigma(x)=x$, $\bayes(x) = \sin(\pi x)$. See Section~\ref{RP.sec.simus} for details. \label{fig.S2.surpen.Epenid}}
\vspace*{-6pt}
\end{figure}


One can think of choosing $C_{\mathrm{ov}}$ by $V$-fold cross-validation, but this would lead to a computationally intractable procedure.
An alternative idea is to use resampling for building a simultaneous confidence region on $\paren{\penid(m)}_{\mM_n}$ instead of estimating $\E\croch{\penid(m)}$ only (see \cite{Arl_Bla_Roq:2008:RC} on confidence regions built with general exchangeable resampling schemes).
Then, the uncertainty on the estimation of $\penid(m)$ can be taken into account for choosing a model, similarly to model selection procedures built upon relative bounds \cite{Aud:2004,Cat:2007}.
Finally, the choice of the overpenalization factor would be replaced by the choice of a confidence level which should be made by the practicioner.
See also \cite[Section~11.3.3]{Arl:2007:phd} for a discussion on a data-driven choice of the overpenalization factor.

\section{Discussion}\label{RP.sec.discussion}

\subsection{Comparison with other procedures}\label{RP.sec.comparaison}
In this article, the Resampling Penalization (RP) family of model selection procedures is defined and showed to satisfy some optimality properties under mild assumptions on the data (Theorems~\ref{RP.the.oracle_traj_non-as} and~\ref{RP.the.holder}).
In particular, RP is robust to the heteroscedasticity of the noise according to both theoretical and experimental results.
The price for robustness is that the computational cost of RP is generally larger than simple procedures like Mallows' $C_p$, even with the suggestions of Section~\ref{RP.sec.pratique.cost}.
The purpose of this subsection is to identify the ``easy'' problems, for which the computational cost of RP can be reduced by using $C_p$-like penalties without enlarging the prediction loss too much.

\subsubsection{Mallows' $C_p$} \label{RP.sec.calc.pen.comp.Mal}
Mallows' $C_p$ penalty is equal to $2 \sigma^2 D_m n^{-1}$ for a model $S_m$ of dimension $D_m$, when the noise-level $\sigma$ is constant. Non-asymptotic results about $C_p$-like penalties can be found in \cite{Bar_Bir_Mas:1999,Bar:2000,Bar:2002,Bir_Mas:2006}. They imply that Mallows' $C_p$ is asymptotically optimal in the homoscedastic framework, when the size of $\M_n$ is polynomial in $n$.

When the mean noise-level is unknown, it must be estimated. A classical estimator of $\E\croch{\sigma^2(X)}$ is defined by \eqref{RP.def.var-estim}.
Baraud \cite{Bar:2000,Bar:2002} showed that the resulting data-driven model selection procedure satisfies a non-asymptotic oracle inequality with leading constant close to one.

Assume for the sake of simplicity that $n$ is even and let $S_{n/2}$ be a model such that each piece of the associated partition contains exactly two data points. Reordering the $(X_i,Y_i)$ according to $X_i$,
\begin{align*}
\pen_{\textrm{Mallows}} (m) &=
\frac{2 D_m}{n^2} \sum_{i=1}^{n/2} \carre{Y_{2i} - Y_{2i-1}}
\end{align*}
so that
\begin{gather} \label{eq.Mal.hist.decomp}
\El \croch{ \pen_{\textrm{Mallows}}(m) } \approx \frac{2}{n} \sum_{\lamm} \paren{D_m \phl} \carre{\sigla} + \frac{2 D_m}{n^2} \sum_{i=1}^{n/2} \paren{\bayes(X_{2i}) - \bayes(X_{2i-1})}^2 \\ \notag
\mbox{where} \qquad \carre{\sigla} \egaldef \E\croch{ \sigma(X)^2 \sachant X \in \Il}
. \end{gather}
This should be compared with the result of Proposition~\ref{RP.pro.EpenRP-Epenid}:
\begin{gather} \label{eq.penid.hist.decomp}
\El \croch{ \penid(m) } \approx \frac{2}{n} \sum_{\lamm}
\paren{\carre{\sigla} + \carre{\sigld}} \\ \notag \mbox{where} \qquad
\carre{\sigld} \egaldef \E\croch{ \paren{\bayes(X) - \bayes_m(X)}^2
\sachant X \in \Il} .\end{gather}
Although both Mallows' $C_p$ and the ideal penalty are in expectation the sum of a ``variance'' term (involving the $\carre{\sigla}$) and a ``bias'' term (involving the variations of $\bayes$ through $(\bayes(X_{2i}) - \bayes(X_{2i-1}))^2$ or $(\sigld)^2$), they differ on at least two points.

First, when $\bayes$ is smooth and $\min_{\lamm} \set{n \phl}$ is large, the ``bias'' term in \eqref{eq.Mal.hist.decomp} is negligible in front of the one of \eqref{eq.penid.hist.decomp}, which means that Mallows' $C_p$ underpenalizes when the ``bias'' component of $\penid$ is large.
Second, the ``variance'' component of $\penid$, which is the main one in general, is distorted in Mallows' $C_p$: the part of the penalty corresponding to $\Il$ is multiplied by $D_m \phl$ which is not close to 1 when the partition $\paren{\Il}_{\lamm}$ is not regular with respect to $\loi(X)$.
This happens for instance in experiments S2 and HSd2 of Section~\ref{RP.sec.simus}. Therefore, there are at least three possibly ``hard'' problem classes:
\begin{itemize}
\item heteroscedastic noise, with irregular histograms and $X$ uniform (for instance S2, HSd2 in Section~\ref{RP.sec.simus}, or Svar2 in \cite{Arl:2008b:app}),
\item heteroscedastic noise, with regular histograms and $X$ highly non-uniform on $\X$,
\item regression function $\bayes$ with jumps (such as HeaviSine\footnote{However, in experiment HSd1, Mallows' $C_p$ still behaves quite well compared to RP. We do not know whether the non-smoothness of $\bayes$ can actually make Mallows' $C_p$ fail.}) or large non-smooth areas (such as Doppler in \cite{Arl:2008b:app}).
\end{itemize}
In either of these cases, one should avoid the use of $C_p$-like penalties, and we suggest resampling penalties as an efficient alternative. As explained in Section~\ref{RP.sec.calc.penvslin} below, the first class of problems can make any penalty proportional to the dimension $D_m$ suboptimal.

\subsubsection{Linear penalties} \label{RP.sec.calc.penvslin}
Mallows' $C_p$ is simple because it is a linear function of the dimension $D_m$ of $S_m$:
\begin{equation} \label{eq.pen.lin} \pen(m) = \widehat{K} D_m \end{equation} and $\widehat{K}$ is the only constant to determine.
Depending on what is known on the mean variance level, the constant $\widehat{K}_{\mathrm{Mallows}}$ can be defined as
\[ 2 \E\croch{\sigma(X)^2} n^{-1} \quad \mbox{or} \quad 2 \widehat{\sigma}^2 n^{-1} . \]
Refined versions of Mallows' $C_p$ have also been proposed \cite{Bar_Bir_Mas:1999,Bar:2002,Bir_Mas:2006} but they are still linear or very close to linearity.

However, according to \eqref{RP.eq.Em_penid}, the ideal penalty is not linear in general, even in expectation.
Moreover, there exist some frameworks in which any penalty of the form \eqref{eq.pen.lin} is suboptimal when data are heteroscedastic \cite{Arl:2008:shape}, that is, it cannot satisfy any oracle inequality with leading constant smaller than some absolute constant $\kappa>1$.
In other words, the \emph{optimal linear penalization procedure} $\penoptlin(m) \egaldef \widehat{K}^{\star} D_m$ is suboptimal, where
\begin{gather*}
\widehat{K}^{\star} \in \arg\min_{K>0} \set{ P\gamma\parenb{\ERM_{\mh(K)}} } \\
\mbox{and} \qquad \forall K>0, \quad \mh(K) \in \arg\min_{\mM_n} \set{ P_n \gamma\paren{\ERM_m} + K D_m} .
\end{gather*}
As showed by Theorem~\ref{RP.the.oracle_traj_non-as}, RP does not suffer from this drawback.

On the one hand, the optimal linear penalization procedure has a better model selection performance than RP for S1, S2 and HSd1, which is not surprising for the ``easy'' problems where Mallows' $C_p$ is almost optimal (S1, HSd1). It is less intuitive for S2 where data are heteroscedastic.
Considering that $\penoptlin$ uses the knowledge of the true distribution $P$, one can understand that it is sufficient to keep a good performance for ``intermediate'' problems.

On the other hand, in experiment HSd2, the optimal linear penalization has a model selection performance $C_{\textrm{or}} = 1.18 \pm 0.01$, which is worse than the one of RP ($C_{\textrm{or}} \leq 1.11$). Thus, the most difficult problem of Section~\ref{RP.sec.simus} (with a large collection of models, heteroscedastic data and bias) gives an example where linear penalties are definitely not adapted, in addition to the ones of \cite{Arl:2008:shape}.

\subsubsection{\emph{Ad hoc} procedures} \label{RP.sec.comparaison.ad-hoc}
One of the main advances with Theorems~\ref{RP.the.oracle_traj_non-as} and~\ref{RP.the.holder} is that RP is proved to work in the heteroscedastic framework contrary to Mallows' $C_p$.
Nevertheless, in a framework such as the one of experiment S2, Mallows' $C_p$ can be adapted to heteroscedasticity by splitting $\X$ into several parts where $\sigma$ is almost constant, and performing the histogram selection procedure with Mallows' $C_p$ separately on each part of $\X$.

More generally, Efromovich and Pinsker \cite{Efr_Pin:1996} and Galtchouk and Pergamenschikov \cite{Gal_Per:2008} (among several others) defined estimators of $\bayes$ that are minimax adaptive in the heteroscedastic framework, the latter by model selection.
In the Gaussian regression framework, Gendre \cite{Gen:2008} proposed a model selection method for estimating simultaneously the regression function and the noise level.

All these procedures may perform slightly better than RP in terms of prediction loss.
They are called ``\latin{ad hoc}'' because they have been specially designed for the heteroscedastic framework (and a particular collection of estimators for \cite{Gal_Per:2008,Gen:2008}).
On the contrary, RP is a general-purpose device: It was neither built to be adaptive to heteroscedasticity nor to take advantage of a specific model, and RP has exactly the same definition in the general prediction framework (see Section~\ref{RP.sec.discussion.classif}).

When no information is available on the data or when no model selection procedure is known for using such information, we suggest the use of RP.
Moreover, available information can be partial or wrong.
Then, using an \latin{ad hoc} procedure would be disastrous whereas a general device like RP would still work.
In short, choose RP if you have no useful information or if you do not trust them.

\subsubsection{Other model selection procedures by resampling} \label{RP.sec.comparaison.other}
The most well-known resampling-based model selection procedure is cross-va\-li\-da\-tion. For practical reasons, it is often used in its $V$-fold version which can have some tricky behavior, in particular for choosing $V$ \cite{Yan:2007b,Arl:2008a}. This can also be showed in the simulation experiments of Section~\ref{RP.sec.simus} (see Table~\ref{RP.Tableun}): In HSd1, $V=2$ performs better than $V \in \set{5, 10, 20}$, a phenomenon explained in \cite{Arl:2008a} by analyzing how the bias of the $V$-fold criterion depends on $V$.

$V$-fold penalization, that is, RP with a $V$-fold subsampling scheme, was proposed in \cite{Arl:2008a} where it was showed to improve significantly the model selection performance of VFCV.
In this paper and in \cite{Arl:2008b:app}, RP with several exchangeable resampling schemes ---generalizing the $V=n$ case--- is proved to perform at least as well as $V$-fold penalization and often better.


Several penalization procedures use the bootstrap for estimating the ideal penalty \cite{Efr:1983,Cav_Shu:1997,Shi:1997}.
As noticed in Remark~\ref{RP.rem.R1W.R2W}, the penalization procedures studied by Shibata \cite{Shi:1997} are quite close to RP, although they are restricted to bootstrap weights, which are the worst ones in the framework of the present paper (see Sections~\ref{sec.calc.delta.classic} and~\ref{RP.sec.pratique.weight}).
Moreover, they do not consider useful to multiply the penalty by a factor $C$ possibly different from one, contrary to what is suggested in RP.
The factor $C$ is crucial because it disconnects the choice of the weights from the overpenalization problem.


In order to select the correct model asymptotically with probability one, Shao \cite{Sha:1996} proposed to use RP with the $\mnefr$ out of $n$ bootstrap and provided a sufficient condition on $\mnefr$ to achieve model consistency.
Thanks to the unified approach for all the exchangeable resampling weights provided in this paper, Shao's condition can be rewritten as $C=1 \gg \CWinf$ (see Remark~\ref{RP.rem.R1W.R2W}), which corresponds to the known fact that model consistency requires overpenalization within a factor tending to infinity with $n$ \cite{Aer_Cla_Har:1999}.
Hence, we conjecture that RP with a constant $C \gg \CWinf$ is model consistent for most exchangeable $W$, which may improve Shao's penalties since Efr($\mnefr$) weights are probably not the best weights in terms of accuracy (see Section~\ref{RP.sec.calc.delta}) and variability\footnote{Taking into account all the data for computing the resampling penalty with Efr($\mnefr$) weights is computationally costly when $n/\mnefr$ is large.}.

\subsection{Resampling Penalization in the general prediction framework}
\label{RP.sec.discussion.classif}
As mentioned in Section~\ref{sec.cadreRP.cadre}, Resampling Penalization is a general-purpose method which is definitely not restricted to the histogram selection problem.
The purpose of this subsection is to define properly RP in the general prediction framework and to discuss briefly what differences can be expected compared to the histogram selection framework.

\subsubsection{Framework} \label{RP.sec.discussion.classif.cadre}
Suppose we observe some data $(X_1,Y_1), \ldots, (X_n,Y_n) \in \X \times \Y$ independent with common distribution $P$.
The goal is to predict $Y$ given $X$ where $(X,Y) \sim P$ is independent of the data.
The quality of a predictor $t: \X \mapsto \Y$ is measured by the prediction loss $ P \gamma(t) \egaldef \E_{(X,Y)} \croch{ \gamma(t,(X,Y)) }$ where $(X,Y) \sim P$ and $\gamma$ is a given contrast function.
Typically, $\gamma(t,(x,y))$ measures the discrepancy between $t(x)$ and $y$.
The excess loss is defined as $\perte{t} \egaldef P\gamma\paren{t} - \inf_{t: \X \flens \Y} P\gamma\paren{t}$, even if $\bayes = \arg\min_t \set{P\gamma\paren{t}}$ is not well-defined.
Classical examples are least-squares regression where $\Y = \R$ and $\gamma(t,(x,y)) = (t(x) - y)^2$ and binary supervised classification where $\Y = \set{0,1}$ and $\gamma(t,(x,y)) = \un_{t(x) \neq y}$ is the 0-1
contrast.

A general prediction algorithm $\ERM$ is then defined as a function associating a predictor to any data sample. In order to simplify the presentation, algorithms are assumed to depend only on the empirical distribution $P_n = n^{-1} \sum_{i=1}^n \delta_{(X_i,Y_i)}$ as an input\footnote{Otherwise, we can consider algorithms whose input is any weighted sample.}.
For instance, the empirical risk minimizer over a set $S_m$ of predictors is defined as $\ERM_m (P_n) \egaldef \arg\min_{t \in S_m} P_n \gamma\paren{t}$, provided the minimum in $S_m$ exists and is unique.

Let us assume that a collection of algorithms $\paren{\ERM_m}_{\mM_n}$ is given. The goal is to select some data-dependent $\mh \in \M_n$ minimizing the prediction loss $P \gamma\paren{\ERM_m (P_n)}$.
The penalization method consists in selecting
\begin{equation*}
\mh \in \arg\min_{\mM_n} \set{ P_n \gamma\paren{ \ERM_m \paren{ P_n }} + \pen(m) } , \end{equation*}
where $\pen : \M_n \flens \R$ is a penalty function, possibly data-dependent.
Since the goal is to minimize the prediction loss, the {\em ideal penalty} is
\[ \penid(m) \egaldef (P - P_n) \gamma\paren{ \ERM_m (P_n) } = F_m(P,P_n) \]
which cannot be used because it depends on the unknown distribution $P$.
When $\M_n$ is not too large (for instance, when $\card(\M_n) \leq C n^{\alpha}$ for some positive constants $C,\alpha$), a natural strategy is to define $\pen(m)$ as an estimator of $\penid(m)$ with a bias as small as possible.

\subsubsection{Definition of Resampling Penalization}
As detailed in Section~\ref{sec.cadreRP.heur}, the resampling heuristics can be used for estimating $\E\croch{\penid(m)} = \E\croch{F_m(P,P_n)}$, leading to the following procedure.
\begin{proc}[Resampling Penalization] \label{RP.def.proc.gal}
\hfill
\begin{enumerate}
\item Replace $\M_n$ by
\begin{equation*}
\Mh_n = \left\{ \mM_n \telque \ERM_m(P_n) \mbox{ is well-defined } \right\} .
\end{equation*}
\item Choose a resampling scheme, that is the distribution $\loi(W)$ of a weight vector $W$.
\item Choose a constant $C \geq \CWinf$.
\item Compute the following resampling penalty for each $\mM_n$:
\begin{equation} \label{RP.def.pen.gal}
\pen(m) = C \Es \croch{ P_n \gamma \paren{ \ERM_m \paren{ \Pnb} } - \Pnb \gamma \paren{\ERM_m \paren{\Pnb}} } ,
\end{equation}
where $\Pnb \egaldef n^{-1} \sum_{i=1}^n W_i \delta_{(X_i,Y_i)}$.
\item Select $\mh \in \arg\min_{m \in \Mh_n} \set{ P_n \gamma\paren{ \ERM_m \paren{ P_n }} + \pen(m) }$.
\end{enumerate}
\end{proc}

As for the histogram selection problem, two possible problems have to be solved.
%
%
First, $\ERM_m (\Pnb)$ may not be well-defined for a.e. $W$ even if $\mMh_n$.
A way to define properly the resampling penalty for every $\mMh_n$ such that $\ERM_m (\Pnb)$ is well-defined for every $W \in (0,+\infty)^n$ is suggested in \cite[Section~8.1]{Arl:2007:phd}.
This assumption is satisfied by regressograms (hence, in the framework of the rest of the paper) for which the suggest of \cite[Section~8.1]{Arl:2007:phd} yields exactly the penalty \eqref{RP.def.pen.his}.

%
Second, the constant $\CWinf$ such that \eqref{RP.def.pen.gal} estimates unbiasedly $\penid(m)$ when $C = \CWinf$ is required in Procedure~\ref{RP.def.proc.gal}.
For the histogram selection problem, the explicit expression of \CWinf\ follows from Propositions~\ref{RP.pro.EpenRP-Epenid} and~\ref{RP.pro.comp.Epen.Ep2}.
In general, the asymptotic theory of exchangeable bootstrap empirical processes \cite[Theorem~3.6.13]{vdV_Wel:1996} suggests that $\CWinf=1$ if $ \var(W_1) \ll 1 $, which holds for the classical weights Efr, Rad, Poi and Rho; nevertheless, asymptotic control on the bias is not sufficient when the collection of algorithms is allowed to depend on the sample size $n$, as in the histogram selection problem.
Therefore, further theoretical investigations would be useful to compute the theoretical value of \CWinf\ to be used in Procedure~\ref{RP.def.proc.gal}.
From the practical point of view, the data-driven calibration algorithm of \cite{Arl_Mas:2008} can be used for choosing the constant $C$ in front of the resampling penalty.

\vspace*{-3pt}
\subsubsection{Model selection properties of Resampling Penalization}

The theoretical validity of Procedure~\ref{RP.def.proc.gal} is only proved for histogram model selection in this paper, because precise non-asymptotic controls of the ideal penalty and its resampling counterpart are needed.
To our knowledge, the only known result about model selection with Resampling Penalization was that RP with the classical bootstrap weights (Efr) is asymptotically optimal for selecting among maximum likelihood estimators in \cite{Shi:1997}, assuming that the distribution $P$ belongs to some parametric family of densities.

RP can be conjectured to enjoy adaptivity properties for a wide class of model selection problems for two main reasons.
First, RP relies on the resampling idea which is known to be robust in a wide variety of frameworks; Theorems~\ref{RP.the.oracle_traj_non-as} and~\ref{RP.the.holder} have confirmed the robustness of RP to heteroscedasticity, whereas RP has not been designed specifically for least-squares regression with heteroscedastic data.
Second, several of the key concentration inequalities used to prove Theorems~\ref{RP.the.oracle_traj_non-as} and~\ref{RP.the.holder} have been extended in \cite[Propositions~8 and~10]{Arl_Mas:2008} to a general framework including bounded regression and binary classification.

%
As mentioned at the end of Section~\ref{RP.sec.discussion.classif.cadre}, Procedure~\ref{RP.def.proc.gal} should be restricted to choosing among a number of algorithms at most polynomial in $n$.
Indeed, when $\card(\M_n)$ is larger, estimating unbiasedly $\penid$ can yield strong overfitting \cite{Bir_Mas:2006}.
Therefore, RP must be modified for large collections $\M_n$.
We suggest to group algorithms according to some modelling complexity index $C_m$, such as the dimension of $S_m$ if $\ERM_m$ is the empirical risk minimizer over some vector space $S_m$; then, for every $C \in \mathcal{C}_n = \set{C_m \telque \mM_n}$, define $\ERM_C \egaldef \ERM_{\mh(C)}$ where $\mh(C) \in \arg\min_{C_m = C} P_n \gamma\paren{ \ERM_m (P_n)}$; finally, apply Procedure~\ref{RP.def.proc.gal} to the collection $\paren{\ERM_C}_{C \in \mathcal{C}_n}$, assuming that $\card\paren{ \mathcal{C}_n}$ is at most polynomial in $n$.

\vspace*{-3pt}
\subsubsection{Related penalties for classification} \label{RP.sec.discussion.classif.classif}
In the classification framework, RP should be compared to several classical resampling-based penalization methods.
First, RP with Efr weights was first introduced by Efron \cite{Efr:1983} and called bootstrap penalization; its main drawback is its bias (as for the histogram selection problem), which can be corrected in several ways, using for instance the double bootstrap penalization or the .632 bootstrap \cite{Efr:1983}. Nevertheless, the computational cost of the double bootstrap is heavy and the general validity of the .632 bootstrap is questionable because of its poor theoretical grounds.

Second, the global Rademacher complexities were introduced in order to obtain theoretically validated model selection procedures in classification \cite{Kol:2001,Bar_Bou_Lug:2002}.
They are resampling estimates of
\[ \penidglo(m) \egaldef \sup_{t \in S_m} \set{(P - P_n) \gamma(t)} \geq (P - P_n) \gamma\paren{\ERM_m(P_n)} = \penid(m) , \]
with Rad weights; more recently, Fromont \cite{Fro:2007} generalized global Rademacher complexities to a wide family of exchangeable resampling weights and obtained non-asymptotic oracle inequalities.
Nevertheless, global complexities (that is, estimates of $\penidglo$) are too large compared to $\penid$ so that they cannot achieve fast rates of estimation when the margin condition \cite{Mam_Tsy:1999} holds.

Therefore, localized penalties taking into account the closeness between $\ERM_m(P_n)$ and $\bayes$ have been introduced, in particular local Rademacher complexities \cite{Lug_Weg:2004,Bar_Bou_Men:2005,Bar_Men_Phi:2004,Kol:2006}; these papers proved sufficiently tight oracle inequalities to ensure that the final prediction loss can achieve fast rates.
Nevertheless, local Rademacher complexities are computationally heavy and depend on several constants which are difficult to calibrate.


RP aims at combining the advantages of these three approaches in classification.
From the practical point of view, RP is computationally tractable (see Section~\ref{RP.sec.pratique.cost}) and reasonably easy to calibrate (see Section~\ref{RP.sec.pratique.const}).
Compared to global Rademacher complexities, resampling penalties estimate directly $\penid$, so that RP should be able to achieve fast rates of estimation when the margin condition holds.
Finally, contrary to the bootstrap penalty, RP can be used with several resampling weights including i.i.d. Rademacher weights (Rad), so that the bias of RP may not have to be corrected.

\subsection{Conclusion}
\label{RP.sec.discussion.conclu}
This article intends to help the practicioner to answer the following question: When should Resampling Penalization be used? To sum up, we list below the advantages and drawbacks of RP \latin{vs.} the classical methods.

\vspace*{-3pt}
\subsubsection*{Advantages of RP} 
\vspace*{-6pt}
\begin{itemize}
\item generality: well-defined in almost any framework.
\item robustness and versatility: designed for the cautious user.
\item adaptivity to several properties, in particular heteroscedasticity and smoothness of the target.
\item flexibility: possibility of overpenalization, either for non-asymptotic prediction or for identification.
\end{itemize}

\vspace*{-9pt}
\subsubsection*{Drawbacks of RP} 
\vspace*{-6pt}
\begin{itemize}
\item computation time: one may prefer $V$-fold procedures such as $V$-fold cross-validation or $V$-fold penalties \cite{Arl:2008a}.
\item possibly outperformed by Mallows' $C_p$ (for easy problems) or \latin{ad hoc} procedures (in some particular frameworks, when some information on the data is available).
\end{itemize}

\section{Proofs} \label{RP.sec.preuves}

\vspace*{-3pt}
\subsection{Notation} \label{RP.sec.proof.nota}
\vspace*{-3pt}

Before starting the proofs, we introduce some additional notation and conventions:
\begin{itemize}
\item The letter $L$ denotes ``some positive absolute constant, possibly different from some place to another''.
In the same way, a positive constant which depends on $c_1, \ldots, c_k$ is denoted by $L_{c_1, \ldots, c_k}$; if \hyp{A} denotes a set of assumptions, $L_{\hyp{A}}$ denotes any positive constant depending on the parameters appearing in \hyp{A}.
\item By convention, $\infty \un_{E}$ and $\un_{E}/0$ are both equal to zero when the event $E$ does not hold.
\item For any $x \in \R$, $x_+ \egaldef \maxi{x}{0} = \max(x,0)$ and $x_- \egaldef \maxi{(-x)}{0}$.
\item For any non-negative random variable $Z$, $\einvz{\loi(Z)} \egaldef \E\croch{Z} \E \croch{Z^{-1} \un_{Z>0}} $.
\item For any model $\mM_n$,
\begin{gather*}
 p_1(m) \egaldef P \paren{\gamma(\ERM_m) - \gamma(\bayes_m)} \qquad  p_2(m) \egaldef P_n \paren{\gamma(\bayes_m) - \gamma(\ERM_m)} \\
\delc(m) \egaldef (P_n - P) \paren{ \gamma(\bayes_m) - \gamma(\bayes) } .
\end{gather*}
\item Histogram-specific notation: for any $q>0$, $\mM_n$, $\lamm$ and any random variable $Z$,
\begin{align*}
\El\croch{Z} &\egaldef \E \croch{Z \sachant \paren{\un_{X_i \in \Il}}_{1 \leq i \leq n, \, \lamm} }
\qquad \qquad
\norm{Z}_{q}^{(\Lambda_m)} \egaldef \El \croch{ \absj{Z}^q }^{1/q}
\\
m_{q,\lambda} &\egaldef \norm{ Y -
\bayes_m(X) }_{q,\lambda} \egaldef \left( \E \left[ \absj{Y -
\bayes_m(X)}^q \sachant X \in \Il \right] \right)^{1/q}
\\
S_{\lambda,1} &\egaldef \sum_{X_i \in \Il} \paren{Y_i - \betl} \qquad \mbox{and} \qquad S_{\lambda,2} \egaldef \sum_{X_i \in \Il} \carre{Y_i - \betl} . \end{align*}
\item Conventions for $p_1$ and $p_2$ when $\ERM_m$ is not well-defined (in the histogram framework):
\begin{align*}
\punmin (m) &\egaldef \punzero(m) + \sum_{\lamm}
\pl \carre{\sigl} \un_{\phl=0} \\
\mbox{with} \quad \punzero (m) &\egaldef \sum_{\lamm}
\frac{\pl \un_{\phl>0} }{(n \phl)^2} S_{\lambda,1}^2
\\ \mbox{and} \quad \widetilde{p}_2(m) &\egaldef p_2(m) + \frac{1}{n} \sum_{\lamm} \carre{\sigl} \un_{n \phl = 0}
\end{align*}
Note that $p_1(m) = \punzero (m) = \punmin (m)$ and $p_2(m) = \widetilde{p}_2(m)$ are well-defined when $\ERM_m$ is uniquely defined, and other models are always removed from $\M_n$.
The above convention is only important when writing expectations, so it is merely technical.
In the following, $\punmin$ (resp. $\widetilde{p}_2$) will often be written simply $p_1$ (resp. $p_2$).
\end{itemize}

Using the above notations, $p_1(m)$ and $p_2(m)$ can now be computed explicitly for histogram models.
For any $\mM_n$ such that $\min_{\lamm} \phl >0$,
\begin{align} \label{VFCV.eq.p1.hist}
p_1(m) &= \sum_{\lamm} \pl \paren{\betl - \bethl}^2 = \frac{1}{n} \sum_{\lamm} \paren{ \frac{\pl}{\phl} \frac{S_{\lambda,1}^2}{n \phl} }
\\ \label{VFCV.eq.p2.hist}
p_2(m) &= \sum_{\lamm} \phl \paren{\betl - \bethl}^2 = \frac{1}{n} \sum_{\lamm} \paren{ \un_{n\phl > 0} \frac{S_{\lambda,1}^2 }{n \phl} }
\end{align}
since $\bethl - \betl = S_{\lambda,1} / (n \phl)$.

\vspace*{-2pt}
\subsection{General framework} \label{RP.sec.preuves.gal}
\vspace*{-3pt}

The main results (Theorems~\ref{RP.the.oracle_traj_non-as} and~\ref{RP.the.holder}) actually are corollaries of a more general oracle inequality (Lemma~\ref{RP.le.gal}).
First, two different assumption sets under which Lemma~\ref{RP.le.gal} holds are stated in this subsection.
The first one \textbf{(Bg)} deals with bounded data, the second one \textbf{(Ug)} with unbounded data.

\vspace*{-2pt}
\subsubsection{Bounded assumption set \textbf{(Bg)}}
\vspace*{-3pt}

\makeatletter
\setlength\leftmargini   {34\p@}
\makeatother
\begin{enumerate}
\item[] There is some noise: $\norm{\sigma(X)}_2>0$.
\item[\hypPpoly] Polynomial size of $\M_n$: $\card(\M_n) \leq \cM n^{\aM}$.
\item[\hypPrich] Richness of $\M_n$: $\exists m_0 \in \M_n$ s.t. $D_{m_0} \in \croch{ \sqrt{n};c_{\mathrm{rich}} \sqrt{n} }$.
\item[\hypPech] The weight vector $W$ is exchangeable, among Efr, Rad, Poi, Rho and Loo.
\item[\hypPconst] The constant $C$ is well chosen: $\eta \CWinf \geq C \geq \CWinf$.
\item[\hypAb] Bounded data: $\norm{Y_i}_{\infty} \leq A < \infty$.
\item[\hypAml] Local moment assumption: there exist $a_{\ell} ,\xi_{\ell}, D_0 \geq 0$ such that for every $q \geq 2$, for every $\mM_n$ such that $D_m \geq D_0$,
\begin{gather*}
P^{\ell}_m(q) \egaldef \frac{\sqrt{D_m \sum_{\lamm} m_{q,\lambda}^4}}{\sum_{\lamm} m_{2,\lambda}^2 } \leq a_{\ell} q^{\xi_{\ell}}
 . \end{gather*}
\item[\hypAp] Polynomially decreasing bias: there exist $\beta_1\geq\beta_2>0$ and $\cbiasmaj,\cbiasmin>0$ such that, for every $\mM_n$, \[ \cbiasmin D_m^{-\beta_1} \leq \perte{\bayes_m} \leq \cbiasmaj D_m^{-\beta_2} . \]
\item[\hypQ] There exist $\cQpm > 0$ and $D_0 \geq 0$ such that for every $\mM_n$ with $D_m \geq D_0$,
\begin{equation*}
\Qmp \egaldef \frac{n \E\croch{p_2(m)}}{D_m} = \frac{1}{D_m} \sum_{\lamm} \sigl^2 \geq \cQpm > 0 . \end{equation*}
\item[\hypArXl] Lower regularity of the partitions for $\loi(X)$: there exists $\crXl>0$ such that for every $\mM_n$, $D_m \min_{\lamm} \pl \geq \crXl$.
\end{enumerate}

\subsubsection{Unbounded assumption set \textbf{(Ug)}}
\hypAb\ is replaced in \hypBg\ by
\makeatletter
\setlength\leftmargini   {38\p@}
\makeatother
\begin{enumerate}
\item[\hypAsigmax] Noise-level bounded from above: $\sigma^2(X) \leq \sigmax^2 < \infty$ a.s.
\item[\hypAciblemax] Bound on the target function: $\norm{\bayes}_{\infty} \leq A < \infty$.
\item[\hypAgeps] Global moment assumption for the noise: there exist $a_{g \epsilon}, \xi_{g \epsilon} \geq 0$ such that for every $q \geq 2$,
\begin{equation*}
P^{g \epsilon} (q) \egaldef \norm{\epsilon}_q \leq a_{g \epsilon} q^{\xi_{g \epsilon}} .
\end{equation*}
\item[\hypAdel] Global moment assumption for the bias: there exists $\cdelmglo>0$ such that for every $\mM_n$ with $D_m \geq D_0$,
\begin{equation*} \norm{\bayes-\bayes_m}_{\infty} \leq \cdelmglo \norm{\bayes(X) - \bayes_m(X)}_2 . \end{equation*}
\end{enumerate}
\makeatletter
\setlength\leftmargini   {22\p@}
\makeatother

\subsubsection{General result}
\begin{lemma} \label{RP.le.gal}
Let $n \in \N\backslash\set{0}$, $\gamma_0>0$ and $\mh$ be defined by Procedure~\ref{RP.def.proc.his}. Assume that either \hypBg\ or \hypUg\ holds with constants independent of $n$.

Then, there exists a constant $K_1$ \textup{(}that depends on $\gamma_0$ and all the constants in \hypBg\ \textup{(}resp. \hypUg\textup{)}, but not on $n$\textup{)} such that
\begin{equation} \label{RP.eq:oracle_non-as_x_traj}
 \perte{\ERM_{\mh}} \leq \crochb{ 2 \eta - 1 + \paren{\ln(n)}^{-1/5} }\inf_{\mM_n} \set{ \perte{\ERM_m} }
\end{equation}
holds with probability at least $1 - K_1 n^{-\gamma_0}$.
\end{lemma}
Lemma~\ref{RP.le.gal} is proved in Section~\ref{RP.sec.proof.le.gal}.

\begin{remark} \label{RP.rem.oracle.traj.sansAp}
If the lower bound in \hypAp\ is removed from the assumption set,
then there exist constants $\gamma_1,\gamma_2>0$ (depending only on $\xi_{\ell}$, resp. on $\xi_{\ell}$ and $\xi_{g \epsilon}$) and an event of probability at least $1 - K_1 n^{-\gamma_0}$ on which
\begin{equation} \label{RP.eq:oracle_non-as_x_traj.sansAp}
 \perte{\ERM_{\mh}} \leq \crochb{ 2 \eta - 1 + \paren{\ln(n)}^{-1/5} } \hspace{-0.2cm} \mathop{\inf_{\mM_n}}_{D_m \geq \paren{\ln(n)}^{\gamma_1}} \set{ \perte{\ERM_m} } + \frac{\paren{\ln(n)}^{\gamma_2}}{n} .
\end{equation}
This assertion is proved in Section~\ref{sec.proof.rem.oracle.traj.sansAp}.
\end{remark}

\begin{remark} \label{RP.rem.oracle.traj.convention}
In the infimum in \eqref{RP.eq:oracle_non-as_x_traj}, $\ERM_m$ may not be well-defined for some $\mM_n$. By convention $\perte{\ERM_m}$ is defined as $+\infty$ for these $m$.

From the proof of Lemma~\ref{RP.le.gal}, there exists a constant $c>0$ (depending on $\aM$, $\gamma_0$ and $\crXl$) such that every model of dimension smaller than $c n \paren{\ln(n)}^{-1}$ belongs to $\Mh_n$ on the event where \eqref{RP.eq:oracle_non-as_x_traj} holds. For each of these models,
\[ \perte{\ERM_m} = \perte{\bayes_m} + \punzero(m) = \perte{\bayes_m} + \punmin(m) \]
so that the infimum can be restricted to models of dimension smaller than $c n \paren{\ln(n)}^{-1}$ with any of these conventions for $\perte{\ERM_m}$.
\end{remark}


The main results of the paper (Theorems~\ref{RP.the.oracle_traj_non-as} and~\ref{RP.the.holder}) can now be proved, which is done in Sections~\ref{sec.proof.the.oracle_traj_non-as}--\ref{RP.sec.proof.holder}.

First, the assumptions of Theorem~\ref{RP.the.oracle_traj_non-as} imply \hypBg.
Second, the alternative assumption sets stated in Section~\ref{RP.sec.main.hyp.alt} imply \hypBg.
Third, the assumptions of Theorem~\ref{RP.the.holder} imply \hypBg\ except the lower bound in \hypAp, so that Remark~\ref{RP.rem.oracle.traj.sansAp} can be used instead of Lemma~\ref{RP.le.gal}.

\subsection[Proof of Theorem~1]{Proof of Theorem~\ref{RP.the.oracle_traj_non-as}} \label{sec.proof.the.oracle_traj_non-as}

Lemma~\ref{RP.le.gal} is applied with $\gamma_0=2$. In order to deduce \eqref{RP.eq.oracle_traj_non-as}, it remains to show that \hypAml\ and \hypQ\ are satisfied. Both hold with $D_0=1$ since for every $\mM_n$,
\begin{align} \label{eq.maj.Plmq}
P^{\ell}_m(q) &= \frac{\sqrt{ \sum_{\lamm} m_{q,\lambda}^4}} {\sqrt{D_m} \Qmp }
\leq \frac{ \norm{Y - \bayes_m(X)}_{\infty}^2 }
{ \Qmp } \leq \frac{4 A^2}{\Qmp} \\ \notag
\Qmp &\egaldef \frac{1}{D_m} \sum_{\lamm} \croch{ \paren{\sigla}^2 + \paren{\sigld}^2} \geq \sigmin^2 .
\end{align}

Let $\Omega_n$ be the event on which \eqref{RP.eq.oracle_traj_non-as} holds true. Then,
\begin{align*}
\E \croch{ \perte{\ERM_{\mh}} } &= \E \croch{ \perte{\ERM_{\mh}} \un_{\Omega_n} } + \E \croch{ \perte{\ERM_{\mh}} \un_{\Omega_n^c} } \\
&\leq \croch{ 2\eta - 1 + \varepsilon_n } \E \croch{ \inf_{\mM_n} \set{ \perte{\ERM_m} } } + A^2 K_1 \Prob\paren{\Omega_n^c} \end{align*}
which proves \eqref{RP.eq.oracle_class_non-as}.
Following Remark~\ref{RP.rem.oracle.traj.convention}, \eqref{RP.eq.oracle_class_non-as} also holds with $\M_n$ replaced by
\[ \setb{ \mM_n \telque D_m \leq c(\aM,\crXl) n \paren{\ln(n)}^{-1} } \]
and the convention $p_1(m) = \punzero(m)$. \qed

\subsection[Proof of Theorem~1: alternative assumptions]{Proof of Theorem~\ref{RP.the.oracle_traj_non-as}: alternative assumptions} \label{RP.sec.preuves.Theoremalt}
In this section, the statements of Section~\ref{RP.sec.main.hyp.alt} are proved.

\subsubsection{No uniform lower bound on the noise-level}
When $\sigmin = 0$ in \hypAn, Lemma~\ref{RP.le.minQmp} below proves that \hypQ\ also holds with $D_0 = L_{\hypBg}$.
Therefore, using \eqref{eq.maj.Plmq}, \hypAml\ holds with the same $D_0$. \qed
\begin{lemma} \label{RP.le.minQmp}
Let $\X \subset \R^k$, $\mM_n$, and assume that positive constants $ \crDu, \alpha_d, \crLu, K_{\sigma}, J_{\sigma}$ exist such that
\makeatletter
\setlength\leftmargini   {31\p@}
\makeatother
\begin{itemize}
\item[\hypArDu] $\max_{\lamm}\set{\diam(\Il)} \leq \crDu D_m^{-\alpha_d} \diam(X) ,$
\item[\hypArLu] $\max_{\lamm}\set{\Leb(\Il)} \leq \crLu D_m^{-1}$ and
\item[\hypAsig] $\sigma$ is piecewise $K_{\sigma}$-Lipschitz with at most $J_{\sigma}$ jumps.
\end{itemize}
Then,
\[ \Qmp \geq \frac{\Leb(\X) \norm{\sigma}^2_{L^2(\Leb)}}{2 \crLu} - \frac{ K_{\sigma}^2 \paren{\crDu}^2 \diam(\X)^2 }{D_m^{2 \alpha_d} } - \frac{J_{\sigma} \norm{\sigma(X)}_{\infty}^2}{2D_m} . \]
\end{lemma}
Lemma~\ref{RP.le.minQmp} is proved in the technical appendix \cite{Arl:2008b:app}.
\begin{remark}
Since $\norm{\sigma(X)}_2 >0$ and $\sigma$ is piecewise Lipschitz, $\norm{\sigma}_{L^2(\Leb)}>0$. Thus, the lower bound on $\Qmp$ is positive when $D_m$ is large enough.
\end{remark}

\subsubsection{Unbounded data}
We still use Lemma~\ref{RP.le.gal}, but the proof is a little longer and requires the following Lemma~\ref{RP.le.hypAdel.suff} which is proved in the technical appendix \cite{Arl:2008b:app}.
\begin{lemma} \label{RP.le.hypAdel.suff}
Assume that $\X \subset \R$ is bounded and the following:
\makeatletter
\setlength\leftmargini   {36\p@}
\makeatother
\begin{enumerate}
\item[\hypAlip] $\exists B, B_0, c_J \,{>}\,0$ such that $\bayes\,{:}\, \X \,{\flens}\, \R$ is $B$-Lipschitz, piecewise $C^1$ and non-constant \textup{(}that is, $\pm \bayes^{\prime} \geq B_0$ on some interval $J \subset \X$
with $\Leb(J) \,{\geq}\, c_J$\textup{)}.
\item[\hypAreg] Regularity of the partitions for $\Leb$: $\exists \crLl, \crLu>0$ such that
\begin{equation*}
\forall \mM_n, \, \forall \lamm, \quad \crLl D_m^{-1} \leq \Leb(\Il) \leq \crLu D_m^{-1} .
\end{equation*}
\item[\hypAd] Density bounded from below: $\exists c_X^{\min}>0$, $\forall I \subset \X$, $P(X \in I) \geq c_X^{\min} \Leb(I) $.
\end{enumerate}
\makeatletter
\setlength\leftmargini   {22\p@}
\makeatother

Then, \hypAdel\ holds true, that is, for every model $S_m$ of dimension $D_m \geq D_0$, \begin{gather*} \norm{\bayes-\bayes_m}_{\infty} \leq \cdelmglo \norm{\bayes(X) - \bayes_m(X)}_2 \\ \mbox{with} \quad \cdelmglo = \paren{\frac{\crLu}{\crLl}}^{3/2} \frac{B \sqrt{24}}{B_0 \sqrt{c_X^{\min} c_J}} \quad \mbox{and} \quad D_0 \egaldef 4 \crLu c_J^{-1} .
\end{gather*}
\end{lemma}

\paragraph{Pathwise oracle inequality}
We prove that \eqref{RP.eq.oracle_traj_non-as} holds with probability $1 - K_1 n^{-\gamma_0}$ for a general $\gamma_0$, since it will be required for proving a classical oracle inequality below.
First, \hypAml, \hypQ\ and \hypAgeps\ hold since for every $\mM_n$
\begin{align*}
P^{\ell}_m(q) &= \frac{\sqrt{ \sum_{\lamm} m_{q,\lambda}^4}} {\sqrt{D_m} \Qmp }
\leq \frac{ \paren{ 2A + \cgauss \sqrt{q} \sigmax }^2 }
{ \Qmp } \leq \frac{q L_{\cgauss,\sigmax,A}}{\Qmp}
\\
\Qmp &\geq \sigmin^2 \\
P^{g \epsilon} (q) &\leq \sigmax \cgauss \sqrt{q} .
\end{align*}

Second, Lemma~\ref{RP.le.hypAdel.suff} (with \hypAlip, \hypAreg\ and \hypAd) shows that \hypAdel\ holds with $ \cdelmglo = L_{\hypUg}$ and $D_0 = L_{\hypUg}$.

\paragraph{Classical oracle inequality}
Let $\Omega_n$ be the event on which \eqref{RP.eq.oracle_traj_non-as} holds true with $\gamma_0 = 6+\aM$. As in the bounded case, it suffices to upper
bound
\begin{align*}
\El \croch{ \perte{\ERM_{\mh}} \un_{\Omega_n^c} } &\leq \sqrt{\Prob(\Omega^c)} \sqrt{\El\left[ \perte{\ERM_{\mh}}^2\right]}
\qquad \mbox{by Cauchy-Schwarz} \\[3pt]
&\leq \sqrt{K_1} n^{-\gamma_0/2}
\sqrt{\El\croch{ 2 \norm{\bayes}_{\infty}^2 + 2 p_1(\mh)^2 }}  \\[3pt]
&\leq L_{\hypUg} n^{-\gamma_0/2} \croch{ 1 + \sqrt{\El\crochbb{ \sum_{\mM_n} p_1(m)^2 \un_{\mMh_n} } } } .
\end{align*}
For every $\mMh_n$, a bound on $\El[(p_1(m))^2]$ is required.
Starting from \eqref{VFCV.eq.p1.hist},
\begin{align*}
\El\left[ p_1(m)^2 \right] &= \frac{1}{n^2} \sum_{\lamm}
\left( \frac{\pl}{\phl} \right)^2 \El \biggl[ \frac{S_{\lambda,1}^4}{(n\phl)^2}\biggr] + \frac{1}{n^2} \sum_{\lambda \neq \lambda^{\prime}} \left[ \frac{\pl} {\phl} \frac{p_{\lambda^{\prime}}} {\ph_{\lambda^{\prime}}} m_{2,\lambda}^2 m_{2,\lambda^{\prime}}^2 \right] \\[3pt]
&\leq \sum_{\lamm} \El \biggl[ \frac{S_{\lambda,1}^4}{(n\phl)^2} \biggr] + \sum_{\lambda \neq \lambda^{\prime}} \paren{ \sigmax^2 + (2A)^2 }^2
\\[3pt]
& \leq D_m^2 L_{\hypUg} \leq n^2 L_{\hypUg}
\end{align*}
since
\begin{gather*}
\El \left[ \frac{S_{\lambda,1}^4}{(n\phl)^2} \right] = \El \left[ \frac{ \left( \sum_{X_i \in \Il} (Y_i - \betl) \right)^4 }{(n\phl)^2} \right]
= \frac{m_{4,\lambda}^4}{n\phl} + \frac{6 (n\phl - 1) m_{2,\lambda}^4}{n\phl} \\[3pt]
\mbox{and} \quad D_m \sum_{\lamm} m_{q,\lambda}^4 \leq \paren{ a_{\ell} q^{\xi_{\ell}}}^2 \paren{ \sigmax^2 + (2A)^2 }^2 .
\end{gather*}

Hence, using that $\card(\M_n) \leq \cM n^{\aM}$,
\[ \El \croch{ \perte{\ERM_{\mh}} \un_{\Omega_n^c} } \leq L_{\hypUg} n^{1 + (\aM - \gamma_0)/2} \]
which proves \eqref{RP.eq.oracle_class_non-as}. \qed

\subsection[Proof of Theorem~2]{Proof of Theorem~\ref{RP.the.holder}} \label{RP.sec.proof.holder}
In this proof, \hypH\ denotes the set of assumptions made in Theorem~\ref{RP.the.holder}. \hypH\ implies all the assumptions of Theorem~\ref{RP.the.oracle_traj_non-as} except maybe the lower bound in \hypAp; indeed, \hypAd\ and the fact that all the models are ``regular'' imply \hypArXl.
Therefore, we can start from \eqref{RP.eq:oracle_non-as_x_traj.sansAp} in Remark~\ref{RP.rem.oracle.traj.sansAp} below Lemma~\ref{RP.le.gal} which does not require the lower bound in \hypAp\ to hold.
The constants $\gamma_i$ are absolute because the data are bounded.

Let $m(T_0) \in \M_n$ be the model of dimension $T_0^k$ closest to $R^{\frac{2k}{2\alpha + k}} n^{\frac{k}{2\alpha +k}} \sigmax^{\frac{-2k}{2\alpha +k}}$.
By definition of $T_0$ and $\M_n$,
\[ 2^{-1} R^{\frac{2}{2\alpha + k}} n^{\frac{1}{2\alpha +k}} \sigmax^{\frac{-2}{2\alpha +k}} \leq T_0 \leq 2 R^{\frac{2}{2\alpha + k}} n^{\frac{1}{2\alpha +k}} \sigmax^{\frac{-2}{2\alpha +k}} . \]
If $n \geq L_{\hypH,c}$, $T_0^k$ is larger than $\paren{\ln(n)}^{\gamma_1}$ and smaller than $c n \paren{\ln(n)}^{-1}$ . Hence, from the proof of Lemma~\ref{RP.le.gal}, $m(T_0) \in \Mh_n$ and $m(T_0)$ has a finite excess loss on the large probability event of Lemma~\ref{RP.le.gal}.
Moreover, \[ \perte{\ERM_{m(T_0)}} \leq \perte{\bayes_{m(T_0)}} + L \E\croch{\punzero(m(T_0))} \] when $n \geq L_{\hypH}$.
Since $\perte{\bayes_{m(T_0)}} \leq R^2 T_0^{-2\alpha}$ and
\begin{align*}
\E \croch{\punzero(m(T_0))} &\leq \paren{ \sup_{np \geq 0} \einvz{\mathcal{B}(n,p)}} \frac{1}{n} \sum_{\lambda \in \Lambda_{m(T_0)}} \paren{ \carre{ \sigla} + \carre{\sigld} }
\\ &\leq \frac{2 R^2 T_0^{1-2\alpha}}{n} + \frac{2 \sigmax^2 D_{m(T_0)}}{n}
\end{align*}
(the bound $\einvz{\mathcal{B}(n,p)} \leq 2$ coming from \cite[Lemma~4.1]{Gyo_etal:2002}),
an event of probability at least $1 - K_1^{\prime} n^{-2}$ exists on which
\[ \perte{\ERM_{\mh}} \leq K_2 R^{\frac{2k}{2\alpha +k}} n^{\frac{-2\alpha }{2\alpha +k}} \sigmax^{\frac{4\alpha }{2\alpha+k}} + \frac{\paren{\ln(n)}^{\gamma_2}}{n} , \]
where $K_2$ may only depend on $k$ and $\alpha$.
Note that the constant $K_1$ has been replaced by $K_1^{\prime} \geq K_1$ so that the probability bound $1 - K_1^{\prime} n^{-2}$ is nonpositive when $n$ is too small.
Enlarging $K_1^{\prime}$ once more, the term $\paren{\ln(n)}^{\gamma_2} n^{-1}$ can be dropped off by adding 1 to the constant $K_2$.
Then, taking expectations as in the proof of Theorem~\ref{RP.the.oracle_traj_non-as}, \eqref{RP.eq.oracle_class_non-as.hold} holds.

When \hypAsig\ holds, $\sigmax$ can be replaced by $\norm{\sigma}_{L^2(\Leb)}$ in the definition of $m(T_0)$. Then, for every $\lambda\in\Lambda_{m(T_0)}$ such that $\sigma$ does not jump on $\Il$,
\begin{align*} \carre{\sigla} \leq \max_{\Il} \sigma^2 &\leq \paren{ \frac{K_{\sigma} }{T_0} + \sqrt{ \int_{\X} \sigma^2(t) \Leb(dt)} }^2
\\
&\leq \paren{1+\theta^{-1}} \frac{K_{\sigma}^2 }{T_0^{2}} + (1+\theta) \int_{\X} \sigma^2(t) \Leb(dt)
\end{align*}
for every $\theta>0$ (since $\Leb(\X)=1$).
If $\sigma$ jumps on $\Il$ (and there exist at most $J_{\sigma}$ such $\lambda$), $\max_{\Il} \sigma^2 \leq \sigmax^2$.
Hence, taking $\theta=T_0^{-1}$,
\begin{align*}
\E \croch{\punzero(m(T_0))} &\leq \frac{2}{n} \paren{R^2 T_0^{1-2\alpha} + \sum_{\lambda\in\Lambda_{m(T_0)}} \carre{\sigla} } \\
&\leq \frac{2 R^2 T_0^{1-2\alpha}}{n} + \frac{2 D_{m(T_0)} \norm{\sigma}_{L^2(\Leb)}^2}{n} + \frac{L_{\hypH}} {n}
\end{align*}
and the end of the proof does not change. In this second case, \hypAn\ can also be removed because all the assumptions stated in the first part of Section~\ref{RP.sec.main.hyp.alt} are satisfied. \qed

\subsection{Additional probabilistic tools} \label{RP.sec.proof.tools}
Several probabilistic results are needed in addition to the ones of Section~\ref{RP.sec.tools} for proving Lemma~\ref{RP.le.gal}.
First, Proposition~\ref{RP.pro.conc.penid} below deals with concentration properties of $p_1$ and $p_2$.
Remark that concentration inequalities for $p_2$ can be obtained in a general framework \cite[Proposition~10]{Arl_Mas:2008}. On the contrary, we do not know any other non-asymptotic bound on the two-sided deviations of $p_1$.

\begin{proposition} \label{RP.pro.conc.penid}
Let $\gamma>0$ and $S_m$ be the model of histograms associated with some partition $(\Il)_{\lamm}$ of $\X$. 
Assume that $\min_{\lamm} \set{n \pl} \geq B_n $ and that positive constants $a_{\ell}$, $\xi_{\ell}$ exist such that \hypAml\ $\forall q \geq 2$, $P^{\ell}_m(q) \leq a_{\ell} q^{\xi_{\ell}}$.
Then, if $B_n \geq 1$, an event of probability at least $1 - L n^{-\gamma}$ exists on which
\begin{align} \notag 
\punmin(m) &\geq \E \croch{\punmin(m)} - L_{a_{\ell},\xi_{\ell},\gamma} \croch{ \frac{\paren{\ln(n)}^{\xi_{\ell} + 2}}{\sqrt{D_m}} + e^{-L B_n} } \E\croch{p_2(m)}
\\
\label{RP.eq.conc.p1.maj}
\punmin(m) &\leq \E \croch{\punmin(m)} + L_{a_{\ell},\xi_{\ell},\gamma} \croch{ \frac{\paren{\ln(n)}^{\xi_{\ell} + 2}}{\sqrt{D_m}} + \sqrt{D_m} e^{-L B_n} } \E\croch{p_2(m)}
\end{align}
\begin{equation*}
\absj{p_2(m) - \E [p_2(m)]} \leq L_{a_{\ell},\xi_{\ell},\gamma} \frac{\paren{\ln(n)}^{\xi_{\ell} + 1}}{\sqrt{D_m}} \E\croch{p_2(m)} .
\end{equation*}

Moreover, if $B_n > 0$, an event of probability at least $1-Ln^{-\gamma}$ exists on which
\begin{equation}
\label{RP.eq.conc.p1.min.2}
\punmin(m) \geq \paren{\frac{1}{2 + \frac{(\gamma+1) \ln(n)}{B_n} } - L_{a_{\ell},\xi_{\ell},\gamma} \croch{ \frac{\paren{\ln(n)}^{\xi_{\ell} + 2}}{\sqrt{D_m}} + e^{-L B_n} } } \E\croch{\widetilde{p_2}(m)} .
\end{equation}
\end{proposition}
Proposition~\ref{RP.pro.conc.penid} is proved in \cite{Arl:2008b:app}.
Second, Lemmas~\ref{RP.le.conc.delta.borne} and~\ref{RP.le.conc.delc.nonborne} below provide concentration inequalities for $\delc(m)$, when the data are either bounded or unbounded.
\begin{lemma} \label{RP.le.conc.delta.borne}
Assume that $\norm{Y}_{\infty} \leq A < \infty$.
Recall that for every $\mM_n$, $\delc(m) = (P_n - P) \paren{\gamma\paren{\bayes_m} - \gamma\paren{\bayes}}$.
Then for every $x \geq 0$, an event of probability at least $1 - 2 e^{-x}$ exists on which
\begin{equation}
\label{RP.eq.conc.delta.borne.2} \forall \eta >0, \quad \absj{\delc(m)} \leq \eta \perte{\bayes_m} + \left( \frac{4}{\eta} + \frac{8}{3} \right) \frac{A^2 x}{n}
.
\end{equation}
In particular,
\begin{equation} \label{RP.eq.conc.delta.borne} \absj{\delc(m)} \leq \frac{\perte{\bayes_m}}{\sqrt{D_m}} + \frac{20}{3} \frac{A^2 }{ \Qmp } \frac{\El[p_2(m)]}{\sqrt{D_m} } x .\end{equation}
\end{lemma}
\begin{proof}[Proof of Lemma~\ref{RP.le.conc.delta.borne}]
\eqref{RP.eq.conc.delta.borne.2} essentially relies on Bernstein's inequality and is proved in details in \cite[Proposition~8]{Arl_Mas:2008}.
Then, \eqref{RP.eq.conc.delta.borne} follows from \eqref{RP.eq.conc.delta.borne.2} with $\eta=D_m^{-1/2}$ and the definition of $\Qmp$.
\end{proof}
\begin{lemma} \label{RP.le.conc.delc.nonborne}
Assume that positive constants $a_{g \epsilon}$, $\xi_{g \epsilon}$, $\sigmax$ and $\cdelmglo$ exist such that
\makeatletter
\setlength\leftmargini   {40\p@}
\makeatother
\begin{enumerate}
\item[\hypAgeps] $\forall q \geq 2$, $P^{g \epsilon} (q) \leq a_{g \epsilon} q^{\xi_{g \epsilon}} ,$
\item[\hypAsigmax] $\norm{\sigma(X)}_{\infty} \leq \sigmax ,$
\item[\hypAdel] $\norm{\bayes-\bayes_m}_{\infty} \leq \cdelmglo \norm{\bayes(X) - \bayes_m(X)}_2 .$
\end{enumerate}
\makeatletter
\setlength\leftmargini   {22\p@}
\makeatother
Then, for every $x \geq 0$, an event of probability at least $1 - e^{-x}$ exists on which
\begin{equation} \label{RP.eq.conc.delc.nonborne}
\absj{\delc(m)} \leq \frac{L_{ a_{g \epsilon}, \xi_{g \epsilon}, c_{\Delta,m}^{g} } x^{\xi_{g \epsilon} + 1/2 }}{\sqrt{D_m}} \croch{ \perte{\bayes_m} + \frac{\sigmax^2 }{\Qmp} \E\croch{p_2(m)} } .
\end{equation}

Moreover, if \hypAgeps\ and \hypAsigmax\ holds true, but \hypAdel\ is replaced by \hypAciblemax\ $\norm{\bayes}_{\infty} \leq A$, then, for every $x \geq 0$, an event of probability at least $1 - e^{-x}$ exists on which
\begin{equation} \label{RP.eq.conc.delc.nonborne.2}
\absj{\delc(m)} \leq L_{ a_{g \epsilon}, \xi_{g \epsilon}, A, \sigmax } n^{-1/2} x^{\xi_{g \epsilon} + 1/2} .
\end{equation}
\end{lemma}
Lemma~\ref{RP.le.conc.delc.nonborne} is proved in Section~\ref{RP.sec.proof.conc}.
Third, Lemma~\ref{RP.le:threshold} ensures that empirical frequencies $n \phl$ are not too far from the expected ones $n \pl$.
\begin{lemma} \label{RP.le:threshold}
Let $(\pl)_{\lamm}$ be non-negative real numbers of sum 1, $(n\phl)_{\lamm}$ be a multinomial vector of parameters
$(n;(\pl)_{\lamm})$ and $\gamma>0$. Assume that $\card(\Lambda_m) \leq n$ and $\min_{\lamm} \set{n \pl} \geq B_n >0$. Then, an event of probability at least $1 - L n^{-\gamma}$ exists on which
\begin{align}
\label{RP.eq.threshold}
\min_{\lamm} \set{n \phl} \geq \frac{\min_{\lamm} \set{n \pl}}{2} - 2 (\gamma+1) \ln(n) .
\end{align}
\end{lemma}
\begin{proof}[Proof of Lemma~\ref{RP.le:threshold}]
First, for every $\lamm$, Bernstein's inequality \cite[Proposition~2.9]{Mas:2003:St-Flour} applied to $n \phl$ shows that an event of probability at least $1 - 2n^{- (\gamma+1)}$ exists on which
\[ n \phl \geq n\pl - \sqrt{2 n \pl (\gamma+1) \ln(n)} - \frac{(\gamma+1) \ln(n)}{3} . \]
Since $\sqrt{2n \pl (\gamma +1) \ln(n) } \leq (n \pl)/2 + (\gamma +1) \ln(n) $,
\eqref{RP.eq.threshold} holds on an event of probability at least $1 - 2 \card(\Lambda_m) n^{- (\gamma+1)} \geq 1 - 2 n^{- \gamma}$.
\end{proof}


Finally, Lemmas~\ref{VFCV.le.calc.p2} and~\ref{RP.le.CWinf.minor} below are useful to compare the expectations of $p_1$ and $p_2$ on the one hand, and the expectations of $\pen$ and $\penid$ for possibly large models on the other hand.
\begin{lemma}[Lemma~7 of \cite{Arl:2008a}]\label{VFCV.le.calc.p2}
If $\min_{\lamm} \set{n \pl} \geq B \geq 1$,
\begin{equation*} 
\paren{ 1 - e^{-B} } \E \croch{\widetilde{p}_2(m)} \leq \E \croch{\punzero(m)} \leq \E \croch{\punmin(m)} \leq \paren{1 + \sup_{np \geq B} \delta_{n,p} } \E \croch{\widetilde{p}_2(m)}
\end{equation*}
where $\delta_{n,p}$ is the same as in \eqref{RP.eq.Ep2.Epenid}. A similar result holds with $p_2$ instead of $\widetilde{p}_2$ inside the expectation.
\end{lemma}
\begin{lemma}\label{RP.le.CWinf.minor}
Assume that $W$ is a weight vector among Efr, Rad, Poi, Rho and Loo. Let $S_m$ be the model of histograms associated with the partition $(\Il)_{\lamm}$, $p_2(m) = P_n \paren{ \gamma(\bayes_m) - \gamma(\ERM_m)}$ and $\pen(m)$ be defined by \eqref{RP.def.pen.his} with $C=\CWinf$ (see Table~\ref{RP.TableR2}). Then, if $\min_{\lamm} \set{n\phl} \geq 3$, %
\begin{equation} \label{RP.eq.CWinf.minor}
\El \croch{\pen(m)} \geq \frac{5}{4} \El \croch{p_2(m)} .
\end{equation}
If $\min_{\lamm} \set{n\phl} \geq T$ for some positive $T$, \eqref{RP.eq.CWinf.minor} still holds for weight vectors among:
\begin{itemize}
\item Efr($\mnefr$) when $\mnefr n^{-1} \geq - T^{-1} \ln(3/4 - 2/T)$
\item Rad($p$) when $T \geq p^{-1} \ln[8/(3(1-p))]$
\item Poi($\mu$) when $T \geq 3$ and $\mu T \geq 1.61$
\item Rho($q_n$) when $T \geq n q_n^{-1} \ln[(4n)/(3(n-q_n))]$.
\end{itemize}
\end{lemma}
Lemma~\ref{RP.le.CWinf.minor} is proved in Section~\ref{RP.sec.calc.reech}.

\subsection[Proof of Lemma~7]{Proof of Lemma~\ref{RP.le.gal}} \label{RP.sec.proof.le.gal}
We first give the complete proof in the bounded case.
Then, we will explain how it can be extended to the unbounded case.

\subsubsection{Bounded case} \label{RP.sec.proof.le.gal.bounded}
For every $\mM_n$, define
\[ \penid^{\prime}(m) \egaldef p_1(m) + p_2(m) - \delc(m) = \penid(m) - (P - P_n) \gamma(\bayes) . \]
By definition of $\penid^{\prime}$ and $\mh$, for every $\mMh_n$,
\begin{equation}
\label{RP.eq.oracle.pr.1}
\perte{\ERM_{\mh}} - \paren{\penid^{\prime}(\mh) - \pen(\mh)} \leq \perte{\ERM_m} + \paren{\pen(m) - \penid^{\prime}(m)} . \end{equation}
The proof of Lemma~\ref{RP.le.gal} is divided into three main parts:
\begin{enumerate}
\item With a large probability, $\pen - \penid^{\prime}$ is negligible in front of $\perte{\ERM_m}$ uniformly over models $S_m$ of ``intermediate'' dimension, that is $\paren{\ln(n)}^{\gamma_1} \leq D_m \leq c n \paren{\ln(n)}^{-1}$ for some constants $c,\gamma_1>0$. This relies on the concentration inequalities and comparisons of expectations stated in Sections~\ref{RP.sec.tools} and~\ref{RP.sec.proof.tools}.
\item The model $\mh$ selected by Resampling Penalization has an ``intermediate'' dimension.
In order to prove this, a lower bound on \[ \crit^{\prime\prime}(m) \egaldef P_n \gamma\paren{\ERM_m} + \pen(m) - P_n \gamma\paren{\bayes} \] is proved for large and small models, and this bound is showed to be larger than $\crit^{\prime\prime}(m_0)$, where $S_{m_0}$ is the model of intermediate dimension belonging to the collection $\paren{S_m}_{\mM_n}$ according to assumption \hypPrich. Lemma~\ref{RP.le.CWinf.minor} is crucial at this point.
\item The oracle model (that is the one minimizing $\perte{\ERM_m}$) is also of ``intermediate'' dimension, which is proven similarly to point 2 with $\crit^{\prime\prime}(m)$ replaced by $\perte{\ERM_m}$.
\end{enumerate}

For every $\mM_n$, define
\[ A_n(m) \egaldef \min_{\lamm} \set{ n \phl } \qquad \mbox{and} \qquad B_n(m) = \min_{\lamm} \set{n\pl} . \]
Let $\Omega_{n,\gamma_0}$ be the event on which the concentration inequalities of Propositions~\ref{RP.pro.conc.penRP} and~\ref{RP.pro.conc.penid} and Lemmas~\ref{RP.le.conc.delta.borne} and~\ref{RP.le:threshold} hold for every $\mM_n$ with $\gamma = \aM+\gamma_0$ (or similarly $x = (\aM+\gamma_0) \ln(n)$ in Lemma~\ref{RP.le.conc.delta.borne}). Using assumption \hypPpoly, the union bound gives $\Prob\paren{\Omega_{n,\gamma_0}} \geq 1 - L_{\cM} n^{-\gamma_0}$.

\paragraph{1. $\pen$ is close to $\penid^{\prime}$ for intermediate models}
Let $c,\gamma_1>0$ be two constants to be chosen later, and consider $\widetilde{\M}_n$, the set of $\mM_n$ such that $\paren{\ln(n)}^{\gamma_1} \leq D_m \leq c n \paren{\ln(n)}^{-1}$.
According to \hypArXl, for every $m \in \widetilde{\M}_n$, $B_n(m) \geq \crXl c^{-1} \ln(n)$ so that \eqref{RP.eq.threshold} ensures that $A_n(m) \geq \ln(n)$ on $\Omega_{n,\gamma_0}$ if $c \leq L_{\crXl,\aM,\gamma_0}$. In particular, $\widetilde{\M}_n \subset \Mh_n$ on $\Omega_{n,\gamma_0}$.

Assume also that $n \geq \exp(D_0)$, so that $D_m \geq D_0$ for every $m \in \widetilde{\M}_n$ if $\gamma_1 \geq 1$.
Now, using both bounds on $D_m$,
\begin{equation*} \begin{split}
\max \left\{ \absj{\punmin(m) - \E\croch{\punmin(m)}} , \absj{p_2(m) - \E\croch{p_2(m)}} , \right. \\
\left. \absj{\delc(m)} , \absj{\pen(m) - \El\croch{\pen(m)}} \right\} \end{split}
\end{equation*} is smaller than $L_{\hypBg} \paren{\ln(n)}^{-1} \paren{ \perte{\bayes_m} + \E\croch{p_2(m)}}$ on $\Omega_{n,\gamma_0}$ provided that $c \leq L_{\crXl,\gamma}$ (to ensure that $B_n(m)$ is large enough) and $\gamma_1 \geq 2 \xi_{\ell} + 6$.
Fix now $c=L_{\crXl,\gamma}>0$ and $\gamma_1 = L_{\xi_{\ell}}$ satisfying these conditions.
Using Proposition~\ref{RP.pro.comp.Epen.Ep2}, Lemma~\ref{VFCV.le.calc.p2} and the lower bound on $B_n(m)$, we have for every $m \in \widetilde{\M}_n$
\[ \frac{- L_{\hypBg}}{\paren{\ln(n)}^{1/4}} \perte{\ERM_m} \leq (\pen-\penid^{\prime})(m) \leq \croch{ 2 (\eta - 1) + \frac{L_{\hypBg}}{\paren{\ln(n)}^{1/4}} }\perte{\ERM_m} .\]
as soon as $n \geq L_{\hypBg}$ (this restriction is necessary because the bounds are in terms of $\perte{\ERM_m}$ instead of $\perte{\bayes_m} + \E\croch{p_2}$).
Combined with \eqref{RP.eq.oracle.pr.1}, this gives: if $n \geq L_{\hypBg}$
\begin{equation} \label{RP.eq.oracle.pr.2}
\perte{\ERM_{\mh}} \un_{\mh \in \widetilde{\M}_n} \leq \croch{ 2 \eta - 1 + \frac{L_{\hypBg}}{\paren{\ln(n)}^{1/4}} } \inf_{m \in \widetilde{\M}_n} \set{ \perte{\ERM_m} } .
\end{equation}

\paragraph{2. $\mh$ has an ``intermediate'' dimension}
The penalized empirical criterion $\crit(m)=P_n\gamma\paren{\ERM_m} + \pen(m)$ has the same minimizers as
\[ \crit^{\prime\prime}(m) = \perte{\ERM_m} + \pen(m) - \penid^{\prime}(m) = \perte{\bayes_m} + \pen(m) - p_2(m) + \delc(m) \]
 over $\Mh_n$.

According to \hypPrich, there exists $m_0 \in \M_n$ such that $\sqrt{n} \leq D_{m_0} \leq c_{\mathrm{rich}} \sqrt{n}$.
If $n \geq L_{\hypBg}$, $m_0 \in \widetilde{\M}_n$ so that (using \hypAp\ and the same inequalities as in the first part of the proof)
\begin{equation} \label{eq.penRP.critm0}
\crit^{\prime\prime}(m_0) \leq \perte{\bayes_{m_0}}
+ \absj{\delc(m_0)} + \pen(m_0) \leq L_{\hypBg} \paren{n^{-\beta_2/2} + n^{-1/2}} . \end{equation}
Therefore, it remains to provide lower bounds on $\crit^{\prime\prime}(m)$ for $m \notin \widetilde{\M}_n$.

On the one hand, on $\Omega_{n,\gamma_0}$ if $D_m < \paren{\ln(n)}^{\gamma_1}$,
\begin{align} \notag
\crit^{\prime\prime}(m) &\geq \perte{\bayes_m} - \absj{\delc(m)} - p_2(m) \\
\label{eq.penRP.critm-pt} &\geq \cbiasmin \paren{\ln(n)}^{- \gamma_1 \beta_1} - L_{A,\gamma_0} \sqrt{\frac{\ln(n)}{n}} - L_{\hypBg} \frac{\paren{\ln(n)}^{1+\xi_{\ell}+\gamma_1}}{n} . \end{align}
On the other hand, if $D_m > c n \paren{\ln(n)}^{-1}$ and $\mMh_n$, by Lemma~\ref{RP.le.CWinf.minor}, $\El\croch{\pen(m) - p_2(m)} \geq \El\croch{p_2(m)} / 4$. Therefore, we have $\pen(m) - p_2(m) \geq (1 - L_{\hypBg} n^{-1/4}) \E\croch{p_2(m)}$ on $\Omega_{n,\gamma_0}$, so that
\begin{equation} \label{eq.penRP.critm-grd} \crit^{\prime\prime}(m) \geq \pen(m) - p_2(m) - \absj{\delc(m)} \geq L_{\hypBg} \paren{\ln(n)}^{-1} \end{equation}
when $n \geq L_{\hypBg}$.
Comparing \eqref{eq.penRP.critm0}, \eqref{eq.penRP.critm-pt} and \eqref{eq.penRP.critm-grd}, it follows that any minimizer $\mh$ of $\crit$ over $\Mh_n$ belongs to $\widetilde{\M}_n$ on $\Omega_{n,\gamma_0}$, provided that $n \geq L_{\hypBg}$.

\paragraph{3. the oracle has an ``intermediate'' dimension}
It remains to prove that the infimum can be extended to $\M_n$ on the right-hand side of \eqref{RP.eq.oracle.pr.2}, with the convention $\perte{\ERM_m} = + \infty$ if $A_n(m) = 0$.
Using similar arguments as above (as well as the definition of $\Omega_{n,\gamma_0}$, in particular \eqref{RP.eq.conc.p1.min.2} for large models), we have $\perte{\ERM_{m_0}} \leq L_{\hypBg} \paren{n^{-\beta_2/2} + n^{-1/2}}$ on $\Omega_{n,\gamma_0}$.
Moreover, for every $m \notin \widetilde{\M}_n$, either $D_m < \paren{\ln(n)}^{\gamma_1}$ and $\perte{\ERM_m} \geq \perte{\bayes_m} \geq L_{\hypBg} \paren{\ln(n)}^{-\gamma_1 \beta_1}$
or $D_m > c n \paren{\ln(n)}^{-1}$ and $\perte{\ERM_m} \geq p_1(m) \geq L_{\hypBg} \paren{\ln(n)}^{-2}$ on $\Omega_{n,\gamma_0}$ by \eqref{RP.eq.conc.p1.min.2} as soon as $n \geq L_{\hypBg}$.
Hence, if $n \geq L_{\hypBg}$, $m \notin \widetilde{\M}_n$ cannot contribute to the infimum in the right-hand side of \eqref{RP.eq.oracle.pr.2}. This concludes the proof of
\eqref{RP.eq:oracle_non-as_x_traj} in the bounded case. \qed

\subsubsection{Unbounded case}

The proof of the bounded case has to be slightly modified. In the definition of $\Omega_{n,\gamma_0}$, the concentration inequalities of Lemma~\ref{RP.le.conc.delta.borne} are replaced by those of Lemma~\ref{RP.le.conc.delc.nonborne}.
Then, $\gamma_1$ has to be chosen such that $\gamma_1 \geq 2 \xi_{g \epsilon} + 3$. The rest of the proof of \eqref{RP.eq.oracle.pr.2} is unchanged.

In order to prove that $\mh \in \widetilde{\M}_n$, \eqref{eq.penRP.critm-pt} has to be slightly changed because of the use of \eqref{RP.eq.conc.delc.nonborne.2} instead of \eqref{RP.eq.conc.delta.borne.2} to bound $\delc(m)$.
The final part of the proof is then modified similarly. \qed

\subsubsection[Proof of Remark~8]{Proof of Remark~\ref{RP.rem.oracle.traj.sansAp}} \label{sec.proof.rem.oracle.traj.sansAp}
We now prove the assertion made in Remark~\ref{RP.rem.oracle.traj.sansAp} below Lemma~\ref{RP.le.gal}.
Starting from \eqref{RP.eq.oracle.pr.2}, we can prove in the same way that $D_{\mh} \leq c n \paren{\ln(n)}^{-1}$, but $D_{\mh} < \paren{\ln(n)}^{\gamma_1}$ cannot be excluded.

Let $m \in \Mh_n$ such that $D_m < \paren{\ln(n)}^{\gamma_1}$.
Assume first that
\begin{equation} \label{eq.Th1.sansAp.min.biais} \perte{\bayes_m} \geq
\frac{2\eta - 1 + \varepsilon_n}{1 - \paren{\ln(n)}^{-1}} \inf_{m \in \widetilde{\M}_n} \set{\perte{\ERM_m}}
 + \frac{\paren{\ln(n)}^{\xi_{\ell}+\gamma_1 + 2}}{\parenb{1 - \paren{\ln(n)}^{-1}} n} , \end{equation}
where $\varepsilon_n \leq L_{\hypBg}\paren{\ln(n)}^{-1/4}$ comes from \eqref{RP.eq.oracle_traj_non-as}.
Then, on $\Omega_{n,\gamma_0}$, using \eqref{RP.eq.conc.delta.borne.2} with $\eta = \paren{\ln(n)}^{-1}$ and \eqref{eq.Th1.sansAp.min.biais},
\begin{align} \notag
\crit^{\prime\prime}(m) &\geq \perte{\bayes_m} - \absj{\delc(m)} - p_2(m)
\\ \notag &\geq (2\eta - 1 + \varepsilon_n) \inf_{m \in \widetilde{\M}_n} \set{\perte{\ERM_m}} + \frac{\paren{\ln(n)}^{\xi_{\ell}+\gamma_1+1} \paren{\ln(n) - L_{\hypBg}}}{n}
\\ &\geq (2\eta - 1 + \varepsilon_n) \inf_{m \in \widetilde{\M}_n} \set{\perte{\ERM_m}} + \frac{\paren{\ln(n)}^{\xi_{\ell}+\gamma_1+2}}{2n} ,
\label{eq.penRP.critm-pt.sansAp} \end{align}
provided that $n \geq L_{\hypBg}$.
In addition, let $ m_0 \in \arg\min_{m^{\prime} \in \widetilde{\M}_n} \set{\perte{\ERM_{m^{\prime}}} } $.
Since $m_0 \in \widetilde{\M}_n$, on $\Omega_{n,\gamma_0}$,
\[ \crit^{\prime\prime}(m_0) = \perte{\ERM_{m_0}} + \pen(m_0) - \penid^{\prime}(m_0) \leq \paren{2 \eta - 1 + \varepsilon_n } \perte{\ERM_{m_0}} , \]
and this upper bound is smaller than the lower bound in \eqref{eq.penRP.critm-pt.sansAp}.

Hence, on $\Omega_{n,\gamma_0}$, if $D_{\mh} < \paren{\ln(n)}^{\gamma_1}$ \eqref{eq.Th1.sansAp.min.biais} cannot be satisfied with $m = \mh$.
Moreover, by \eqref{RP.eq.conc.p1.maj}, for every $\mM_n$ such that $D_m \leq c n \paren{\ln(n)}^{-1}$
\[ \punmin(m) \leq L_{\hypBg} \paren{\ln(n)}^{\xi_{\ell} + 2} \frac{D_m}{n}\] on $\Omega_{n,\gamma_0}$.
Therefore,
\begin{align} \notag
\perte{\ERM_{\mh}} &= \perte{\bayes_{\mh}} + \punmin(\mh) \\
&\leq \frac{2\eta - 1 + \varepsilon_n}{1 - \paren{\ln(n)}^{-1}} \inf_{m \in \widetilde{\M}_n} \set{\perte{\ERM_m}} + L_{\hypBg} \frac{\paren{\ln(n)}^{\xi_{\ell}+\gamma_1 + 2}}{n} \notag \\
&\leq \parenb{2\eta - 1 + \paren{\ln(n)}^{-1/5}}\inf_{m \in \widetilde{\M}_n} \set{\perte{\ERM_m}} + \frac{\paren{\ln(n)}^{\xi_{\ell}+\gamma_1 + 3}}{n}
 \label{eq.Th1.sansAp.mh-pt} \end{align}
assuming that $n \geq L_{\hypBg}$.

When $D_{\mh} \geq \paren{\ln(n)}^{\gamma_1}$, \eqref{RP.eq:oracle_non-as_x_traj} holds on $\Omega_{n,\gamma_0}$ which implies \eqref{eq.Th1.sansAp.mh-pt}.
Hence, \eqref{eq.Th1.sansAp.mh-pt} holds on $\Omega_{n,\gamma_0}$.

Finally, with the same arguments as in Section~\ref{RP.sec.proof.le.gal.bounded}, the infimum on the right-hand side of \eqref{eq.Th1.sansAp.mh-pt} can be extended to the set of $\mM_n$ such that $D_m \geq \paren{\ln(n)}^{\gamma_1}$, with the convention $\perte{\ERM_m} = + \infty$ if $A_n(m) = 0$.
Enlarging the constant $K_1$ to remove the condition $n \geq L_{\hypBg}$, \eqref{RP.eq:oracle_non-as_x_traj.sansAp} is proved to hold with $\gamma_2 = \gamma_1 + \xi_{\ell} + 3$.
The proof is quite similar in the unbounded case.
\qed

\subsection{Expectations} \label{sec.proof.expect}
\begin{proof}[Proof of Proposition~\ref{RP.pro.EpenRP-Epenid}]
On the one hand, \eqref{RP.eq.Em_penid} and \eqref{RP.eq.Ep2.Epenid} are consequences of \eqref{VFCV.eq.p1.hist} and \eqref{VFCV.eq.p2.hist}; note that \eqref{RP.eq.Ep2.Epenid} holds whatever the convention taken for $p_1$ and $p_2$ in Section~\ref{RP.sec.proof.nota}.

On the other hand, \eqref{RP.eq.Em_pen} follows from Lemma~\ref{VFCV.le.cal.pen.Wech} below which is slighlty more geenral since $W$ is allowed to depend on $\paren{\un_{X_i \in \Il}}_{(i,\lambda)}$.
\end{proof}

\begin{lemma}\label{VFCV.le.cal.pen.Wech}
Let $S_m$ be the model of histograms adapted to some partition $\paren{\Il}_{\lamm}$ of $\X$, $W \in [0;\infty)^n$ be a random vector such that for every $\lamm$, $(W_i)_{X_i \in \Il}$ is exchangeable and independent of $(X_i,Y_i)_{X_i \in \Il}$. Let $\pen(m)$ be defined by \eqref{RP.def.pen.his} and assume $\min_{\lamm}\set{n\phl} \geq 1$.
%
Then,
\begin{equation}
\label{VFCV.eq.pen.Wech}
 \pen(m) = \frac{C}{n} \sum_{\lamm} \paren{R_{1,W}(n,\phl) + R_{2,W}(n,\phl)} \frac{n\phl S_{\lambda,2} - S_{\lambda,1}^2 }{n\phl (n\phl - 1)} \un_{n \phl \geq 2} , \end{equation}
where $R_{1,W}$ and $R_{2,W}$ are defined by \eqref{RP.def.R1} and \eqref{RP.def.R2}, that is
\begin{align*}
R_{1,W}(n,\phl) &\egaldef \E\croch{ \frac{(W_1 - \Wl)^2}{\Wl^2} \sachant X_1 \in \Il, \Wl >0 } \\
\mbox{and} \quad R_{2,W}(n,\phl) &\egaldef \E\croch{ \frac{(W_1 - \Wl)^2}{\Wl} \sachant X_1 \in \Il} .
\end{align*}
\end{lemma}
\begin{proof}[Proof of Lemma~\ref{VFCV.le.cal.pen.Wech}]
First, as $\penid(m)$ was split into $p_1(m)$ and $p_2(m)$ (plus a centered term), the resampling penalty (without the constant $C$) is split into two terms:
\begin{align}
\label{RP.def.ph1}
\ph_1(m) &= \sum_{\lamm} \Es \croch{ \phl \paren{\bethlW - \bethl}^2 \sachant \Wl >0} \\[3pt]
\label{RP.def.ph2}
\ph_2(m) &= \sum_{\lamm} \Es \croch{ \phlW \paren{\bethlW - \bethl}^2 } .
\end{align}

A key quantity to compute is the following: for every $\lamm$ and $\Wl>0$,
\begin{align} \notag
& \quad \Es \croch{\phl \carre{\bethlW - \bethl} \sachant \Wl}
\\[3pt] \notag
&= \Es \croch{ \phl \carre{ \frac{1}{n\phl}
\sum_{X_i \in \Il} (Y_i - \betl) \left(1 - \frac{W_i}{\Wl} \right) } \sachant \Wl } \\ \label{RP.eq.calc.ph.int1}
&= \frac{1}{n^2 \phl} 
\sum_{X_i \in \Il} \carre{Y_i - \betl} \Es \croch{ \carre{ 1 -
\frac{W_i}{\Wl} } \sachant \Wl}
\\[3pt]
\notag & \quad + \frac{1}{n^2 \phl} \mathop{\sum_{i \neq j}}_{X_i
\in \Il, X_j \in \Il} (Y_i - \betl) (Y_j - \betl) \Es \croch{
\left( 1 - \frac{W_i}{\Wl} \right) \left( 1 -
\frac{W_j}{\Wl} \right) \sachant \Wl} 
.
\end{align}
Since the weights are exchangeable, $(W_i)_{X_i \in \Il}$ is also exchangeable conditionally on $\Wl$ and $(X_i)_{1 \leq i \leq n}$. Hence, the ``variance'' term
\[ R_V(n,n\phl,\Wl,\loi(W)) \egaldef \Es \croch{ \parenb{W_i - \Wl}^2 \sachant \Wl } \] does not depend on $i$ (provided that $X_i \in \Il$) and the ``covariance'' term
\[ R_C(n,n\phl,\Wl,\loi(W)) \egaldef \Es \croch{ \parenb{W_i - \Wl}\parenb{W_j - \Wl} \sachant \Wl } \] does not depend on $(i,j)$ (provided that $i \neq j$ and $X_i, X_j \in \Il$). Moreover,
\begin{align*} 0 &= \Es \croch{ \paren{ \sum_{X_i \in \Il} \parenb{W_i - \Wl} }^2 \sachant \Wl} \\
&= n \phl R_V(n,n\phl,\Wl,\loi(W)) + n \phl \paren{ n \phl - 1} R_C(n,n\phl,\Wl,\loi(W))
\end{align*}
so that if $n \phl \geq 2$,
\begin{equation} \label{RP.eq.calc.ph.int2} R_C(n,n\phl,\Wl,W) = \frac{-1}{n\phl -1} R_V(n,n\phl,\Wl,\loi(W)) \end{equation}
and $R_V(n,1,\Wl,\loi(W)) = 0$.
Then, \eqref{RP.eq.calc.ph.int1} and \eqref{RP.eq.calc.ph.int2} imply
\begin{align} \label{RP.eq.calc.ph.int3}
\Es \croch{\phl \carre{\bethlW - \bethl} \sachant \Wl}
&= \frac{R_V(n,n\phl,\Wl,\loi(W)) }{\Wl n^2 \phl} \un_{n \phl \geq 2} \\ \notag
&\quad \times \croch{ \frac{n\phl}{n\phl - 1} S_{\lambda,2} - \frac{1}{n\phl -
1} S_{\lambda,1}^2 }
\end{align}
Finally, \eqref{VFCV.eq.pen.Wech} follows from the combination of \eqref{RP.def.ph1} and \eqref{RP.def.ph2} with \eqref{RP.eq.calc.ph.int3}.
\end{proof}

\subsection{Resampling constants} \label{RP.sec.calc.reech}
Some results relative to the exchangeable weights introduced in Section~\ref{sec.cadreRP.heur} are proved in this subsection.
First, Lemma~\ref{RP.le.calc.R1.R2} below provides explicit formulas for $R_{1,W}(n,\phl)$ and $R_{2,W}(n,\phl)$
which appear in the explicit formula \eqref{VFCV.eq.pen.Wech} for the resampling penalty.
\begin{lemma} \label{RP.le.calc.R1.R2}
Let $n \in \N$ and $\phl \in (0,1]$ such that $n \phl \in \{1, \ldots, n\}$. Then, for every $\mefr \in \N \backslash \{0\}$, $p \in (0;1]$, $\mu >0$ and $q \in\set{1,\ldots,n}$,
\begin{align}
\label{RP.eq.calc.R1.R2.Efr}
R_{1,\mathrm{Efr}(\mefr)} &= \frac{n}{\mefr} \einv{\mathcal{B}(\mefr,\phl)} \paren{1 - \frac{1}{n\phl}} &\quad
R_{2,\mathrm{Efr}(\mefr)} &= \frac{n}{\mefr} \paren{1 - \frac{1}{n\phl}}
\\
\label{RP.eq.calc.R1.R2.Rad}
R_{1,\mathrm{Rad}(p)} &= \frac{1}{p}\einv{\mathcal{B}(n\phl,p)} -1 &\quad R_{2,\mathrm{Rad}(p)} &= \frac{1}{p} -1
\\
\label{RP.eq.calc.R1.R2.Poi}
R_{1,\mathrm{Poi}(\mu)} &= \frac{1}{\mu} \einv{\mathcal{P}(n\phl\mu)} \paren{1 - \frac{1}{n \phl} } &\quad
R_{2,\mathrm{Poi}(\mu)} &= \frac{1}{\mu} \paren{1 - \frac{1}{n \phl} }
\\
\label{RP.eq.calc.R1.R2.Rho}
R_{1,\mathrm{Rho}(q)} &= \frac{n}{q}\einv{\mathcal{H}(n,n\phl,q)} -1 &\quad R_{2,\mathrm{Rho}(q)} &= \frac{n}{q} -1
\\ \notag
R_{1,\mathrm{Loo}} &= \frac{n \phl}{n (n\phl -1)} \un_{n\phl \geq 2} &\quad R_{2,\mathrm{Loo}} &= \frac{1}{n-1}
\end{align}
where $\mathcal{B}$, $\mathcal{P}$ and $\mathcal{H}$ denote respectively the Binomial, Poisson and Hypergeometric distributions and $\einv{\mu}=\E\croch{Z} \E\croch{Z^{-1} \sachant Z>0}$ with $Z \sim \mu$.
\end{lemma}

\begin{proof}[Proof of Lemma \ref{RP.le.calc.R1.R2}]
Since $W$ is independent of the data, the observations with $X_i \in \Il$ can be assumed to be the $n \phl$ first ones: $(X_1,Y_1), \ldots, (X_{n\phl},Y_{n\phl})$. The random vector $(W_i)_{1 \leq i \leq n\phl}$ is then exchangeable (since $W$ is exchangeable). Hence, by definition of $\Wl=(n\phl)^{-1} \sum_{i=1}^{n\phl} W_i$,
\begin{equation} \label{RP.eq.E[Wi|Wl]}
\forall i \in \{1, \ldots, n \phl \}, \quad \Es \croch{ W_i \sachant \Wl } = \Wl .
\end{equation}

Then, the quantity \[ R_V(n,n\phl,\Wl,\loi(W)) = R_V(\Wl) = \E\croch{\parenb{W_i - \Wl}^2 \sachant \Wl} \] appearing both in $R_{1,W}$ and $R_{2,W}$ is the variance of the weight $W_i$ conditionally on $\Wl$.

\paragraph{Exchangeable subsampling weights}
A {\em subsampling weight} is defined as any resampling weight $W$ such that $W_i \in \{0, \kappa\}$ a.s. for every $i$. Such weights can be written $W_i = \kappa \un_{i \in I}$ for some random $I \subset \{1, \ldots, n\}$. Rad and Rho are the two main examples of such weights and they are both exchangeable.
This kind of weights are called ``bootstrap without replacement weights'' in \cite[Example~3.6.14]{vdV_Wel:1996}.
First, when $W$ is an exchangeable subsampling weight, \eqref{RP.eq.E[Wi|Wl]} implies
\[ \Wl = \Es \croch{ W_i \sachant \Wl} = \kappa \Prob\paren{W_i = \kappa \sachant \Wl}\]
so that
\[ \loi \paren{ W_i \sachant \Wl } = \kappa \mathcal{B}(\kappa^{-1} \Wl) \qquad \mbox{and} \qquad R_V(\Wl) = \Wl (\kappa - \Wl) .\]

Then, this result is applied to Rad with $\kappa = p^{-1}$ and
$\loi(\Wl) = (n\phl p)^{-1} \times \mathcal{B}(n\phl,p)$ which proves
\eqref{RP.eq.calc.R1.R2.Rad}. In the Rho case, $\kappa = (n/q)$ and
$\loi(\Wl) = (q \phl)^{-1}\* \mathcal{H} \paren{n, n\phl,q}$ so that
\eqref{RP.eq.calc.R1.R2.Rho} follows. The Loo is a particular case of
Rho (with $q=n-1$) and $\einv{\mathcal{H}(n,n\phl,n-1)}$ can be
computed with \eqref{RP.eq.einv.Loo} in
Lemma~\ref{RP.le.einv.hypergeom}.

\paragraph{Efron}
Efron weights can also be written
\begin{equation} \label{RP.eq.Efron.Uj} W_i = \frac{n}{\mefr} \card \set{ 1 \leq j \leq \mefr \telque U_j = i } \end{equation}
 with $(U_j)_{1 \leq j \leq \mefr}$ a sequence of independent random variables with uniform distribution over $\{1, \ldots, n\}$. Therefore,
\[ \loi(\Wl) = (\mefr \phl)^{-1} \mathcal{B}(\mefr,\phl) \quad \mbox{and} \quad \loi \paren{ W_i \sachant \Wl } = \frac{n}{\mefr} \mathcal{B} \paren{ \mefr \phl \Wl, \frac{1}{n\phl} } \]
so that \[ R_V(\Wl) = \frac{n}{\mefr} \Wl \paren{ 1 - \frac{1}{n\phl}} \] and \eqref{RP.eq.calc.R1.R2.Efr} follows.

\paragraph{Poisson}
One can check that the weights defined by \eqref{RP.eq.Efron.Uj} with $\mefr=N_n \sim \mathcal{P}(\mu n)$ independent of the $(U_j)_{j \geq 1}$, are actually Poisson ($\mu$) weights; this is the classical poissonization trick \cite[Chapter~3.5]{vdV_Wel:1996}.
Moreover, conditionally on $\Wl$ and $N_n = \mefr$, the same reasoning as for Efron($\mefr$) (with a multiplicative constant $\mu^{-1}$ instead of $n/\mefr$) leads to
\eqref{RP.eq.calc.R1.R2.Poi}.
\end{proof}

\begin{proof}[Proof of Proposition~\ref{RP.pro.comp.Epen.Ep2}]
From \eqref{VFCV.eq.pen.Wech}, \eqref{RP.eq.comp.Epen.Ep2} holds with
\[ \delta_{n,\phl}^{(\mathrm{penW})} = \CWinf \paren{ R_{1,W}(n,\phl) + R_{2,W}(n,\phl) } - 2 . \]
Combining Lemma~\ref{RP.le.calc.R1.R2} with Lemma~\ref{RP.le.einv.binom} (for Efr and Rad), Lemma~\ref{RP.le.einv.hypergeom} (for Rho and Loo) and Lemma~\ref{RP.le.einv.Poi} (for Poi), the following non-asymptotic bounds are obtained:
\begin{enumerate}
\item Efron ($\mnefr$): let $\kappa_1 = 5.1$ and $\kappa_2 = 3.2$, then
\begin{equation} \label{RP.eq.delta.penEfr}
\minipar{ \kappa_2 - 1} { \frac{\kappa_1}{ \paren{B n \phl}^{1/4} }} \geq \delta_{n,\phl}^{(\mathrm{penEfr}(\mnefr))} \geq \frac{-2}{n\phl} - e^{ - B n \phl} .\end{equation}
\item Rademacher ($p$):
\begin{gather} \label{RP.eq.delta.penRad}
\frac{2}{1-p} \croch{ \minipar{ \kappa_2 - 1} { \frac{\kappa_1}{ \paren{n p \phl}^{1/4} }} } \geq \delta_{n,\phl}^{(\mathrm{penRad}(p))} \geq \frac{ - 2 e^{ - p n \phl}} {1-p} \\
\minipar{ 1 + 3 \times 10^{-4} } { \frac{\kappa_1 \times 2^{1/4}}{ \paren{n \phl}^{1/4} }} \geq \delta_{n,\phl}^{(\mathrm{penRad}(1/2))} \geq - \un_{n \phl \leq 2} \label{RP.eq.delta.penRad2}
.\end{gather}
\item Poisson ($\mu$):
\begin{equation} \label{RP.eq.delta.penPoi}
\mini{ 1 } { \frac{2 \paren{1+e^{-3}} } { \paren{ \mu n \phl - 2}_+ }} \geq \delta_{n,\phl}^{(\mathrm{penPoi}(\mu))} \geq \frac{-2}{n\phl} - \paren{ \mini{ e^{ - \mu n \phl}} { \un_{\mu n \phl < 1.61} } } .\end{equation}
\item Random hold-out ($q_n$): on the one hand,
\[ \delta_{n,\phl}^{(\mathrm{penRho}(q_n))} = \frac{n}{n - q} \paren{ \einv{\mathcal{H}(n, n\phl, q_n)} - 1} \geq \frac{e^{-n \phl B_-}}{1 - B_+} , \]
where the lower bounds assume that $0 < B_- \leq q_n n^{-1} \leq B_+ < \infty$. On the other hand, under the same condition
\[ \delta_{n,\phl}^{(\mathrm{penRho}(q_n))} \leq \frac{L}{B_- (1- B_+)} \sqrt{ \frac {\ln(n\phl)}{n\phl} } \] provided that $n \phl \geq L_{B_-, B_+}$. When $q_n = \left\lfloor n/2 \right\rfloor$, this upper bound is combined with \eqref{RP.eq.einv.hypergeom.sup}.
\item Leave-one-out:
\begin{equation} \label{RP.eq.delta.penLoo}
\frac{ \un_{n\phl \geq 2} } {n \phl - 1} \geq \delta_{n,\phl}^{(\mathrm{penLoo})} \geq - \un_{n\phl = 1} .\end{equation}
\end{enumerate}
\end{proof}

\begin{proof}[Proof of Lemma~\ref{RP.le.CWinf.minor}]
Lemma~\ref{RP.le.CWinf.minor} is a byproduct of the proof of Proposition~\ref{RP.pro.comp.Epen.Ep2} (combined with Lemma~\ref{VFCV.le.calc.p2}).
\end{proof}

\vspace*{-2pt}
\subsection{Concentration inequalities} \label{RP.sec.proof.conc}
\vspace*{-3pt}

In this subsection, concentration inequalities are proved for the resampling penalty (Proposition~\ref{RP.pro.conc.penRP}) and for $\delc(m)$ with unbounded data (Lemma~\ref{RP.le.conc.delc.nonborne}).

\vspace*{-2pt}
\subsubsection[Proof of Proposition~3]{Proof of Proposition~\ref{RP.pro.conc.penRP}} \label{RP.sec.proof.pro.conc.penRP}
\vspace*{-3pt}

According to \eqref{VFCV.eq.pen.Wech}, $\pen(m)$ is a U-statistics of order 2 conditionally on $(\un_{X_i \in \Il})_{(i,\lambda)}$. Then, \cite[Lemma~5]{Arl:2008a} with
\begin{align*} a_{\lambda} &= \frac{R_{1,W}(n,\phl) + R_{2,W}(n,\phl)} {n (n\phl - 1)} \qquad b_{\lambda} = \frac{- \paren{ R_{1,W}(n,\phl) + R_{2,W}(n,\phl) }} {n^2 \phl (n\phl - 1)} ,
\end{align*}
implies that for every $q \geq 2$
\begin{equation*} 
\begin{split}
\norm{\pen(m) - \El [\pen(m)]}^{(\Lambda_m)}_q \leq L_{a_{\ell}, \xi_{\ell}} D_m^{-1/2} A_n^{-1/2} \qquad \qquad \\
\qquad \times \sup_{np \geq A_n} \set{ R_{1,W}(n,p) + R_{2,W}(n,p) } q^{\xi_{\ell} + 1} \E\croch{p_2(m)} .
\end{split}
\end{equation*}
Conditional concentration inequalities follow from the classical link between moments and concentration \cite[Lemma~8.10]{Arl:2007:phd}, with a probability bound $1 - n^{-\gamma}$. Since $1 - n^{-\gamma}$ is deterministic, this implies unconditional concentration inequalities.

The second statement follows from the proof of Proposition~\ref{RP.pro.comp.Epen.Ep2} where non-asymptotic upper bounds on
\[ 2 + \delta_{n,\phl}^{(\mathrm{penW})} = \CWinf \times \paren{R_{1,W}(n,\phl) + R_{2,W}(n,\phl)}  \]
can be found. \qed

\vspace*{-2pt}
\subsubsection[Proof of Lemma~12]{Proof of Lemma~\ref{RP.le.conc.delc.nonborne}}
\vspace*{-3pt}

From \cite[Lemma~8.18]{Arl:2007:phd} which is stated and proved in \cite{Arl:2008b:app},
\begin{align*} \norm{\delc(m)}_q &\leq \frac{2 \sqrt{\kappa} \sqrt{q}}{\sqrt{n}} \norm{F_m - \E[F_m]}_q \\
\text{with} \quad F_m &\egaldef (Y - \bayes_m(X))^2 - (Y- \bayes(X))^2 \\
&= (\bayes_m(X) - \bayes(X))^2 - 2 \epsilon \sigma(X) (\bayes_m(X) - \bayes(X)) .
\end{align*}
Note that $\epsilon \sigma(X) (\bayes_m(X) - \bayes(X))$ is centered conditionally on $X \in \Il$ for every $\lamm$. Hence,
\begin{gather}
\norm{\delc(m)}_q \leq \frac{2 \sqrt{\kappa} \sqrt{q}} {\sqrt{n}} \paren{ \norm{\bayes - \bayes_m}^2_{\infty} + 2 \sigmax \norm{\bayes - \bayes_m}_{\infty} \norm{\epsilon}_q } \label{RP.eq.mom.delc.nonborne}
.
\end{gather}

Using now assumptions \hypAgeps\ and \hypAdel, for every $q \geq 2$,
\begin{align*} \norm{\delc(m)}_q &\leq
2 \sqrt{\kappa} \sqrt{q} \left( (\cdelmglo)^2 \perte{\bayes_m} + 2 \cdelmglo \sqrt{\perte{\bayes_m}} P^{g \epsilon}(q) \sigmax \right) \frac{1}{\sqrt{n}} \\
&\leq L_{\cdelmglo} \sqrt{q} D_m^{-1/2} \perte{\bayes_m} + L_{ a_{g \epsilon}, \xi_{g \epsilon}, c_{\Delta,m}^{g} } q^{\xi_{g \epsilon} + 1/2 } \frac{\sigmax^2 \sqrt{D_m}} {n} .
\end{align*}
Taking $\theta = D_m^{-1/2}$, \eqref{RP.eq.conc.delc.nonborne} follows from the classical link between moments and concentration inequalities \cite[Lemma~8.10]{Arl:2007:phd}.
For the second statement, start back from \eqref{RP.eq.mom.delc.nonborne} and use that $\norm{\bayes - \bayes_m}_{\infty} \leq 2A$.
\qed


\subsection{Expectations of inverses} \label{RP.sec.app.einv.preuves}
This subsection is devoted to the proofs of the lemmas of Section~\ref{RP.sec.app.einv}. Note that \cite[Section~2 of the Technical appendix]{Arl:2008a} explains how to generalize \eqref{RP.eq.einv-binom.inf-sup} to a wide class of random variables.
Two useful results can be found in \cite[Technical appendix]{Arl:2008a}: first, the general lower bound
\begin{equation} \label{RP.eq.einv.minor.gal} \einv{Z} \geq \Prob(Z>0) , \end{equation}
comes from Jensen inequality. Second, defining
\begin{equation} \label{RP.def.einfz}
\einvz{\loi(Z)} \egaldef \E \croch{Z} \E \croch{ Z^{-1} \un_{Z>0}} = \einv{Z} \Prob(Z>0) ,
\end{equation}
the following upper bound holds as soon as $\Prob(c_Z > Z > 0)=0$:
\begin{align}
\forall \alpha>0, \quad
\einvz{Z} &= \E \croch{Z^{-1} \un_{\alpha \E[Z] > Z>0}} \E[Z] + \E \croch{Z^{-1} \un_{Z \geq \alpha \E[Z]}} \E[Z] \notag \\
&\leq \Prob \paren{\alpha \E[Z] > Z>0} \E[Z] c_Z^{-1} + \alpha^{-1} .
 \label{RP.eq.einv.major.gal}
\end{align}
%

\subsubsection[Binomial case (proof of (19) in Lemma~4)]{Binomial case (proof of \eqref{aRP.eq.einv.sym} in Lemma~\ref{RP.le.einv.binom})} \label{aRP.sec.einv.sym}
When $n \geq 9$, the upper bound follows from \eqref{RP.def.einfz} together with Lemma~4.1 of \cite{Gyo_etal:2002} (showing that $\einvz{\mathcal{B}(n,p)}\leq 2n/(n+1)$). When $n \leq 8$, $\einv{\mathcal{B}(n,1/2)}\leq 1.21$ (see for instance \cite[Section~8.7]{Arl:2007:phd}).
For the lower bound, the crucial point is that $Z \sim \mathcal{B}\paren {n, \frac{1}{2} } $ is nonnegative and symmetric, that is, $\loi(Z) = \loi(n-Z)$.
Using only this property and defining $p_0 = \Prob(Z = 0) = \Prob(Z = n) = 2^{-n}$, we
have\looseness=1
\begin{align} \notag
 \einv{Z} &=
\frac{\P \paren{ Z = n \sachant Z>0 } }{2} + \E \croch{ \frac{1}{Z}
\sachant 0 < Z < 2 } \frac{n}{2} \frac { \P(0<Z<n) } {\P( Z >0)} \\
\notag
&= \frac{p_0}{2(1-p_0)} + \frac{1-2p_0}{1-p_0} \frac{n}{2} \E \croch{
\frac{1}{2} \left( \frac{1}{Z} + \frac{1}{n-Z} \right)
\sachant 0 < Z < n } \\
\label{aRP.eq.einv.sym.gal}
&= \frac{p_0}{2(1-p_0)} + \frac{1-2p_0}{1-p_0} \paren{ 1 + \frac{n}{2} \E
\croch{ \frac{\carre{Z - \frac{n}{2}}}{Z (n-Z)} \sachant 0 < Z < n
} } .
\end{align}
Since $Z$ is binomial with parameters $(n,1/2)$
\begin{align*} \frac{n(1-2p_0)}{2} \E \croch{ \frac{\carre{Z - \frac{n}{2}}}{Z (n-Z)}
\sachant 0 < Z < n } &\geq \Prob \left( Z = 1 \text{ or } Z = n-1 \right) \frac {(n-2)^2} {4 (n-1)}
\end{align*}
if $n \geq 3$.
Putting this into \eqref{aRP.eq.einv.sym.gal},
 we obtain:
\begin{align*}
\einv{\mathcal{B}\paren{n,\frac{1}{2}}} &\geq \frac {1}
{1- 2^{-n}} \left( 2^{-n-1} + 1 - 2^{1-n} + \frac{n
(n-2)^2}{2^{n+1} (n-1) } \right) \geq 1 . \qed
\end{align*}

\subsubsection[Hypergeometric case (proof of Lemma~5)]{Hypergeometric case (proof of Lemma~\ref{RP.le.einv.hypergeom})} \label{aRP.sec.einv.hypergeom}
Let $Z \sim \mathcal{H}(n,r,q)$. It has an expectation $\E \left[ Z \right] = (qr)/n$.
\paragraph{General lower bound}
It follows from \eqref{RP.eq.einv.minor.gal}, \[ \Prob \left( Z = 0 \right) \leq \left( 1 - \frac{r}{n} \right)^{q} \leq \exp \paren{-\frac{qr}{n}} \]
and the fact that if $r \geq n-q+1$, $\P(Z>0)=1$.
\paragraph{A general upper bound}
According to \eqref{RP.def.einfz} and the lower bound for $\Prob(Z>0)$ above,
an upper bound on $\einv{\mathcal{H}(n,r,q)}$ can be derived from an upper bound on $\einvz{\mathcal{H}(n,r,q)}$.
Recall the following concentration result by Hush and Scovel \cite{2005:Hus_Sco}: for every $x \geq 2$,
\begin{equation*} \begin{split}
&\Prob \left( \E(Z) - Z > x \right) \\
&< \exp \left( -2 (x-1)^2 \left[
\maxipar{\frac{1}{r+1} + \frac{1}{n-r+1}} {\frac{1}{q+1} +
\frac{1}{n-q+1}}\right] \right) .
\end{split} 
\end{equation*}
Combined with the above concentration inequality,
\eqref{RP.eq.einv.major.gal} with $c_Z = 1$, $\E[Z]=qr n^{-1}$
and $\alpha = 1 - \frac{n \beta}{q}$ for any $\frac{q}{n} > \beta \geq \frac{2}{r} $
yields
\begin{align*}
\einvz{\mathcal{H}(n,r,q)}
&\leq \frac{qr}{n} \exp \left[ -
\frac{2 (\beta r - 1)^2 }{r+1} \right] + \frac{1}{1 - \frac{n
\beta}{q}} .
\end{align*}
Therefore,
\begin{equation}
\einv{\mathcal{H}(n,r,q)}
\leq
\frac{ \inf_{\frac{q}{n} > \beta \geq \frac{2}{r}} \Bigl\{
\frac{qr}{n} \exp \bigl[ - \frac{2 (\beta r - 1)^2 }{r+1} \bigr]
+ \frac{1}{1 - \frac{n \beta}{q}} \Bigr\} } { 1 - \exp \left( - \frac{qr}{n} \right)}
\label{RP.eq.einv.hypergeom}
\end{equation}
holds for every $n \geq r,q \geq 1$.
\paragraph{End of the proof of \eqref{RP.eq:einv_hypergeom:maj_non-asympt}}
With the additional conditions on $n$, $r$ and $q$, $\beta$ can be taken equal to $\frac{1+\sqrt{\frac{3}{4}
\ln(r)(r+1)}}{r}$ in \eqref{RP.eq.einv.hypergeom} so that
\begin{align*}
\einvz{\mathcal{H}(n,r,q)} &\leq \frac{1}{2\sqrt{r}} + \frac {1}
{1 - \frac{n}{q} \biggl( \frac{1+\sqrt{\frac{3}{4} \ln(r)(r+1)}}{r} \biggr)
} \leq 1 + \frac{n}{q} K(\epsilon) \sqrt{\frac{\ln(r)}{r}} \\
\mbox{with} \quad
K(\epsilon) &= \frac{1}{2 \sqrt{\ln(2)}} + \frac{1}{\epsilon^2} \left( \sqrt{\frac{\ln(3)}{3}} + \frac{3}{4}\right) .\end{align*}
Using \eqref{RP.def.einfz} and the upper bound on $\Prob \left( Z = 0 \right)$, \eqref{RP.eq:einv_hypergeom:maj_non-asympt} follows since $r \geq 2$ and
\begin{equation*}
\kappa_3(\epsilon) = 0.9 + 1.4 \times \epsilon^{-2} \geq 1.02 \times K(\epsilon) + 0.03 . \end{equation*}
%
\paragraph{``Rho'' case} Assume now that
$q=\lfloor \frac{n}{2} \rfloor$ so that $\frac{n}{q} =
2 + \frac{1}{\lfloor \frac{n}{2} \rfloor} \leq 3$ and tends to
2 when $n$ tends to infinity.

For $r \geq 6$, $\beta = \frac{2}{r}$ in
\eqref{RP.eq.einv.hypergeom} yields
\[ \einv{\mathcal{H}(n,6,q)} \leq 9.68 \qquad \einv{\mathcal{H}(n,7,q)} \leq
7.61 \qquad \einv{\mathcal{H}(n,8,q)} \leq 7.46 \qquad
\einv{\mathcal{H}(n,9,q)} \leq 7.32 \]

For $r \geq 10$, $\beta = \frac{1}{4} + \frac{1}{r}$ in
\eqref{RP.eq.einv.hypergeom} yields
\[ \sup_{r \geq 10} \einv{\mathcal{H}(n,r,q)} \leq 7.49
\qquad \sup_{r \geq 26} \einv{\mathcal{H}(n,r,q)} \leq 3 .\]
\paragraph{Small values of $r$} must be treated appart. For
$r=1$, it is easy to compute $\einv{\mathcal{H}(n,1,q)} = q n^{-1} \leq 1$.
When $n=r$, we have $\einv{\mathcal{H}(n,n,q)} = 1$. Otherwise, using the fact that
for every $n \geq r+1$, $\frac{n!}{(n-r)!} \geq
\frac{(r+1)!}{(r+1)^r} n^r$,
\[ \einvz{\mathcal{H}(n,r,q)} \leq \frac{r}{R} \frac{(r+1)^r}{(r+1)! R^r}
\left( \sum_{k=1}^r \binom{r}{k} \frac{(R-1)^{r-k}}{k} \right) \]
with $R = \frac{n}{q} \in [1; +\infty)$.
For $r=2$, this upper bound is lower than $1.6$. If $\frac{n}{q} \leq 3$ (which holds in the ``Rho'' case),
\[ \einv{\mathcal{H}(n,3,q)} \leq 4.67 \qquad \einv{\mathcal{H}(n,4,q)} \leq
8.15 \qquad \einv{\mathcal{H}(n,5,q)} \leq 14.29 . \]
\paragraph{``Loo'' case}
Assume now $q=n-1$. On the one hand, if $r=1$, the conditioning makes $Z$ deterministic and equal to 1 so that \[ \einv{\mathcal{H}(n,1,n-1)} = \E[Z] = 1 - \frac{1}{n} . \]

On the other hand, if $r \geq 2$, $Z>0$ holds a.s. since it only take two values:
\[ \Prob\paren{ Z = r-1 } = \frac{r}{n} \quad \mbox{and} \quad \Prob\paren{ Z = r } = \frac{n-r}{n} . \]
Hence,
\[ \einv{\mathcal{H}(n,r,n-1)} = \frac{(n-1)r}{n} \paren{ \frac{r}{(r-1)n} + \frac{n-r}{n r} } = 1 + \frac{1}{n} \paren{ \frac{(n-1)r}{n(r-1)} - 1} . \] The lower bound is straightforward since $n \geq r$.
\paragraph{``Lpo'' case}
As noticed in Lemma~\ref{RP.le.calc.R1.R2},
\[ \forall r \geq p+1, \qquad \einv{\mathcal{H}(n,r,n-p)} \geq 1 . \]

Moreover, when $r \geq p+1$ the support of $\mathcal{H}(n,r,n-p)$ is $\{r-p, \ldots, r\}$ and
\begin{align*}
\einv{\mathcal{H}(n,r,n-p)} &= \frac{(n-p)r}{n} \sum_{j=r-p}^r \frac{\binom{r}{j} \binom{n-r}{n-p-j}}
{j \binom{n}{n-p}} \\
&= \frac{(n-p)r}{n} \sum_{k=\maxi{(p+r-n)}{0}}^{p} \frac{\binom{r}{k} \binom{n-r}{p-k}} {(r-k) \binom{n}{p}} . \end{align*}

More precisely, the $k$-th term of the sum is equal to
\begin{align*}
\frac{(n-p)r}{n} \frac{\binom{r}{k} \binom{n-r}{p-k}} {(r-k) \binom{n}{p}} 
\leq \paren{\frac{r}{n}}^k \paren{1-\frac{r}{n}}^{p-k} \binom{p}{k} \frac{r}{r-p} \frac{n^p} {n \cdots (n-p+1)} ,
\end{align*}
so that
\begin{align*}
\einv{\mathcal{H}(n,r,n-p)} &\leq
\frac{r n^p} {(r-p) n \cdots (n-p+1)} .
\end{align*}
The result follows. \qed

\begin{remark}[Asymptotics]
If for some $\alpha>0$, $q_k r_k^{1/2-\alpha}
n_k^{-1} \xrightarrow[k\rightarrow + \infty]{} +\infty$ and $n_k \geq r_k \rightarrow + \infty$,
then $\einv{\mathcal{H}(n_k,r_k,q_k)} \rightarrow 1$ when $k \rightarrow \infty$. The upper bound
is obtained by taking
\[\beta = \frac{1+\sqrt{(r_k+1) \ln \bigl( \frac{q_k r_k}{n_k} \bigr)
}} {r_k}\] in \eqref{RP.eq.einv.hypergeom}, which is possible for $k$
sufficiently large. The lower bound is straightforward.
\end{remark}

\subsubsection[Poisson case (proof of Lemma~6)]{Poisson case (proof of Lemma~\ref{RP.le.einv.Poi})} \label{aRP.sec.einv.poisson}
Let $Z \sim \mathcal{P}(\mu)$ and define $g: [0;\infty) \mapsto \R$ by $g(0)=0$ and for every $\mu>0$
 \[ g(\mu) \egaldef \einv{\mathcal{P}(\mu)} = \mu \E \croch{ Z^{-1} \sachant Z >0 }= \frac{\mu
e^{-\mu}} {1 - e^{-\mu}} \sum_{k=1}^{+\infty}
\frac{\mu^k} {k \times k!} = \frac{\mu}{e^{\mu}-1}
\int_0^{\mu} \frac{e^x - 1}{x} dx .\]
The function $g$ is continuous at 0 and has a first derivative $g^{\prime}(0)=1$. For every $x \geq 0$, define \[ h(x) = \frac
{e^x - 1}{x} \qquad H(x) = \int_0^x h(t) dt \qquad
a(x)= \frac{h^{\prime}(x)}{h(x)} = 1 - \frac{e^x - 1 - x}{x(e^x - 1)} . \] where the last equality holds if $x>0$ and $a(0)=1/2$. Then, $g(u) = H(u)/h(u)$ satisfies the following ordinary differential equation:
\[ g(0)=0 \qquad \forall u \geq 0, \quad g^{\prime}(u) = 1 - a(u) g(u) .\]
Since \[ \forall u \geq 0, \quad \frac{1}{2} \leq a(u) \leq 1 \qquad \mbox{and} \qquad
\lim_{u \rightarrow + \infty} a(u) = 1 , \] $g$ satisfies a differential inequation
\[ 1 - \frac{g}{2} \leq g^{\prime} \leq 1 - g \qquad g(0)=0 .\]
Then, for every $x \geq x_0 \geq 0$,
\begin{equation} \label{RP.eq.einv.Poi.int1} 2 \left[ 1 - e^{2(x_0 - x)} \left(1 -\frac{g(x_0)}{2} \right) \right] \geq g(x) \geq 1 + (g(x_0) - 1) e^{x_0 - x} .\end{equation}

\subparagraph{Lower bound}
The general lower bound \eqref{RP.eq.einv.minor.gal} gives \[ g(\mu) \geq \P(Z > 0) = 1 - e^{-\mu} ,\]
which can be improved. Indeed, if $g(x_0) \geq 1$, \eqref{RP.eq.einv.Poi.int1} shows that $g(x) \geq 1$ for every $x \geq x_0$. Since $g=H/h$ and for every $u \geq 0$,
\[ H(u) \geq u + \frac{u^2}{4} + \frac {u^3}{18} , \mbox{ it follows that} \quad
 g(u) \geq \frac{u \bigl(u + \frac{u^2}{4} + \frac {u^3}{18} \bigr)}{e^u - 1} .\] Then, $g(1.61) \geq 1$, so that $g(x) \geq 1$ for every $x \geq 1.61$.

\subparagraph{Upper bound} Using \eqref{RP.eq.einv.Poi.int1} with $x_0 = 0$
gives \[ \forall x \geq 0, \quad g(x) \leq 2 - 2 e^{-2x} \leq 2.\]
Moreover, for every $\epsilon\in(0;1)$, $1 - \epsilon \leq a(x) \leq 1$ as soon as $x \geq \epsilon^{-1}$. Then, on $[\epsilon^{-1}; \infty)$, $g$ satisfies the differential inequation \[ g^{\prime} \geq 1 - (1-\epsilon) g . \]
Integrating this between $\epsilon^{-1}$ and $2 \epsilon^{-1}$,
\begin{equation*}
g(2 \epsilon^{-1}) \leq \frac{1}{1-\epsilon} \croch{ 1 + \paren{g(\epsilon^{-1}) (1-\epsilon) -1 } \exp \paren{- \epsilon^{-1} (1-\epsilon)^{-1} } } . \end{equation*}
For every $x > 2$, $\epsilon = 2 x^{-1} \in (0;1)$ so that
\begin{align*}
g(x) &\leq 1 + \frac {2 + (x-4) \exp \parenb{ - \frac {x^2} {2(x-2)}} } {x-2} 
\leq 1 + \frac{2(1+e^{-3})}{x-2} .
\end{align*}
The result follows. \qed

\section*{Acknowledgments}
The author would like to thank gratefully Pascal Massart for several fruitful discussions.
The author also acknowledges several suggestions from the anonymous referees that greatly improved the paper.

\bibliographystyle{plain}

\end{document}